\documentclass[english, usenames,dvipsnames,svgnames,table]{article}
\usepackage{preamble}
\usepackage[toc,page]{appendix}

\title{On the Telescopic Picard Group}

\begin{document}

\begin{titlepage}
    \maketitle
    \begin{abstract}
        We prove that for any prime $p$ and height $\chrHeight\ge 1$, the telescopic Picard group $\Pic(\SpTn)$ contains a subgroup of the form $\Zp\times \ZZ / a_p(\extOrder)$, where $a_p = 1$ if $p=2$ and $a_p=2$ if $p$ is odd. 
        Using Kummer theory, we obtain an $(\uFpn \rtimes \ZZ / \chrHeight)$-Galois extension of $\SSTn$, obtaining the first example of a lift of a non-Abelian Galois extension of the $\Kn$-local sphere to the telescopic world, at arbitrary positive height and prime.

        Our proof proceeds by setting up a higher categorical framework for the periodicity theorem, utilizing the symmetries of this framework to construct Picard elements.

        \begin{figure}[h]
            \centering
            \includegraphics[width=0.76\linewidth]{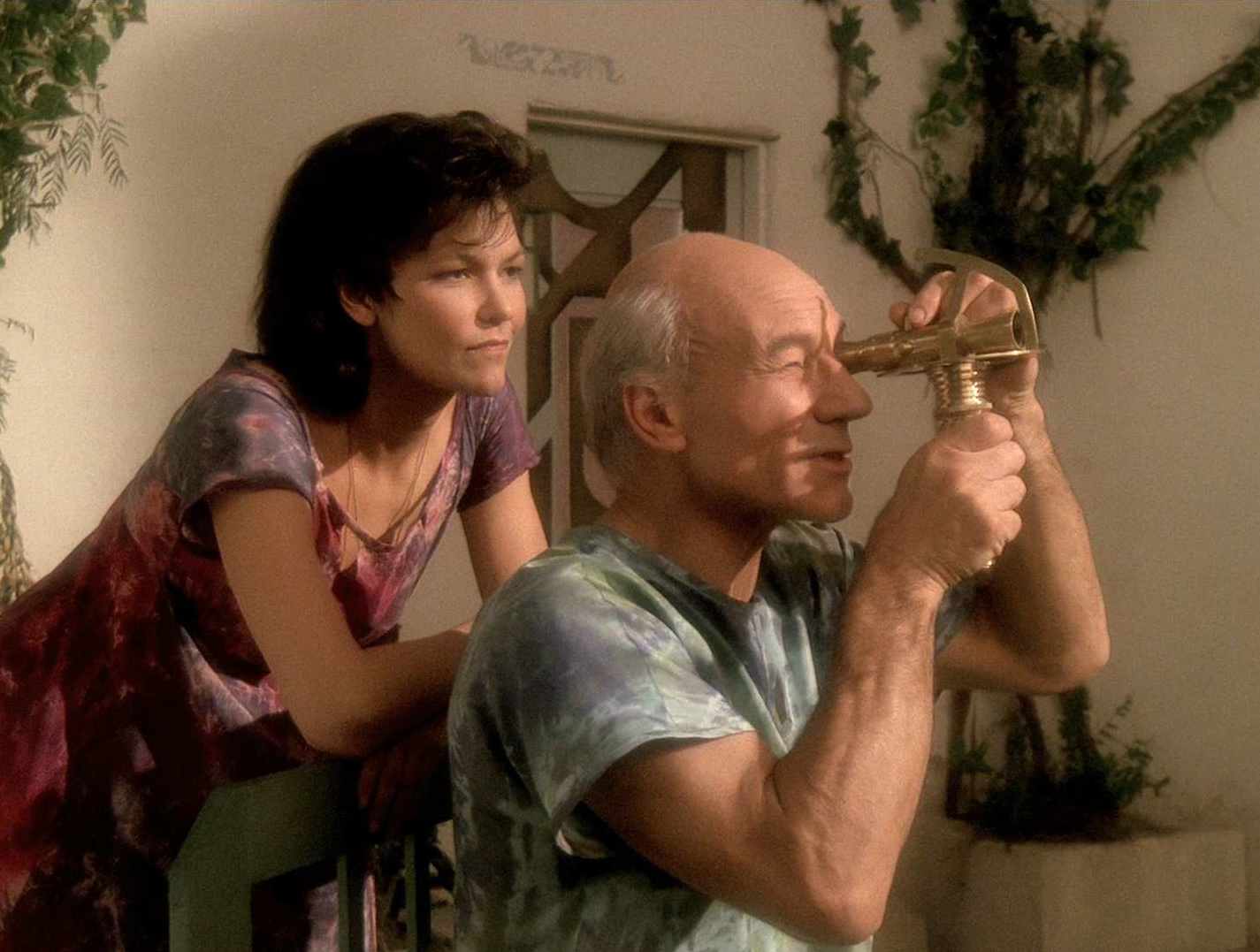}
            \caption{Patrick Stewart and Margot Rose in Star Trek: The Next Generation (1987)}
        \end{figure}
    \end{abstract}

\end{titlepage}

\tableofcontents

\newpage

\section{Introduction}\label{sec:intro}
    Throughout this paper, we use the term category to mean an $(\infty,1)$-category. Similarly we use the term $n$-category to mean an $(\infty,n)$-category.

    \subsection{Background and overview}
        \subsubsection*{Chromatic homotopy theory}
        Chromatic homotopy theory is the study of the category $\Sp_{(p)}$ of $p$-local spectra using the \quotes{height filtration} inherent in the category of compact objects. Two main approaches are taken for the extension of the filtration from compacts to the entire category, resulting in associated graded categories known as \quotes{monochromatic layers}.
        
        The first approach filters the category according to its relationship with formal groups, using the Lubin--Tate spectrum or Morava $E$-theory, which pertains to formal group deformations. This yields the categories $\SpKn$ of $\Kn$-local spectra for the monochromatic layers. This approach is more amenable to computations due to its correspondence with algebraic geometry, Dieudonn\'e modules and Galois descent.
        
        The second strategy, known as the telescopic approach, extends the filtration by colimits from the category of compact spectra to the category of all spectra. The monochromatic layers in this case are the categories $\SpTn$ of $\Tn$-local spectra.
        
        The two approaches are related by the relation $\SpKn\subseteq \SpTn$, which is known to be an equivalence for $n=0,1$ as shown by Miller~\cite{Miller-1981-Moore} and Mahowald~\cite{Mahowald-1981-bo}. However, a recent work of Burkland, Hahn, Levy and Schlank~\cite{Burklund-Hahn-Levy-Schlank-2023-Telescope} proved that the inclusion is strict for $n\ge 2$ and any prime.

        \subsubsection*{The Picard group}
        The Picard group of a symmetric monoidal category $\cC$ is the group of $\otimes$-invertible elements in $\cC$ up to isomorphism. This invariant has been extensively studied, particularly in the context of modules over a ring and, more generally, quasicoherent sheaves over a scheme.
    
        The $\Kn$-local Picard group $\Pic(\SpKn)$ has been the subject of extensive research, initiated by Hopkins--Mahowald--Sadofsky~ \cite{Hopkins-Mahowald-Sadofsky-1994-Picard}, and continuing through numerous contributions, including~\cite{Goerss-Henn-Mahowald-Rezk-2014-Picard, Heard-2015-Picard, Pstragowski-2018-Picard, Beaudry-Bobkova-Goerss-Henn-Pham-Stojanoska-2022-Picard, Mor-2023-Picard, BLLLSZ-2024-Picard} leading to its complete classification for sufficiently large primes and classification up to torsion for all primes in~\cite{Barthel-Schlank-Stapleton-Weinstein-2024-Picard}.
        In particular, $\Pic(\SpKn)$ admits a subgroup of the form
        \begin{equation*}
            \Zn = \Zn[\chrHeight,p] \coloneqq \lim_k \ZZ/(p^k(2p^n-2)),
        \end{equation*}
        topologically generated by $\Sigma \SSKn$. Note that $\Zn \cong \Zp \times \ZZ/(a_p(\extOrder))$ where $a_p = 1$ if $p=2$ and $a_p = 2$ if $p$ is odd.
        Moreover, when $\chrHeight \ge 2$, $\Pic(\SpKn)$ is of rank 2 over $\Zp$.
        
        In contrast, much less is known about $\Pic(\SpTn)$. In~\cite{CSY-cyclotomic}, Carmeli, Schlank and Yanovski show that $\Pic(\SpTn)$ contains a subgroup isomorphic to $\ZZ / (p-1)$. Our first main theorem is the following extension:
        \begin{alphThm}[\cref{thm:ZpxZ/vn->Pic}]\label{alphthm:subgroup-of-pic}
            The group $\Pic(\SpTn)$ admits a subgroup isomorphic to $\Zn$, topologically generated by $\Sigma \SSTn$.
        \end{alphThm}
        In particular, it lifts the corresponding subgroup of $\Pic(\SpKn)$.
    
        \subsubsection*{Galois theory}
        The interplay between formal group theory and $\SpKn$ leads to the construction of the Lubin--Tate commutative ring spectrum $\En = \En(\Fpbar)$ (\cite{Goerss-Hopkins-2004-En, Lurie-2018-Elliptic2}) which is a faithful and relatively computable multiplicative (co)homology theory for $\Kn$-local spectra. 
    
        Conceptually, by the works of~\cite{Devinatz-Hopkins-2004-Morava-stabilizer},~\cite{Rognes-2008-Galois},~\cite{Baker-Richter-2008-Galois-En}, and~\cite{Mathew-2016-Galois}, $\En$ can be viewed as the Galois-closure of the $\Kn$-local sphere $\SSKn$ in $\SpKn$, with Galois group the Morava stabilizer group $\extMorStb = \Aut(\Gamma) \rtimes \Zhat$, where $\Gamma$ is a formal group law of height $\chrHeight$ over $\Fpbar$. 
        Using~\cite[Section 9]{Mathew-2016-Galois}, the $G$-Galois extensions of $\SSKn$ are completely classified by continuous homomorphisms $\extMorStb \to G$. 
        
        This naturally raises the question of lifing Galois extensions from the $\Kn$-local category to the $\Tn$-loca category. Carmeli, Schlank and Yanovski~\cite{CSY-cyclotomic} have shown shown that any Abelian Galois extension of $\SSKn$ lifts. We present the first lift of a non-Abelian Galois extension from $\SSKn$ to $\SSTn$.
        During the writing of this project, Robert Burklund, Dustin Clausen, and Ishan Levy announced a proof that any finite Galois extension of $\SSKn$ lifts uniquely to $\SSTn$, using different methods from ours.

        $\Aut(\Gamma)$ can be explicitly described as the group of units $\unitsOrd$ of the order
        \begin{equation*}
            \MorOrd = \Witt(\Fpn) \left\langle S \middle\vert S^{\chrHeight} = p, Sw = w^{\varphi} S \ \forall w\in \Wn \right\rangle,
        \end{equation*}
        where $\varphi$ is a lift of the Frobenius to $\Witt(\Fpn)$. The map sending $S$ to 0 provides a map $\pi\colon \unitsOrd \onto \Fpn\units$. Let $\ourExtT$ be the pro-Galois extension classified by the group homomorphism
        \begin{equation*}
            \extMorStb = \unitsOrd \rtimes \Zhat \xonto{(\det \oplus \pi) \rtimes\id} (\Zp\units \oplus_{\Fp\units} \uFpn) \rtimes \Zhat,
        \end{equation*}
        where $\Zhat$ acts on $\Zp\units \oplus_{\Fp\units} \uFpn$ via conjugation by $S$, i.e.\ it acts trivially on $\Zp$ and it acts as the Frobenius on $\uFpn$ (in particular $\chrHeight$ acts trivially). The group $\Zp\units \oplus_{\Fp\units} \uFpn$ appearing here is the abelianization of the special Morava stabilizer group $\MorStb = \Aut(\Gamma)$.
        
        \begin{alphThm}[\cref{thm:Gal-summary}]\label{alphthm:Gal-summary}
            There exists a $((\Zp\units \oplus_{\Fp\units} \uFpn) \rtimes \Zhat)$-pro-Galois extension $\ourExtTf$ of $\SSTn$, lifting $\ourExtT$.
        \end{alphThm}

    \subsection{Methods}
        \subsubsection*{Asymptotically defined endomorphisms of the identity}
        One of the fundamental theorems in chromatic homotopy theory is Hopkins and Smith's periodicity theorem~\cite[Theorem~9]{Hopkins-Smith-nilpotenceII}. This states that any compact spectrum of type $\ge \chrHeight$ admits a $\vn$-self-map, which is asymptotically unique. 

        We set a categorical framework for asymptotically defined maps, and moreover showing that categories of asymptotically defined maps admit an asymptotically defined natural transformation $\id \To \id$.

        Let $D$ be a quotient subgroup of $\Zhat$, and write it as a sequential limit of surjections $D = \lim_{\selfMapDeg_i} \ZZ/\selfMapDeg_i$.

        \begin{definition*}[\cref{dfn:cats-with-asymp-defined-endomorphism}, \cref{rmrk:objects-of-SED}]
            Let $\cC\in\Catperf$. A $D$-asymptotically defined endomorphism of the identity $\alpha$ on $\cC$ is the data of an exhausting filtration of $\cC$
            \begin{equation*}
                \cC_1 \xto{F_1} \cC_2 \xto{F_2} \cdots \to \cC,
            \end{equation*}
            equipped with natural transformations $\alpha_i \colon \Sigma^{\selfMapDeg_i} \To \id_{\cC_i}$ and isomorphisms $F_i\alpha_i^{\selfMapDeg_{i+1}/\selfMapDeg_i} \simeq \alpha_{i+1}F_i$ in $\Nat(\Sigma^{\selfMapDeg_{i+1}}F_i, F_i)$.
    
            We say that $\alpha$ is a $D$-asymptotically defined \textbf{isomorphism} of the identity if $\alpha_i$ is an isomorphism for all $i$.
        \end{definition*}
        
        Categories with asymptotically defined endomorphisms and isomorphisms of the identity assemble to 2-categories $\Catperf\SED$ and $\Catperf\SEDiso$. 
        There is an underlying category functor $U \colon \Catperf\SED \to \Catperf$ sending $(\cC_1\to \cC_2 \to \cdots)$ to the colimit $\colim_i \cC_i$.

        We show (\cref{prop:SMd-adjoint-U}) that $U$ admits a right adjoint ${(-)}\SMD$, taking $\cC$ to the system with $\cC_i = \cC\SMd[\selfMapDeg_i]$ --- the category all self maps in $\cC$  of degree $\selfMapDeg_i$. 
        The inclusion $\Catperf\SEDiso \into \Catperf\SED$ admits a left adjoint $L$ (\cref{cor:universal-prop-of-inverting}), inverting $\alpha_i$ at each degree.

        We summarize these adjunctions in the following diagram:
        \begin{equation*}
            \begin{tikzcd}
        	{\Catperf\SEDiso} && {\Catperf\SED} && \Catperf
        	\arrow[""{name=0, anchor=center, inner sep=0}, hook, from=1-1, to=1-3]
        	\arrow[""{name=1, anchor=center, inner sep=0}, "L"', curve={height=30pt}, from=1-3, to=1-1]
        	\arrow[""{name=2, anchor=center, inner sep=0}, "U", from=1-3, to=1-5]
        	\arrow[""{name=3, anchor=center, inner sep=0}, "{{(-)}\SMD}", curve={height=-30pt}, from=1-5, to=1-3]
        	\arrow["{\rotatebox[origin=c]{180}{$\top$}}"{description}, draw=none, from=1, to=0]
        	\arrow["{\rotatebox[origin=c]{180}{$\top$}}"{description}, draw=none, from=3, to=2]
            \end{tikzcd}.
        \end{equation*}
        Going back to our case of interest, we let $D = \Zn = \Zp \times \ZZ/ (a_p(\extOrder))$. We then construct a category with an asymptotically defined endomorphism $\vnSp\subseteq \Sp\SMvn$, by letting $\vnSpk$ be the full subcategory of $\Sp\SMd[\vnSelfDeg]$ consisting of $\vn$-self-maps.
        The periodicity theorem can then be reformulated as follows:
        \begin{alphThm}[\cref{thm:periodicity-rephrased}]\label{alphthm:periodicity}
            The category with asymptotically defined endomorphism of the identity $(\vnSp, \vn)\in\Catperf\SEvn$ satisfies:
            \begin{enumerate}
                \item The underlying category of $\vnSp$ is $\FSpc$;
                \item The underlying category of $L(\vnSp)$ is $\TnComp$.
            \end{enumerate}
        \end{alphThm}

        \begin{remark*}
            This result is similar to the the construction of the Bousfield--Kuhn functor in the unstable setting, as done in \cite{Lurie-Thursday-BK}.
        \end{remark*}

        \subsubsection*{The telescopic Picard group}
        When $\cC$ is presentable, sending $X$ to the functor $X\otimes -$ defines an isomorphism from $\Pic(\cC)$ to the group of $\cC$-linear automorphisms of $\cC$.
        Specializing to $\cC=\SpTn$, we note that any automorphism of $\SpTn$ is $\SpTn$-linear. As $\SpTn$ is compactly generated we deduce that \mbox{$\Pic(\cC) = \pi_0\Aut(\TnComp)$}.
        Therefore we can construct Picard elements by constructing automorphisms of $\TnComp$.
        
        One advantage of our categorified setting, is that it gives rise to $\EE_1$ group actions on categories admitting an asymptotically defined natural isomorphism. More precisely, categories with an automorphism $\Sigma^{\selfMapDeg} \isoto \id$ admit a $\Omega S^1/\selfMapDeg$-action given by suspension. Heading towards the colimit of such categories, we prove that these actions are compatible when varying $d$. We define the space
        \begin{equation*}
            S^2_D \coloneqq \lim_i \Sigma (S^1/\selfMapDeg_i)
        \end{equation*}
        and compute its second homotopy group (\cref{lem:pi2-of-SD})
        \begin{equation*}
            \pi_2(S^2_D) \cong D.
        \end{equation*}
        
        \begin{proposition*}[\cref{cor:group-action-on-SEDiso}]
             The $\EE_1$-group $\Omega^2 S^2_D$ acts on any $\cC$ in $\Catperf\SED$. The topological generator $1 \in D \cong \pi_0\Omega^2 S^2_D$ acts by suspension. 
        \end{proposition*}

        In particular, $\vnSp[\vn^{-1}]$ has a $\Omega^2 S^2_{\Zn}$-action, and therefore so does its underlying category $\TnComp$. We get a group homomorphism on homotopy groups
        \begin{equation*}
            \Zn = \pi_0(\Omega^2 S^2_{\Zn}) \to \Pic(\SpTn).
        \end{equation*}

        To prove \cref{alphthm:subgroup-of-pic} it is left to show that it is injective. Recall that $\Zn = \Zp \times \ZZ/ (a_p(\extOrder))$. We prove it separately for each component:
        \begin{itemize}
            \item For $\Zp$, we tensor with $\En$ remembering only the action of $\Zp\subseteq \Zp\units \subseteq \extMorStb$. It is simple to show that the composition
            \begin{equation*}
                \Zp \to \Pic(\SpTn) \xto{\En\otimes -} \Pic(\Mod_{\En}(\SpTn^{\B\Zp}))
            \end{equation*}
            is injective.
            \item For $\ZZ/ (a_p(\extOrder))$, we compute with $\Kn$-homology, landing in invertible $\Kn_*$-graded modules, which are exactly $\Kn_*, \Kn_{*-1} ,\dots, \Kn_{*-(2(\extOrder)-1)}$. The composition sends $d$ to $\Kn_{*-d}$ if $p$ is odd and to $\Kn_{*-2d}$ is $p=2$, showing that it is injective in both cases.
        \end{itemize}

        \subsubsection*{Kummer theory}
        To construct Galois extensions, we shall make use of the $\infty$-categorical Kummer theory, developed by Carmeli-Schlank-Yanovski, which we now briefly recall.

        Any Picard element $X$ is in particular dualizable and therefore has a dimension $\dim(X) \in \pi_0\ounit\units$. By~\cite[Corollary3.21]{CSY-cyclotomic}, ${\dim(X)}^2 = 1$.
        \begin{definition*}[{\cite[Definition3.22]{CSY-cyclotomic}}]
            Define the even Picard group as the subgroup of \emph{even Picard elements}:
            \begin{equation*}
                \evPic(\cC) \coloneqq \{X \in \Pic(\cC) \mid \dim(X) = 1\}.
            \end{equation*}
        \end{definition*}
        
        \begin{theorem*}[{Kummer theory~\cite[Proposition~3.23]{CSY-cyclotomic}}]
            Let $\cC$ be an additive presentably symmetric monoidal category with a choice of a primitive $m$-th root of unity (\cite[Definition~3.3]{CSY-cyclotomic}). Then there is a split short exact sequence of Abelian groups
            \begin{equation*}
                0 
                \to (\pi_0\ounit\units) / {(\pi_0 \ounit\units)}^m 
                \to \pi_0 \hGal{\ZZ / m}(\cC)
                \to \evPic(\cC)[m]
                \to 0.
            \end{equation*}
        \end{theorem*}
        This enables us to lift torsion even Picard elements to Galois extensions assuming the existence of a root of unity.
        First, consider the $\Tn$-localization of the spherical Witt vector of $\Fpn$, denoted $\SWnf \coloneqq \LTn \sphereWitt(\Fpn)$, which is a faithful $\ZZ/n$-Galois extension of $\SSTn$, admitting a $(p^n-1)$-st root of unity. 
        \cref{alphthm:subgroup-of-pic} together with Kummer theory thus give an $\uFpn$-Galois extension of $\SWnf$, which corresponds to an $(\uFpn \rtimes \ZZ/\chrHeight)$-Galois extension of $\SSTn$.

        Similarly, adding higher roots of unity to our Galois extension as in~\cite{CSY-cyclotomic} (equivalently, replacing $\SWnf$ with $\LTn(\sphereWitt(\Fpbar))[\omega_{p^{\infty}}^{(n)}]$ in the construction) gives a $((\Zp\units \oplus_{\Fp\units} \uFpn) \rtimes \Zhat)$-Galois extension.

        To prove \cref{alphthm:Gal-summary} it remains to show that its $\Kn$-localization is represented by said character. This is established by a uniqueness arguments.

    \subsection{Organization}
        \cref{app:higher-cats} provides a brief overview of $(\infty,n)$ categories and lax limits. In this section, we explicitly compute the lax limits along the directed circle $\CEnd = \B\NN$ and the directed 2-sphere $\CEndII$, and examine their interrelations. Readers may skip this section on a first read, trusting the results that will be applied in later sections.
    
        We use the constructions and interrelations of \cref{app:higher-cats} in \cref{sec:self-maps} in order to construct $\Catperf\SED$, $\Catperf\SEDiso$ and all the adjunctions described.

        In \cref{sec:Tn-cat}, we build on the previous section to construct the lift $\vnSp$ and prove \cref{alphthm:periodicity}. In \cref{sec:Tn-Pic}, we leverage this lift and the natural group action on $\Catperf\SEDiso$ to prove \cref{alphthm:subgroup-of-pic}. Finally, in \cref{sec:galois}, we use Kummer theory to lift the Galois extension from $\SSKn$ to $\SSTn$, proving \cref{alphthm:Gal-summary}.
        
    \subsection{Conventions}
        
        We use the following terminology and notation:
        \begin{enumerate}
            \item The category of stable categories and exact functors is denoted $\Catst$, and the full subcategory of idempotent complete stable categories is denoted by $\Catperf$. The idempotent completion functor is denoted by ${(-)}^{\mrm{idem}} \colon \Catst\to\Catperf$.
            \item The category of spaces (or animae, or groupoids) is denoted by $\spc$.
            \item We denote by $\cC\core\subseteq \cC$ the maximal subgroupoid of a category $\cC$.
            \item We denote the space of morphisms between two obejcts $X,Y$ in a category $\cC$ by $\Map_{\cC}(X,Y)$ and omit $\cC$ when it is clear from context. If $\cC$ is stable we denote the mapping spectrum of $X,Y$ by $\hom_{\cC}(X,Y)$ or by $\hom(X,Y)$ if $\cC$ is clear from context.
            \item The category of presentable categories with colimit-preserving functors is denoted by $\PrL$.
            \item For $\cC\in\CAlg(\PrL)$ and $\cD,\cE\in\Mod_{\cC}(\PrL)$ we denote the space of $\cC$-linear functors from $\cD$ to $\cE$ by $\Map^{\cC}(-,-)$. Similarly we denote the space of $\cC$-linear endomoprhisms and $\cC$-linear automorphisms by $\End^{\cC}$ and $\Aut^{\cC}$ respectively.
            \item For a category $\cC$ we denote its full subcategory of compact objects by $\cC\comp$. 
            \item For a symmetric monoidal category $\cC$ we denote its space of dualizable objects by $\cC\dbl$.
            \item We denote lax limits by $\laxlim$ and oplax limits by $\oplaxlim$. Similarly, we will denote lax operations with an arrow.
            \item For two $n$-categories $\cC,\cD$, we denote by $\laxFun(\cC,\cD)$ the $n$-categories of functors from $\cC$ to $\cD$ and lax natural transformations. We similarly denote $\oplaxFun(\cC,\cD)$ the $n$-category of functors and oplax natural transformations.
            \item For an $n$-category $\cC$ and $X,Y\in\cC$ we denote by $\MAP_{\cC}(X,Y)$ the mapping $(n-1)$-cateogry. Similarly we denote the $(n-1)$-category of endomorphisms of $X$ by $\END_{\cC}(X) \coloneqq \MAP_{\cC}(X,X)$. We omit $\cC$ when it is clear from context.
            \item For the remainder of the paper we fix a height $\chrHeight > 0$ and a prime $p$.
        \end{enumerate}
    
    \subsection{Acknowledgments} 
        I would like to thank my advisor Tomer Schlank for many helpful conversations, comments and ideas. 
        I would like to thank the entire Seminarak group, especially Shay Ben-Moshe and Shaul Ragomiv for useful comments on previous drafts of the paper.
        I would like to thank Lior Yanovski, Shaul Barkan, Anish Chedalavada, Shaul Ragimov, Shay Ben-Moshe and Leor Neuhauser, for useful discussions.
        
        I would like to express my gratitude to the University of Chicago and to the \quotes{Spectral Methods in Algebra, Geometry, and Topology} trimester program at the Hausdorff Research Institute for Mathematics, funded by the Deutsche Forschungsgemeinschaft (DFG, German Research Foundation) under Germany's Excellence Strategy - EXC-2047/1 - 390685813. Parts of this work were completed during my time as their guest.

        This project has received funding from the European Research Council (ERC) under the European Union's Horizon 2020 research and innovation programme (Grant agreement No. 3012006831).

\section{Higher categories and lax limits}\label{app:higher-cats}
    This section is devoted for the introduction of our chosen framework of $n$-categories, lax limits, and their computation in two special cases: 
    \begin{enumerate}
        \item The lax limit along the walking 1-endomorphism 
        \begin{equation*}
            \CEnd = \B\NN = \begin{tikzcd}
                \bullet
                \arrow[from=1-1, to=1-1, loop, in=55, out=125, distance=10mm]
            \end{tikzcd},
        \end{equation*}
        \item The lax limit along the walking 2-endomorphism of the identity
        \begin{equation*}
            \CEndII = \begin{tikzcd}
                \bullet
                \arrow[Rightarrow, from=1-1, to=1-1, loop, in=55, out=125, distance=10mm]
            \end{tikzcd} = 
            \begin{tikzcd}
                \star & \star
                \arrow[""{name=0, anchor=center, inner sep=0}, "\id_\star", curve={height=-12pt}, from=1-1, to=1-2]
                \arrow[""{name=1, anchor=center, inner sep=0}, "\id_\star"', curve={height=12pt}, from=1-1, to=1-2]
                \arrow[shorten <=3pt, shorten >=3pt, Rightarrow, from=0, to=1]
            \end{tikzcd}.
        \end{equation*}
    \end{enumerate}
    We use these computations in \cref{sec:self-maps}, either choosing $\cC\in\Catperf$ and $\Sigma^{\selfMapDeg}$ as an endomorphism, or choosing $\Catperf\in\bigCatt$ with $\Sigma^{\selfMapDeg}$ as the endomorphism of $\id_{\Catperf}$.

    \subsection{Higher categories and lax natural transformations}\label{subsec:lax-limits}
        We will model $n$-categories by Rezk's complete Segal $\Theta_n$ spaces~\cite{Rezk-2010-n-categories},~\cite[\textsection~13]{Barwick-Schommer-Pries-unicity},~\cite[\textsection~2]{Loubaton-2024-infty-categories}, as we now briefly recall.
        We call a classical (1,1)-category a strict category and define inductively strict $n$-categories as strict categories enriched in strict $(n-1)$-categories. An important class of strict $n$-categories are the $k$-cells $\DD^k$ (also known as the walking $k$-morphisms, $k$-disks or $k$-globes) and their \quotes{sums}, which are finite connected colimits along inclusions of smaller cells in the boundary, corresponding to composable higher arrows.
        The $k$-cells for $k\le 3$ are
        \begin{equation*}
            \begin{tikzcd}
                {\DD^0 = \pt = } & \bullet && {\DD^1 = [1] = } & \bullet & \bullet && {\DD^2=} & \bullet & \bullet \\
                \\
                && {\DD^3=} & \bullet && \bullet
                \arrow[from=1-5, to=1-6]
                \arrow[""{name=0, anchor=center, inner sep=0}, curve={height=-12pt}, from=1-9, to=1-10]
                \arrow[""{name=1, anchor=center, inner sep=0}, curve={height=12pt}, from=1-9, to=1-10]
                \arrow[""{name=2, anchor=center, inner sep=0}, curve={height=-24pt}, from=3-4, to=3-6]
                \arrow[""{name=3, anchor=center, inner sep=0}, curve={height=24pt}, from=3-4, to=3-6]
                \arrow[shorten <=3pt, shorten >=3pt, Rightarrow, from=0, to=1]
                \arrow[""{name=4, anchor=center, inner sep=0}, shift left=3, curve={height=-12pt}, shorten <=8pt, shorten >=8pt, Rightarrow, from=2, to=3]
                \arrow[""{name=5, anchor=center, inner sep=0}, shift right=3, curve={height=12pt}, shorten <=7pt, shorten >=7pt, Rightarrow, from=2, to=3]
                \arrow[shorten <=7pt, shorten >=7pt, Rightarrow, scaling nfold=3, from=5, to=4]
            \end{tikzcd}.
        \end{equation*}
        and some examples of their cellular sums are
        \begin{equation*}
            \begin{tikzcd}
                \bullet & \bullet & \bullet & \bullet && \bullet & \bullet & \bullet \\
                \\
                \bullet & \bullet && \bullet && \bullet && \bullet && \bullet & \bullet
                \arrow[from=1-1, to=1-2]
                \arrow[from=1-2, to=1-3]
                \arrow[from=1-3, to=1-4]
                \arrow[from=1-6, to=1-7]
                \arrow[""{name=0, anchor=center, inner sep=0}, curve={height=-12pt}, from=1-7, to=1-8]
                \arrow[""{name=1, anchor=center, inner sep=0}, curve={height=12pt}, from=1-7, to=1-8]
                \arrow[""{name=2, anchor=center, inner sep=0}, curve={height=-18pt}, from=3-1, to=3-2]
                \arrow[""{name=3, anchor=center, inner sep=0}, curve={height=18pt}, from=3-1, to=3-2]
                \arrow[""{name=4, anchor=center, inner sep=0}, from=3-1, to=3-2]
                \arrow[""{name=5, anchor=center, inner sep=0}, curve={height=-30pt}, from=3-4, to=3-6]
                \arrow[""{name=6, anchor=center, inner sep=0}, curve={height=30pt}, from=3-4, to=3-6]
                \arrow[""{name=7, anchor=center, inner sep=0}, curve={height=-30pt}, from=3-8, to=3-10]
                \arrow[""{name=8, anchor=center, inner sep=0}, curve={height=30pt}, from=3-8, to=3-10]
                \arrow[""{name=9, anchor=center, inner sep=0}, from=3-8, to=3-10]
                \arrow[from=3-10, to=3-11]
                \arrow[shorten <=3pt, shorten >=3pt, Rightarrow, from=0, to=1]
                \arrow[shorten <=2pt, shorten >=2pt, Rightarrow, from=2, to=4]
                \arrow[shorten <=2pt, shorten >=2pt, Rightarrow, from=4, to=3]
                \arrow[""{name=10, anchor=center, inner sep=0}, shift left=4, curve={height=-12pt}, shorten <=9pt, shorten >=9pt, Rightarrow, from=5, to=6]
                \arrow[""{name=11, anchor=center, inner sep=0}, shift right=4, curve={height=12pt}, shorten <=9pt, shorten >=9pt, Rightarrow, from=5, to=6]
                \arrow[""{name=12, anchor=center, inner sep=0}, shorten <=8pt, shorten >=8pt, Rightarrow, from=5, to=6]
                \arrow[shorten <=4pt, shorten >=4pt, Rightarrow, from=7, to=9]
                \arrow[""{name=13, anchor=center, inner sep=0}, shift left=4, curve={height=-6pt}, shorten <=5pt, shorten >=5pt, Rightarrow, from=9, to=8]
                \arrow[""{name=14, anchor=center, inner sep=0}, shift right=4, curve={height=6pt}, shorten <=4pt, shorten >=4pt, Rightarrow, from=9, to=8]
                \arrow[shorten <=4pt, shorten >=4pt, Rightarrow, scaling nfold=3, from=11, to=12]
                \arrow[shorten <=4pt, shorten >=4pt, Rightarrow, scaling nfold=3, from=12, to=10]
                \arrow[shorten <=6pt, shorten >=6pt, Rightarrow, scaling nfold=3, from=14, to=13]
            \end{tikzcd}.
        \end{equation*}

        The collection of all cellular sums forms a (strict) category $\Theta_n$. The category $\Theta_n$ can be formally constructed inductively as the wreath product $\Delta \wr \Theta_{n-1}$. 
        The category $\Cat_n$ of $n$-categories is defined to be the full subcategory of $\catname{PSh}(\Theta_n)$ on those presheaves satisfying Segal and completion conditions. 
        There is an embedding $\Theta_n \into \Cat_n$ representing the corresponding $n$-categories. The image of $a\in \Theta_n$  represents the evaluation at $a$ functor:
        \begin{equation*}
            \Map_{\Cat_n} (a, \cC) \simeq \cC_a.
        \end{equation*}

        For an $n$-category $\cC\in\Cat_n$, we define its space of $k$-morphisms to be $\Map(\DD^k, \cC)$.

        The category of $\infty$-categories $\Cat_{\infty} = \lim_n \Cat_n$ was long conjectured to have a monoidal structure called the Gray tensor product which we think of as a lax cartesian product and denote $\gray$. The Gray tensor product, originally defined in~\cite{Gray-2006-Gray} for strict 2-categories and developed further in many works such as~\cite{Johnson-Freyd-Scheimbauer-2017-Funlax, Gaitsgory-Rozenblyum-2019-DAG, Gagna-Harpaz-Lanari-2021-Gray, Maehara-2021-Gray, Campion-Maehara-2023-Gray, Loubaton-2024-infty-categories}. In~\cite{Campion-2023-Gray}, Campion have constructed a simple universal property for the Gray tensor product and have constructed a model satisfying it. Alas, not all models have been shown to satisfy it.
        We will use Lubaton's model~\cite[Construction 2.3.1.17]{Loubaton-2024-infty-categories}, that was shown to have some computational properties we will employ. 
        
        For a typical example, the Gray tensor product of the arrow category $[1]$ with itself is the lax square
        \begin{equation*}
            \begin{tikzcd}
                \bullet & \bullet \\
                \bullet & \bullet
                \arrow[from=1-1, to=1-2]
                \arrow[from=1-1, to=2-1]
                \arrow[from=1-2, to=2-2]
                \arrow[Rightarrow, from=2-1, to=1-2]
                \arrow[from=2-1, to=2-2]
            \end{tikzcd}.
        \end{equation*}

        Sometimes, wanting to stay in the realm of $n$-categories, we will denote also by $\gray$ the truncated Gray tensor product
        \begin{equation*}
            \Cat_n \times \Cat_n \xto{\gray} \Cat_{2n} \xto{\tau_{\le n}} \Cat_n.
        \end{equation*}

        For $\cC\in\Cat_n$, the functors 
        \begin{equation*}
            \cC \gray - \colon \Cat_n \to \Cat_n, \qquad - \gray \cC \colon \Cat_n \to \Cat_n
        \end{equation*}
        admit right adjoints denoted by 
        \begin{equation*}
            \laxFun(\cC,-)\colon \Cat_n \to \Cat_n, \qquad \oplaxFun(\cC,-)\colon \Cat_n \to \Cat_n
        \end{equation*}
        which are the $n$-categories of functors and (op)lax natural transformations (see~\cite[Definition 4.1.4.1]{Loubaton-2024-infty-categories}). They are described as sheaves on $\Theta_n$ as follows:
        \begin{equation*}
            \begin{split}
                & \laxFun(\cC, \cD)_a \simeq \Map(a, \laxFun(\cC,\cD)) \simeq \Map(\cC \gray a, \cD), \\
                & \oplaxFun(\cC, \cD)_a \simeq \Map(a, \oplaxFun(\cC,\cD)) \simeq \Map(a \gray \cC, \cD).
            \end{split}
        \end{equation*}
        In particular, $k$-morphisms in $\laxFun(\cC,\cD)$ are functors $\cC \gray \DD^k \to \cD$. 
        Thus objects are just functors $\cC \to \cD$ and morphisms are lax natural transformations, which informally, are  a collection of morphisms $\alpha_X \colon FX \to GX$ for every $X\in\cC$ and a 2-morphism
        \begin{equation*}
            \begin{tikzcd}
                FX && FY \\
                \\
                GX && GY
                \arrow["Ff", from=1-1, to=1-3]
                \arrow["{\alpha_X}", from=1-1, to=3-1]
                \arrow["{\alpha_Y}", from=1-3, to=3-3]
                \arrow["{\alpha_f}", shorten <=17pt, shorten >=17pt, Rightarrow, from=3-1, to=1-3]
                \arrow["Gf", from=3-1, to=3-3]
            \end{tikzcd}
        \end{equation*}
        for every $f\colon X\to Y \in \cC$.

        These constructions satisfy certain useful properties:
        \begin{lemma}\label{lem:laxFun-pushouts-to-pullbacks}
            For $\cD\in\Cat_n$ the functor $\laxFun(-, \cD) \colon \Cat_n\op \to \Cat_n$ preserves limits (that is $\laxFun(-,\cD)$ sends colimits in $\Cat_n$ to limits in $\Cat_n$).
        \end{lemma}
        \begin{proof}
            Let $\cC_{(-)} \colon \cI \to \Cat_n$ be a diagram and $\cD\in\Cat_n$. 
            Using that the Gray tensor product commutes with colimits in each variable (\cite[Construction 2.3.1.17]{Loubaton-2024-infty-categories}), we get for every $\cE\in\Cat_n$:
            \begin{equation*}
                \begin{split}
                    \Map(\cE, \laxFun(\colim_i \cC_i, \cD)) 
                    & \simeq \Map((\colim_i \cC_i) \gray \cE, \cD) \\
                    & \simeq \Map(\colim_i (\cC_i \gray \cE), \cD) \\
                    & \simeq \lim_i \Map(\cC_i \gray \cE, \cD) \\
                    & \simeq \lim_i \Map(\cE, \laxFun(\cC_i, \cD)) \\
                    & \simeq \Map(\cE, \lim_i \laxFun(\cC_i, \cD)).
                \end{split}
            \end{equation*}
            As this is a natural isomorphism the result follows by Yoneda lemma.
        \end{proof}

        \begin{lemma}\label{lem:laxFun-and-gray}
            Let $\cI,\cJ, \cD \in\Theta_n$. Then
            \begin{equation*}
                \laxFun(\cJ, \laxFun(\cI, \cD)) \simeq \laxFun(\cI \gray \cJ, \cD).
            \end{equation*}
        \end{lemma}
        \begin{proof}
            Let $\cE\in \Cat_n$. Then we compute:
            \begin{equation*}
                \begin{split}
                    \Map(\cE, \laxFun(\cJ, \laxFun(\cI, \cD)))
                    & \simeq \Map(\cJ \gray \cE, \laxFun(\cI, \cD)) \\
                    & \simeq \Map(\cI \gray \cJ \gray \cE, \cD) \\
                    & \simeq \Map(\cE, \laxFun(\cI \gray \cJ, \cD)).
                \end{split}
            \end{equation*}
            The result follows by Yoneda lemma.
        \end{proof}

    \subsection{Lax limits}
        Let $\cC\in\Cat_n$ and $\cI\in\Cat_n$. The terminal functor $\cI \to \pt$ induces a functor
        \begin{equation*}
            \Delta \colon \cC \simeq \laxFun(\pt, \cC) \to \laxFun(\cI, \cC)
        \end{equation*}
        which sends $X\in \cC$ to the constant diagram on $X$. 

        \begin{definition}
            A right adjoint for $\Delta$ at $F\in\laxFun(\cI,\cC)$ is called the lax limit of $F$ and is denoted $\laxlim_{\cI} F$. 
        \end{definition}

        We call an arrow $\Delta X \to F$ in $\laxFun(\cI, \cC)$ a lax cone over $F$. Thus the lax limit of $F$ is a terminal lax cone over $F$.

        \begin{lemma}[Lax Fubini]\label{lem:lax-fubini}
            Let $\cC \in \Cat_{n}$ be a category having all small lax limits. Let $\cI, \cJ \in \Cat_n$ and let $F \colon \cI \gray \cJ \to \cC$ be a diagram. Then
            \begin{equation*}
                \laxlim_{\cI \gray \cJ} F \simeq \laxlim_{\cJ} \laxlim_{\cI} F.
            \end{equation*}
            That is, $\laxlim_{\cI \gray \cJ}$ is isomorphic to the composition
            \begin{equation*}
                \laxFun(\cI \gray \cJ, \cC) \simeq \laxFun(\cJ, \laxFun(\cI,\cC)) \xto{\laxFun(\cJ, \laxlim_\cI(-))} \laxFun(\cI, \cC) \xto{\laxlim_\cI} \cC.
            \end{equation*}
        \end{lemma}
        \begin{proof}
            The composition
            \begin{equation*}
                \cC \xto{\Delta_\cJ} \laxFun(\cJ, \cC) \xto{\laxFun(\cJ, \Delta_\cI(-))} \laxFun(\cJ, \laxFun(\cI, \cC))
            \end{equation*}
            is equivalent to $\Delta_{\cI \gray \cJ}$. The lemma follows by taking right adjoints.
        \end{proof}

        \begin{lemma}\label{lem:lax-lim-of-pushout}
            Let $\cC$ be an $n$-category having all colimits and all lax limits. Let
            \begin{equation*}
                \begin{tikzcd}
                    {\mcal{A}} & {\mcal{B}} \\
                    {\mcal{C}} & {\mcal{D}}
                    \arrow["{i_B}", from=1-1, to=1-2]
                    \arrow["{i_C}", from=1-1, to=2-1]
                    \arrow["{j_B}", from=1-2, to=2-2]
                    \arrow["{j_C}", from=2-1, to=2-2]
                    \arrow["\lrcorner"{anchor=center, pos=0.125, rotate=180}, draw=none, from=2-2, to=1-1]
                \end{tikzcd}
            \end{equation*}
            be a pushout square in $\Cat_n$ and call $j_A \colon \mcal{A} \to \mcal{D}$ the composite map. Let $F\colon \mcal{D}\to \cC$ be a functor of $n$-categories. Then
            \begin{equation*}
                \laxlim_{\mcal{D}} F \simeq \laxlim_{\mcal{B}} (Fj_B) \times_{\laxlim_{\mcal{A}} (Fj_A)} \laxlim_{\mcal{C}} (Fj_C).
            \end{equation*}
        \end{lemma}
        \begin{proof}
            By \cref{lem:laxFun-pushouts-to-pullbacks}, $\laxFun(-,\cC)$ sends colimits to limits and in particular
            \begin{equation*}
                \laxFun(\mcal{D}, \cC) \cong \laxFun(\mcal{B}, \cC) \times_{\laxFun(\mcal{A},\cC)} \laxFun(\mcal{C},\cC).
            \end{equation*}
            As the map
            \begin{equation*}
                \mcal{D} \simeq \mcal{B} \sqcup_{\mcal{A}} \mcal{C} \to \pt \sqcup_{\pt} \pt \simeq \pt
            \end{equation*}
            is the terminal map, the functor
            \begin{equation*}
                \Delta_{\mcal{B}} \times_{\Delta_{\mcal{A}}} \Delta_{\mcal{C}} 
                \colon \cC 
                \simeq \cC \times_{\cC} \cC 
                \to \laxFun(\mcal{B}, \cC) \times_{\laxFun(\mcal{A},\cC)} \laxFun(\mcal{C},\cC) 
                \simeq \laxFun(\mcal{D}, \cC)
            \end{equation*}
            is identified with $\Delta_{\mcal{D}}$. The result now follows by~\cite[Theorem 5.5]{Horev-Yanovski-2017-adjoints}.
        \end{proof}

        When $\cC = \Cat_n$ is the $(n+1)$-category of $n$-categories, the lax limit of a functor $F\colon \cI \to \Cat_n$ is the category of sections of the cartesian fibration $\int_\cI F \to \cI$ defined by the Grothendieck construction, i.e. unstraightening (\cite[Construction~4.1.2.5, Proposition~4.2.3.8, Example~4.2.3.13]{Loubaton-2024-infty-categories}). In particular, all lax limits in $\Cat_n$ exist. One can for example compute:
        \begin{example}[{\cite[Example~4.2.3.13]{Loubaton-2024-infty-categories}}]\label{exm:constant-lax-lim}
            Let $\cC,\cI\in\Cat_{n}$. Then the lax limit of the constant functor $\Delta \cC \colon \cI \to \Cat_n$ is
            \begin{equation*}
                \laxlim_{\cI} \Delta\cC \simeq \laxFun(\cI, \cC).
            \end{equation*}
        \end{example}
        
        \cref{lem:lax-lim-of-pushout} allows us to compute lax limits using cell decomposition.
        In the next two subsections we will use it to compute some lax limits, specifically, of functors from \quotes{directed spheres} to $\Cat_n$.
        Let $\partial \DD^k$ be the maximal $(k-1)$-subcategory of $\DD^k$. We define the directed sphere as \mbox{$\laxsphere^k \coloneqq \DD^k / \partial \DD^k$}.
        
        \begin{remark}\label{rmrk:functors-from-CEnd}
            The directed sphere is the walking $k$-endomorphism, that is, a functor $\laxsphere^{k} \to \cU$ chooses an object $X\in \cU$ and an endomorphism $f$ of $\id_{\id_{\ddots_{\id_X}}}$.

            More concretely, for $k=1$ it chooses an object $X$ and an endomorphism $f\colon X\to X$, for $k=2$ it chooses an object $X$ and an endomorphism $f\colon \id_X \To \id_X$.
            
            A functor $\laxsphere^k \to \cU$ chooses an invertible $k$-endomorphism if and only if it factors through the groupoidification $|\laxsphere^k| \simeq S^k$.
        \end{remark}

        \subsubsection*{Lax limits along an arrow}\label{subsec:lax-limit-arrow}
        The following lemma is well known. See \cite[Theorem~2.4.2]{Christ-Dyckerhoff-Walde-2023-complex} for a computation in $\Prst$ or more generally in $\Mod_{\Mod_R}(\Prst)$. The general method for this computation is described in \cite{Christ-Dyckerhoff-Walde-2024-lax-additivity}. 
        \begin{lemma}\label{lem:lax-lim-arrow}
            Let $G\colon [1] \to \Cat_n$ choosing an $n$-functor $G\colon \cC \to \cD$. Then its lax limit sits in a pullback diagram
            \begin{equation*}
                \begin{tikzcd}
                    {\laxlim_{[1]} G} & {\laxFun([1],\cD)} \\
                    {\cC} & {\cD}
                    \arrow[from=1-1, to=1-2]
                    \arrow[from=1-1, to=2-1]
                    \arrow["{\source}", from=1-2, to=2-2]
                    \arrow["{G}", from=2-1, to=2-2]
                    \arrow["\lrcorner"{anchor=center, pos=0.125, rotate=0}, draw=none, from=1-1, to=2-2]
                \end{tikzcd},
            \end{equation*}
            i.e.\ it is the category consisting of tuples $(X,Y,g)$ where $X\in \cC$, $Y\in \cD$ and $g\colon GX \to Y$, with lax squares as morphisms.
        \end{lemma}
        \begin{proof}
            The Grothendieck construction over an arrow is computed in~\cite[Lemma~3.8]{Gepner-Haugseng-Nikolaus-2017-lax-colimits}:
            \begin{equation*}
                \int_{[1]} G = (\cC\times [1]) \sqcup_{(\cC\times\{1\})} (\cD\times \{1\})
            \end{equation*}
            with the fibration $\int_{[1]} G \to [1]$ defined as the projection to the second coordinate. Then the category of sections $[1] \to G$ is exactly as described.
        \end{proof}

    \subsection{Lax $\NN$-fixed points, or lax limits along $\CEnd$}\label{subsec:lax-fixed-points}
        Recall that $\CEnd = \DD^1 / \partial \DD^1 = [1] / \partial [1] = \B\NN$ is the walking endomorphism. 
        Fix a 2-category $\cV$ and assume it admits lax limits along $\CEnd$.
        In this subsection we will study lax limits along $\CEnd$, in particular showing that for a diagram $(X,g) \colon \CEnd\to \cV$ the endomorphism $g$ lifts to an endomorphism of the lax limit $\laxlim_{\CEnd}(X,g)$, and that there is a 2-morphism $\hat{v}\colon g \To \id$ of endomorphisms of the lax limit.
        We end the subsection by giving concrete formulas for all construction in the special case where $\cV = \Cat$.

        \begin{definition}\label{dfn:lax-fixed-points-I}
            Let $(X,g) \colon \CEnd = \B\NN \to \cV$ be a functor. That is, the data of an object $X\in \cV$ and an endomorphism $g\colon X \to X$, or equivalently a monoid-action of $\NN$ on $X$.
            Define the lax $\NN$-fixed points of $X$ as the lax limit
            \begin{equation*}
                X\hN \coloneqq \laxlim_{\CEnd} (X,g).
            \end{equation*}
        \end{definition}

        \begin{remark}\label{rmrk:natural-map-fixed-points-I}
            We always have a map from the universal cone (when it exists) to the universal lax cone
            \begin{equation*}
                \lim_{\CEnd} (X,g) \to \laxlim_{\CEnd}(X,g) = X\hN.
            \end{equation*}
            If $g$ is invertible, then the functor $(X,g) \colon \CEnd \to \cC$ factors through $S^1 = |\CEnd|$ and this map is identified with
            \begin{equation*}
                X^{h\ZZ} = \lim_{S^1} X \simeq \lim_{\CEnd} X \to \laxlim_{\CEnd} X = X\hN.
            \end{equation*}
        \end{remark}

        \begin{definition}\label{dfn:underlying-object-I}
            The commutative diagram
            \begin{equation*}
                \begin{tikzcd}
                    \pt && \CEnd \\
                    & \cV
                    \arrow[from=1-1, to=1-3]
                    \arrow["X"', from=1-1, to=2-2]
                    \arrow["{(X,g)}", from=1-3, to=2-2]
                \end{tikzcd}
            \end{equation*}
            induces, by the functoriality of lax limits, a map
            \begin{equation*}
                u \colon X\hN = \laxlim_{\CEnd}(X,g) \to \laxlim_{\pt}X = X
            \end{equation*}
            which we call the underlying map.
        \end{definition}

        \begin{remark}\label{rmrk:lax-cone-over-S1}
            A lax cone over $(X,g)$ is a map $f \colon Y \to X$ and a 2-morphism $\eta \colon gf \To f$
            \begin{equation*}
                \begin{tikzcd}
                    & Y \\
                    X && X
                    \arrow["f"', from=1-2, to=2-1]
                    \arrow[""{name=0, anchor=center, inner sep=0}, "f", from=1-2, to=2-3]
                    \arrow["g"', from=2-1, to=2-3]
                    \arrow["\eta", shorten <=12pt, shorten >=12pt, Rightarrow, from=2-1, to=0]
                \end{tikzcd}.
            \end{equation*}

            A map $Y \to X\hN$ is equivalent to such a lax cone and the composition $Y \to X\hN \xto{u} X$ agrees with $f\colon Y \to X$.
        \end{remark}

        \begin{corollary}\label{cor:2-morphism-of-lax-fixed-points}
            There exists a canonical 2-morphism $v \colon gu \To u$.
        \end{corollary}
        \begin{proof}
            The universal lax cone corresponds to the identity map $X\hN \to X\hN$ and so is given as $u\colon X\hN \to X$ with a 2-morphism $v \colon gu \To u$.
        \end{proof}

        \begin{definition}\label{dfn:morphism-of-lax-fixed-points}
            The map
            \begin{equation*}
                g \colon (X,g) \to (X,g) \quad \in \quad \laxFun(\CEnd, \cV),
            \end{equation*}
            given by $g\colon X \to X$ and the trivial higher cells, induces an arrow 
            \begin{equation*}
                g\colon X\hN = \laxlim_{\CEnd} (X,g) \to \laxlim_{\CEnd}(X,g) = X\hN.
            \end{equation*}
        \end{definition}

        \begin{lemma}\label{lem:2-morphism-of-lax-fixed-points}
            There exists a 2-morphism
            \begin{equation*}
                \hat{v} \colon g \To \id_{X\hN} 
            \end{equation*}
            of endomorphisms of $X\hN$ lifting $v$ of \cref{cor:2-morphism-of-lax-fixed-points}, i.e.\ $u\hat{v} \simeq v$.
        \end{lemma}
        \begin{proof}
            We use the universal lax cone
            \begin{equation*}
                \begin{tikzcd}
                    & {X\hN} \\
                    \\
                    X && X
                    \arrow["u"', from=1-2, to=3-1]
                    \arrow[""{name=0, anchor=center, inner sep=0}, "u", from=1-2, to=3-3]
                    \arrow["g", from=3-1, to=3-3]
                    \arrow["v", shorten <=13pt, shorten >=13pt, Rightarrow, from=3-1, to=0]
                \end{tikzcd}
            \end{equation*}
            to construct a diagram in $\laxFun(\CEnd, \cV)$ of the form
            \begin{equation*}
                \begin{tikzcd}
                    & {\Delta X\hN = (X\hN,\id_{X\hN})} \\
                    \\
                    {(X,g)} && {(X,g)}
                    \arrow["u"', from=1-2, to=3-1]
                    \arrow[""{name=0, anchor=center, inner sep=0}, "u", from=1-2, to=3-3]
                    \arrow["g", from=3-1, to=3-3]
                    \arrow["v", shorten <=26pt, shorten >=26pt, Rightarrow, from=3-1, to=0]
                \end{tikzcd}.
            \end{equation*}
            Specifically it is the given as the the diagram
            \begin{equation*}
                \begin{tikzcd}
                    & {X\hN} & {} & X \\
                    X \\
                    & {X\hN} && X \\
                    X
                    \arrow[""{name=0, anchor=center, inner sep=0}, "u"{description}, from=1-2, to=1-4]
                    \arrow["u"{description}, from=1-2, to=2-1]
                    \arrow[Rightarrow, no head, from=1-2, to=3-2]
                    \arrow["v"{description}, shorten <=3pt, shorten >=3pt, Rightarrow, from=2-1, to=3-2]
                    \arrow["v"{description}, shorten <=11pt, shorten >=11pt, Rightarrow, from=1-4, to=3-2]
                    \arrow["g"{description}, from=1-4, to=3-4]
                    \arrow["g"{description}, from=2-1, to=1-4]
                    \arrow["g"{description}, from=2-1, to=4-1]
                    \arrow[""{name=1, anchor=center, inner sep=0}, "u"{description}, from=3-2, to=3-4]
                    \arrow["u"{description}, from=3-2, to=4-1]
                    \arrow["g"{description}, from=4-1, to=3-4]
                    \arrow["v"{description}, shorten <=20pt, shorten >=13pt, Rightarrow, from=2-1, to=0]
                    \arrow["v"{description}, shorten <=20pt, shorten >=13pt, Rightarrow, from=4-1, to=1]
                \end{tikzcd}
            \end{equation*}
            (with the trivial front 2-cell and interior 3-cell), which gives the required 2-morphism.
        \end{proof}   

        \subsubsection*{The categorical case}
        We will now look at the special case $\cV = \Cat$, finding explicit formulas for each of the constructions above. First, we can describe the lax fixed points as a lax equalizer as in~\cite{Nikolaus-Scholze-2018-TC}:

        \begin{lemma}\label{lem:lax-fixed-points-I}
            Let $(\cC, G)\colon \CEnd \to \Cat$. Then the lax fixed points $\cC\hN$ agrees with the lax equalizer $\laxeq(\cC\xrightrightarrows[\id]{G} \cC)$. That is, it sits in a pullback diagram
            \begin{equation*}
                \begin{tikzcd}
                    {\cC\hN} & {\cC^{[1]}} \\
                    \cC & {\cC\times\cC}
                    \arrow["{p^*}", from=1-1, to=1-2]
                    \arrow["u", from=1-1, to=2-1]
                    \arrow["\lrcorner"{anchor=center, pos=0.125}, draw=none, from=1-1, to=2-2]
                    \arrow["{\source,\target}", from=1-2, to=2-2]
                    \arrow["{G,\id}", from=2-1, to=2-2]
                \end{tikzcd}
            \end{equation*}
        \end{lemma}

        \begin{proof}
          $\CEnd$ is the pushout
          \begin{equation*}
              \begin{tikzcd}
                {\{0,1\}} & {[1]} \\
                \pt & \CEnd
                \arrow[from=1-1, to=1-2]
                \arrow[from=1-1, to=2-1]
                \arrow[from=1-2, to=2-2]
                \arrow[from=2-1, to=2-2]
                \arrow["\lrcorner"{anchor=center, pos=0.125, rotate=180}, draw=none, from=2-2, to=1-1]
            \end{tikzcd}
          \end{equation*}
          so the result follows by \cref{lem:lax-lim-of-pushout}.
        \end{proof}

        \begin{corollary}\label{cor:properties-of-lax-fixed-points-I}
            An object of $\cC\hN$ is equivalent to the data of $(X,v_X)$, where $X\in \cC$ and $v_X\colon  G X \to X$.
            The mapping space between two objects $(X,v_X), (Y, v_Y)$ is given by
            \begin{equation*}
                \Map_{\cC\hN}((X,v_X), (Y,v_Y)) 
                \simeq \eq\left( 
                    \Map_{\cC}(X,Y) \xrightrightarrows[
                        { v_Y \circ G- }
                        ]{ -\circ v_X } 
                        \Map_{\cC}(G X,Y)
                    \right).
            \end{equation*}
        
            The underlying object functor $u\colon \cC\hN \to \cC$ is exact, conservative and colimit-preserving.
        \end{corollary}

        \begin{proof}
            This is a special case of~\cite[Proposition~II.1.5.]{Nikolaus-Scholze-2018-TC}.
        \end{proof}

        \begin{remark}\label{rmrk:universal-cone-categorical}
            Chasing the description of \cref{lem:lax-lim-of-pushout}, the 2-morphism of \cref{cor:2-morphism-of-lax-fixed-points}, coming from the universal lax cone of $\cC\hN$, is $(u,v)$ where $v \colon Gu \To u$ is given on an object $(X, GX\xto{v_X} X)$ by $v_X$.
        \end{remark}

        \begin{remark}\label{rmrk:fixed-points-map-I}
            Assuming $G$ is invertible, by the description in \cref{lem:lax-fixed-points-I}, the functor $\cC^{h\ZZ} \to \cC\hN$ of \cref{rmrk:natural-map-fixed-points-I} sends a tuple $(X, GX\isoto X)$ to the tuple $(X, GX\to X)$, forgetting that the map is an isomorphism.
        \end{remark}

        \begin{lemma}\label{lem:morpshim-of-categories-lax-fixed-points}
            The morphism $G\colon \cC\hN \to \cC\hN$ of \cref{dfn:morphism-of-lax-fixed-points} sends $(GX \xto{v_X} X) \in \cC\hN$ to $(G^2 X \xto{Gv_X} GX)$. 
        \end{lemma}
        \begin{proof}
            Recall that a lax cone from $\cD$ over $(\cC,G)\colon \CEnd \to \Cat$ is the data of a functor $F\colon \cD \to \cC$ and a natural transformation $\eta \colon G\circ F \To F$. The map
            \begin{equation*}
                G \colon (\cC, G) \to (\cC, G) \quad \in \quad \laxFun(\CEnd, \Cat)
            \end{equation*}
            sends such a lax cone $(\cD \xto{F} \cC, \eta)$ to the lax cone $(\cD \xto{GF} \cC, G\eta)$. In particular, it sends the universal lax cone $(\cC\hN \xto{u} \cC, v)$ to $(\cC\hN \xto{Gu} \cC, Gv)$ which corresponds to the required functor $\cC\hN \to \cC\hN$.
        \end{proof}

        \begin{lemma}\label{lem:2-morphism-categorices-of-lax-fixed-points}
            The 2-morphism $\hat{v} \colon G \To \id_{\cC\hN}$ of \cref{lem:2-morphism-of-lax-fixed-points} is given on an object $(GX \xto{v_X} X)$ by
            \begin{equation*}
                \begin{tikzcd}
                    {G^2X} && GX \\
                    \\
                    GX && X
                    \arrow["Gv_X", dashed, from=1-1, to=1-3]
                    \arrow["Gv_X", from=1-1, to=3-1]
                    \arrow["v_X", from=1-3, to=3-3]
                    \arrow["v_X", dashed, from=3-1, to=3-3]
                \end{tikzcd}.
            \end{equation*}
        \end{lemma}
        \begin{proof}
            By \cref{lem:2-morphism-of-lax-fixed-points} under composition with $u$ it is given by $v \colon Gu \to u$. The result now follows by \cref{rmrk:universal-cone-categorical} and its composition with $G$.
        \end{proof}

    \subsection{Lax limits along $\CEndII$}\label{subsec:lax-limit-along-sphere}
        We now categorify \cref{subsec:lax-fixed-points}, finding a natural place where $(-)\hN$ lands, remembering the 2-morphism of \cref{lem:2-morphism-of-lax-fixed-points}.

        Let $\CEndII \coloneqq \DD^2 / \partial \DD^2$ be the walking endomorphism of the identity, that is $\CEndII$ consists of one object $\star$ and is generated by a 2-morphism $\id_\star \to \id_\star$.
        
        Fix a 3-category $\cU$ and assume lax limits of the shape $\CEndII$ exist in $\cU$ and lax limits of the shape $\CEnd$ exist in the mapping 2-categories of $\cU$. A diagram
        \begin{equation*}
            (X,\gamma) \colon \CEndII \to \cU
        \end{equation*}
        chooses $X\in \cU$ and a 2-endomorphism $\gamma \colon \id_X \To \id_X$. In this subsection we will consider the lax limit $\laxlim_{\CEndII} (X,\gamma)$.

        \begin{remark}\label{rmrk:natural-map-fixed-points-II}
            As before, we have a map from the universal cone (if it exists) to the universal lax cone
            \begin{equation*}
                \lim_{\CEndII} (X,\gamma) \to \laxlim_{\CEndII}(X,\gamma).
            \end{equation*}
            If $\gamma$ is invertible then $(X,\gamma) \colon \CEndII \to \cU$ factors through $S^2 = |\CEndII|$ and this map identifies with
            \begin{equation*}
                \lim_{S^2} (X,\gamma) \simeq \lim_{\CEndII} (X, \gamma) \to \laxlim_{\CEndII} (X,\gamma).
            \end{equation*}
        \end{remark}

        \begin{definition}\label{dfn:underlying-map-II}
            The commutative diagram
            \begin{equation*}
                \begin{tikzcd}
                    \pt && \CEndII \\
                    & \cU
                    \arrow[from=1-1, to=1-3]
                    \arrow["X"', from=1-1, to=2-2]
                    \arrow["{(X,\gamma)}", from=1-3, to=2-2]
                \end{tikzcd}
            \end{equation*}
            induces, by the functoriality of lax limits, a map 
            \begin{equation*}
                U \colon \laxlim_{\CEndII} (X,\gamma) \to \laxlim_{\pt} X \simeq X
            \end{equation*}
            which we call the underlying map.
        \end{definition}

        \begin{remark}\label{rmrk:lax-cone-over-S2}
            A lax cone over $(X,\gamma)$ is a map $f\colon Y \to X$ and a 3-morphism $\eta \colon \gamma f \Rrightarrow \id_f$:
            \begin{equation*}
                \begin{tikzcd}
                    Y && X
                    \arrow[""{name=0, anchor=center, inner sep=0}, "f", curve={height=-18pt}, from=1-1, to=1-3]
                    \arrow[""{name=1, anchor=center, inner sep=0}, "f"', curve={height=18pt}, from=1-1, to=1-3]
                    \arrow[""{name=2, anchor=center, inner sep=0}, "{\id_f}", shift left=2, curve={height=-12pt}, shorten <=5pt, shorten >=5pt, Rightarrow, from=0, to=1]
                    \arrow[""{name=3, anchor=center, inner sep=0}, "{\gamma f}"', shift right=2, curve={height=12pt}, shorten <=5pt, shorten >=5pt, Rightarrow, from=0, to=1]
                    \arrow["\eta", shorten <=6pt, shorten >=6pt, Rightarrow, scaling nfold=3, from=3, to=2]
                \end{tikzcd}.
            \end{equation*}
            Any map $Y \to \laxlim_{\CEndII}(X,\gamma)$ is equivalent to such a lax cone and the composition \mbox{$Y \to \laxlim_{\CEndII} (X,\gamma) \xto{U} X$} is identified with $f$.
        \end{remark}

        \begin{corollary}\label{cor:3-morphism-of-lax-limit}
            There exists a canonical 3-morphism $\alpha \colon \gamma U \Rrightarrow U$
            \begin{equation*}
                \begin{tikzcd}
                    U && U
                    \arrow[""{name=0, anchor=center, inner sep=0}, "{\gamma U}", curve={height=-18pt}, Rightarrow, from=1-1, to=1-3]
                    \arrow[""{name=1, anchor=center, inner sep=0}, "{\id_U}"', curve={height=18pt}, Rightarrow, from=1-1, to=1-3]
                    \arrow["\alpha", shorten <=5pt, shorten >=5pt, Rightarrow, scaling nfold=3, from=0, to=1]
                \end{tikzcd}.
            \end{equation*}
        \end{corollary}
        \begin{proof}
            The universal lax cone over $(X,\gamma)$ corresponds to the identity morphism $\laxlim_{\CEndII} (X,\gamma) \to \laxlim_{\CEndII} (X,\gamma)$ therefore it is given by $(U, \alpha)$ where $\alpha \colon \gamma U \Rrightarrow U$.
        \end{proof}

        As $\gamma$ is an endomorphism of $\id_X$, it induces a functor
        \begin{equation*}
            (\id_X, \gamma) \colon \CEnd \to \END(X).
        \end{equation*}

        \begin{lemma}\label{lem:lax-fixed-points-lifts}
            The map $\id_X \hN \colon X \to X$ lifts to $\laxlim_{\CEndII}(X,\gamma)$. That is, there exists a map, which  we also denote by $\id_X \hN \colon X \to \laxlim_{\CEndII}(X,\gamma)$ rendering the following diagram commutative:
            \begin{equation*}
                \begin{tikzcd}
                    X && {\laxlim_{\CEndII}(X,\gamma)} \\
                    \\
                    && X
                    \arrow["{\id_X \hN}", from=1-1, to=1-3]
                    \arrow["{\id_X\hN}"', from=1-1, to=3-3]
                    \arrow["U", from=1-3, to=3-3]
                \end{tikzcd}.
            \end{equation*}
        \end{lemma}
        \begin{proof}
            A map $X\to \laxlim_{\CEndII}(X,\gamma)$ is the same as a lax cone, i.e.\ a map $f\colon X\to X$ and a 3-morphism
            \begin{equation*}
                \begin{tikzcd}
                    f && f
                    \arrow[""{name=0, anchor=center, inner sep=0}, "\gamma f", curve={height=-18pt}, Rightarrow, from=1-1, to=1-3]
                    \arrow[""{name=1, anchor=center, inner sep=0}, "\id_f"', curve={height=18pt}, Rightarrow, from=1-1, to=1-3]
                    \arrow["\eta", shorten <=5pt, shorten >=5pt, Rightarrow, scaling nfold=3, from=0, to=1]
                \end{tikzcd}.
            \end{equation*}
            Such a lax cone, composed with the underlying map $U$ will agree with $f$. Therefore to construct such a lifting we must give a 3-morphism
            \begin{equation*}
                \begin{tikzcd}
                    {\id_X\hN} && {\id_X\hN}
                    \arrow[""{name=0, anchor=center, inner sep=0}, "\gamma", curve={height=-18pt}, Rightarrow, from=1-1, to=1-3]
                    \arrow[""{name=1, anchor=center, inner sep=0}, "\id"', curve={height=18pt}, Rightarrow, from=1-1, to=1-3]
                    \arrow["\hat{v}", shorten <=5pt, shorten >=5pt, Rightarrow, scaling nfold=3, from=0, to=1]
                \end{tikzcd},
            \end{equation*}
            which is exactly \cref{lem:2-morphism-of-lax-fixed-points}.
        \end{proof}

        \subsubsection*{The categorical case}\label{subsubec:categorical-S2}
        We will now look at the special case $\cU = \Cat_2$, finding explicit formulas for each of the constructions above:

        \begin{lemma}\label{lem:lax-fixed-points-II}
            Let $(\cV, \gamma)\colon \CEndII \to \Cat_2$, i.e.\ a 2-category $\cV$ and a natural transformation $\gamma \colon \id_{\cV} \To \id_{\cV}$. An object of $\laxlim_{\CEndII}(\cV, \gamma)$ is equivalent to the data of $X\in \cV$ with a 2-morphism $\alpha_X \colon \gamma_X \To \id_X$.
        \end{lemma}

        \begin{proof}
            We can write $\CEndII = \DD^2 / \partial \DD^2$ also as $\pmb{\square} / \partial\, \pmb{\square}$, where $\pmb{\square} \coloneqq [1] \gray [1]$ is the lax square
            \begin{equation*}
                \pmb{\square} = [1] \gray [1] = \begin{tikzcd}
                    \bullet & \bullet \\
                    \bullet & \bullet
                    \arrow[from=1-1, to=1-2]
                    \arrow[from=1-1, to=2-1]
                    \arrow[from=1-2, to=2-2]
                    \arrow[Rightarrow, from=2-1, to=1-2]
                    \arrow[from=2-1, to=2-2]
                \end{tikzcd}
            \end{equation*}
            and $\partial \, \pmb{\square}$ is its maximal 1-subcategory. To see this, note that by~\cite[Corollary~5.12]{Campion-2023-pasting}, both quotients can be computed in strict $n$-categories, where it is obvious.
            Thus, by \cref{lem:lax-lim-of-pushout}, $\laxlim_{\CEndII}(\cV,\gamma)$ sits in a pullback diagram
            \begin{equation*}
                \begin{tikzcd}
                    {\laxlim_{\CEndII}(\cV,\gamma)} & {\laxlim_{\pmb{\square}}(\cV,\gamma)} \\
                    \cV & {\laxlim_{\partial\, \pmb{\square}}(\cV,\gamma)}
                    \arrow[from=1-1, to=1-2]
                    \arrow["U", from=1-1, to=2-1]
                    \arrow["\lrcorner"{anchor=center, pos=0.125}, draw=none, from=1-1, to=2-2]
                    \arrow[from=1-2, to=2-2]
                    \arrow[from=2-1, to=2-2]
                \end{tikzcd}.
            \end{equation*}
            We can compute $\laxlim_{\pmb{\square}}$ using Fubini (\cref{lem:lax-fubini}):
            \begin{equation*}
                \laxlim_{\pmb{\square}} (\cV, \gamma) = \laxlim_{([1] \gray [1])} (\cV,\gamma) \simeq \laxlim_{[1]} \laxlim_{[1]} (\cV,\gamma)
            \end{equation*}
            where the diagram is of the form
            \begin{equation*}
                \begin{tikzcd}
                    \cV && \cV \\ \\
                    \cV && \cV
                    \arrow["\id", from=1-1, to=1-3]
                    \arrow[""{name=0, anchor=center, inner sep=0}, "\id"', from=1-1, to=3-1]
                    \arrow[""{name=1, anchor=center, inner sep=0}, "\id", from=1-3, to=3-3]
                    \arrow["\id", from=3-1, to=3-3]
                    \arrow["\gamma", shorten <=20pt, shorten >=20pt, Rightarrow, from=0, to=1]
                \end{tikzcd}.
            \end{equation*}
            We first compute the vertical lax limit (along the identity morphism), getting the diagram
            \begin{equation*}
                \big(\laxFun([1], \cV) \xto{\gamma} \laxFun([1], \cV)\big) \colon [1] \to \Cat_2,
            \end{equation*}
            where $\gamma$ is the functor taking an arrow $f\colon X \to Y$ to the composition
            \begin{equation*}
                X \xto{\gamma_X} X \xto{f} Y.
            \end{equation*}
            Computing the lax limit, using \cref{lem:lax-lim-arrow}, we get the category 
            \begin{equation*}
                \laxFun([1], \laxFun([1], \cV)) \times_{{\laxFun([1],\cV)}^2} \laxFun([1],\cV)
            \end{equation*}
            with $(\gamma,\id) \colon \laxFun([1],\cV) \to {\laxFun([1],\cV)}^2$. That is the category of lax squares
            \begin{equation}\label{eq:lax-square}
                \begin{tikzcd}
                    X && X' \\
                    X \\
                    Y && Y'
                    \arrow["g", from=1-1, to=1-3]
                    \arrow["{\gamma_X}", from=1-1, to=2-1]
                    \arrow["f'", from=1-3, to=3-3]
                    \arrow["f", from=2-1, to=3-1]
                    \arrow["\alpha", shorten <=17pt, shorten >=17pt, Rightarrow, from=3-1, to=1-3]
                    \arrow["h", from=3-1, to=3-3]
                \end{tikzcd}.
            \end{equation}
            The lax limit $\laxlim_{\partial\, \pmb{\square}} (\cV,\gamma)$ is taken along the constant functor $\Delta\cV \colon \partial\, \pmb{\square} \to \Cat_2$. Therefore
            \begin{equation*}
                \laxlim_{\partial\, \pmb{\square}} (\cV,\gamma) \simeq \laxFun(\partial\, \pmb{\square}, \cV).
            \end{equation*}
            It is left showing that the functor $\laxlim_{\pmb{\square}} (\cV,\gamma) \to \laxlim_{\partial\, \pmb{\square}} (\cV,\gamma)$ sends a lax square of the form (\ref{eq:lax-square}) to the (non-commutative) square
            \begin{equation*}
                \begin{tikzcd}
                    X && X' \\ \\
                    Y && Y'
                    \arrow["g", from=1-1, to=1-3]
                    \arrow["f'", from=1-3, to=3-3]
                    \arrow["f", from=1-1, to=3-1]
                    \arrow["h", from=3-1, to=3-3]
                \end{tikzcd}.
            \end{equation*}
            To see this, note that a lax cone over $(\cV,\gamma)\colon \pmb{\square} \to \Cat_2$ is a diagram of the form
            \begin{equation*}
                \begin{tikzcd}
                    &&&&&&& \cV \\
                    \cD &&&& \cV \\
                    \\
                    \\
                    &&&&&&& \cV \\
                    &&&& \cV
                    \arrow[color={rgb,255:red,92;green,92;blue,214}, Rightarrow, no head, from=1-8, to=5-8]
                    \arrow[""{name=0, anchor=center, inner sep=0}, from=2-1, to=1-8]
                    \arrow[from=2-1, to=2-5]
                    \arrow[""{name=1, anchor=center, inner sep=0}, from=2-1, to=5-8]
                    \arrow[""{name=2, anchor=center, inner sep=0}, from=2-1, to=6-5]
                    \arrow[color={rgb,255:red,214;green,92;blue,92}, Rightarrow, no head, from=2-5, to=1-8]
                    \arrow[color={rgb,255:red,102;green,64;blue,26}, Rightarrow, no head, from=2-5, to=6-5]
                    \arrow[""{name=3, anchor=center, inner sep=0}, "\gamma"', shorten <=36pt, shorten >=36pt, Rightarrow, from=6-5, to=1-8]
                    \arrow[color={rgb,255:red,26;green,102;blue,26}, Rightarrow, no head, from=6-5, to=5-8]
                    \arrow[color={rgb,255:red,92;green,92;blue,214}, shorten <=11pt, shorten >=11pt, Rightarrow, from=1-8, to=1]
                    \arrow[shorten <=33pt, shorten >=33pt, Rightarrow, scaling nfold=3, from=2-1, to=3]
                    \arrow[color={rgb,255:red,214;green,92;blue,92}, shorten <=1pt, shorten >=1pt, Rightarrow, from=2-5, to=0]
                    \arrow[color={rgb,255:red,102;green,64;blue,26}, shorten <=6pt, shorten >=6pt, Rightarrow, from=2-5, to=2]
                    \arrow[color={rgb,255:red,26;green,102;blue,26}, shorten <=5pt, shorten >=5pt, Rightarrow, from=6-5, to=1]
                \end{tikzcd}
            \end{equation*}
            and it is sent to the lax cone over $\Delta \cV \colon \partial\, \pmb{\square} \to \Cat_2$ 
            \begin{equation*}
                \begin{tikzcd}
                    &&&&&&& \cV \\
                    \cD &&&& \cV \\
                    \\
                    \\
                    &&&&&&& \cV \\
                    &&&& \cV
                    \arrow[""{name=0, anchor=center, inner sep=0}, from=2-1, to=1-8]
                    \arrow[from=2-1, to=2-5]
                    \arrow[""{name=1, anchor=center, inner sep=0}, from=2-1, to=5-8]
                    \arrow[""{name=2, anchor=center, inner sep=0}, from=2-1, to=6-5]
                    \arrow[color={rgb,255:red,92;green,92;blue,214}, shorten <=11pt, shorten >=11pt, Rightarrow, from=1-8, to=1]
                    \arrow[color={rgb,255:red,214;green,92;blue,92}, shorten <=1pt, shorten >=1pt, Rightarrow, from=2-5, to=0]
                    \arrow[color={rgb,255:red,102;green,64;blue,26}, shorten <=6pt, shorten >=6pt, Rightarrow, from=2-5, to=2]
                    \arrow[color={rgb,255:red,26;green,102;blue,26}, shorten <=5pt, shorten >=5pt, Rightarrow, from=6-5, to=1]
                    \arrow[color={rgb,255:red,92;green,92;blue,214}, Rightarrow, no head, from=1-8, to=5-8]               \arrow[color={rgb,255:red,214;green,92;blue,92}, Rightarrow, no head, from=2-5, to=1-8]
                    \arrow[color={rgb,255:red,102;green,64;blue,26}, Rightarrow, no head, from=2-5, to=6-5]
                    \arrow[color={rgb,255:red,26;green,102;blue,26}, Rightarrow, no head, from=6-5, to=5-8]
                \end{tikzcd}
            \end{equation*}
            forgetting $\gamma$ and the 3-cell.
            The universal cone, has as a functors to $\cV$, the functors sending a lax square of the form (\ref{eq:lax-square}) to $X$, $X'$, $Y$, $Y$. The 2-morphisms are given by $g$, $f$, $f'$, $h$ and the 3-morphism is exactly $\alpha$. The corresponding lax cone over $\partial\, \pmb{\square}$ represents the functor sending such a lax square to 
            \begin{equation*}
                \begin{tikzcd}
                    X && X' \\ \\
                    Y && Y'
                    \arrow["g", from=1-1, to=1-3]
                    \arrow["f'", from=1-3, to=3-3]
                    \arrow["f", from=1-1, to=3-1]
                    \arrow["h", from=3-1, to=3-3]
                \end{tikzcd}.
            \end{equation*}
        \end{proof}

        \begin{remark}\label{rmrk:universal-cone-categorical-II}
            Chasing the description of \cref{lem:lax-lim-of-pushout}, the 3-morphism of \cref{cor:3-morphism-of-lax-limit}, coming from the universal lax cone of $\laxlim_{\CEndII} (\cV,\gamma)$, is $(U,\alpha)$ where $\alpha \colon \gamma U \Rrightarrow U$ is given on an object $(X, \alpha_X \colon \gamma_X \To \id_X)$ as $\alpha_X$.
        \end{remark}

        \begin{remark}\label{rmrk:fixed-points-map-II}
            Assuming $\gamma$ is invertible, by the description \cref{lem:lax-fixed-points-II}, the functor $\lim_{S^2} (\cV, \gamma) \to \laxlim_{\CEndII} (\cV,\gamma)$ of \cref{rmrk:natural-map-fixed-points-II} sends a tuple $(X, \gamma_X \isoTo \id_X)$ to the tuple $(X, \gamma_X \To \id_X)$, forgetting that the map is invertible.
        \end{remark}

        Reinterpreting \cref{lem:lax-fixed-points-lifts} for the special case $\cV = \Cat_2$, as $\gamma \colon \id_{\cV} \to \id_{\cV}$, any $X\in \cV$ admits a morphism $\gamma_X \colon X \to X$. We can then say:
        \begin{corollary}
            Let $X\in \cV$, then there exists a lift $X\hN \in \laxlim_{\CEndII} (\cV, \gamma)$. That is, there exists such an object and $U(X\hN) \in \cV$ agrees with the lax fixed points of \cref{subsec:lax-fixed-points}.
        \end{corollary}

        \begin{remark}
            The 2-morphism $\hat{v} \colon \gamma_{X\hN} \To \id_{X\hN}$, which is part of the data of an object in $\laxlim_{\CEndII}(\cV, \gamma)$ is the 2-morphism constructred in \cref{lem:2-morphism-of-lax-fixed-points}.
        \end{remark}

        The functor $\id_{\cV}\hN$ of \cref{lem:lax-fixed-points-lifts} can be thought of as the functor
        \begin{equation*}
            (-)\hN \colon \cV \to \laxlim_{\CEndII} (\cV, \gamma)
        \end{equation*}
        taking $X\in \cV$ to $X\hN$.

        \begin{proposition}\label{prop:hN-adjoint-to-U}
            The functor $(-)\hN \colon \cV \to \laxlim_{\CEndII} (\cV, \gamma)$ is right adjoint to the underlying functor $U\colon \laxlim_{\CEndII} (\cV, \gamma) \to \cV$.
        \end{proposition}

        \begin{proof}
            In this proof we will have to distinguish between the lax fixed points landing in $\cV$ and the lax fixed points landing in $\laxlim_{\CEndII}(\cV,\gamma)$. We will denote the latter by $(-)\hN$ while the former will be denoted as the composition $U((-)\hN)$. 
            
            We prove the adjunction by constructing a unit and a counit. The map
            \begin{equation*}
                u \colon U\circ (-)\hN = U((-)\hN) \to \id_\cV
            \end{equation*}
            defined in \cref{dfn:underlying-object-I} will be a counit for the adjunction\footnote{$u$ is an unfortunate name for a counit.}.

            To construct a unit, we must construct a 2-morphism $\id_{\laxlim_{\CEndII} (\cV, \gamma)} \To U\hN$.\footnote{The name $U\hN$ might be ambiguous, as it could refer to either $U \circ (-)\hN$ or the lax fixed points of $U$ under the $\NN$-action induced by $\gamma$. Fortunately, these two definitions coincide.}
            Equivalently, by the exponential adjunction, we need to provide a functor 
            \begin{equation*}
                \hat{\alpha} \colon \laxlim_{\CEndII} (\cV, \gamma) \to (\laxlim_{\CEndII}(\cV, \gamma))^{[1]} \simeq \laxlim_{\CEndII}(\cV^{[1]}, \gamma).
            \end{equation*}

            Using \cref{rmrk:lax-cone-over-S2}, it is equivalent to a functor $\alpha \colon \laxlim_{\CEndII} (\cV, \gamma) \to \cV^{[1]}$ together with a 3-morphism
            \begin{equation*}
                \begin{tikzcd}
                    \alpha && \alpha
                    \arrow[""{name=0, anchor=center, inner sep=0}, "{\gamma}", curve={height=-18pt}, Rightarrow, from=1-1, to=1-3]
                    \arrow[""{name=1, anchor=center, inner sep=0}, "\id"', curve={height=18pt}, Rightarrow, from=1-1, to=1-3]
                    \arrow["\eta", shorten <=5pt, shorten >=5pt, Rightarrow, scaling nfold=3, from=0, to=1]
                \end{tikzcd}.
            \end{equation*}
 
            Recall the 3-morphism $\alpha \colon \gamma U \Rrightarrow \id_U$ of \cref{rmrk:universal-cone-categorical-II}. This is a lax cone from $U$ over the diagram
            \begin{equation*}
                (U, \gamma U) \colon \CEnd \to \MAP(\laxlim_{\CEndII} (\cV,\gamma), \cV)
            \end{equation*}
            and so corresponds (\cref{rmrk:lax-cone-over-S1}) to a 2-morphism $\alpha \colon U \To U(U\hN) \simeq U \circ (-)\hN \circ U$, which we can think of as a functor $\alpha \colon \laxlim_{\CEndII}(\cV,\gamma) \to \cV^{[1]}$. Indeed, on an object $X$, $\alpha_X \colon X \to X\hN$ corresponds by \cref{rmrk:lax-cone-over-S1} to a 2-morphism $\alpha_X \colon \gamma_X \To \id_X$.

            Giving a 3-morphism $\eta$ as above requires to give for any $(X,\alpha_X)\in \laxlim_{\CEndII}(\cV,\gamma)$ a 2-morphism $\gamma_X \To \id_X$ in a coherent way, but such 2-morphism is given by $\alpha_X$.

            It is now left to check the zig-zag identities of the unit and the counit, i.e.\ that the compositions
            \begin{align*}
                & U \xTo{\id \circ \hat\alpha} U\circ U\SMd = U{(-)}\SMd \circ U \xTo{u\circ \id} U \\
                & {(-)}\SMd \xTo{\hat\alpha \circ \id} U\SMd \circ {(-)}\SMd = {(-)}\SMd \circ U{(-)}\SMd \xTo{\id \circ u} {(-)}\SMd
            \end{align*}
            are natural isomorphisms. It is therefore sufficient to check that on objects, and this is straightforward.

        \end{proof}



\section{Categories of (asymptotically defined) self-maps}\label{sec:self-maps}
    In this section we introduce a general framework for periodicity phenomena.
    
    Let $D$ be a quotient group of $\Zhat$. The goal of this section is the construction of the categories $\Catperf\SED$, $\Catperf\SEDiso$ of categories with asymptotically defined natural endomorphisms of the identity (of specific degrees) and of asymptotically defined natural isomorphisms of the identity (of specific degrees), along with their relations.

    In the next section, we will reformulate the periodicity theorem~\cite[Theorem 9]{Hopkins-Smith-nilpotenceII} describing precisely the existence of an \quotes{asymptotically defined $\vn$-natural endomorphism} of the identity of $\FSpc$, and deduce applications from it to the $\Tn$-local category.

    In \cref{subsec:SM-general} we describe, for a finite group $D \cong \ZZ/\selfMapDeg$ and a category $\cC\in \Catperf$, the general construction and attributes of the category of self-maps of $\cC$. In \cref{subsec:shifted-endomorphisms} we categorify the construction and define the category of categories with shifted endomorphism of the identity (i.e.\ natural transformations $\Sigma^d \To \id$), defined as the lax limit of $\Catperf$ along $\CEndII$ as introduced in \cref{subsec:lax-limit-along-sphere}. We then describe its properties and its connection to categories of self-maps. 
    In \cref{subsec:nilpotents-and-invertible} we discuss the connection to the (non-lax) limit --- the category of categories with shifted isomorphism of the identity and the related category of categories with locally nilpotent shifted endomorphism of the identity. 
    Finally, in \cref{subsec:asymptotically-defined}, we extend to pro-finite $D$ by taking the colimit of our constructed categories, defining categories with asymptotically defined self maps and categories of categories with asymptotically defined endomorphisms of the identity.
            
    As we will work with group actions and (de)categorifications, we will need the following 2 simple remarks:

    \begin{remark}\label{rmrk:G-action-on-id-categorified}
        Let $\cV$ be a 2-category, $X\in\cV$ and $A$ a pointed space. An $A$-local system $A\to \END(X)$ sending the base point to $\id_X$, is, by the core-forgetful adjunction, a map of pointed spaces $A\to \END(X)\core \simeq \Omega_X \cV\core$, which in turn is the same as a map of pointed spaces $\Sigma A \to \cV\core$ choosing $X$. Again, using the core-forgetful adjunction, it is the same as a local system $\Sigma A \to \cV$ sending the base point to $X$.

        In particular, a $G$-action on $\id_X$ is the same as a functor $\Sigma \BG \to \cV$ choosing $X$, or equivalently, a $\Omega \Sigma \BG$-action on $X$. 
    \end{remark}

    \begin{remark}\label{rmrk:G-action-on-id-categorified-part-2}
        Let $\cV$ be a 2-category and $X\in\cV$. A $G$-action on $X$ is the same as a functor $\BG \to \cV$ choosing $X$, and therefore induces a local system $G\to \END(X)$. When $G$ is connected, it is equivalent to a $\Omega G$-action on $\id_X$.
    \end{remark}

    \subsection{Self map categories}\label{subsec:SM-general}
        Note that lax limits in $\Catperf$ agree with lax limits in $\Cat$ so we may use the results of \cref{subsec:lax-fixed-points} replacing $\Cat$ with $\Catperf$. Let $\cC\in\Catperf$ and $\selfMapDeg\in\NN$. We will use \cref{subsec:lax-fixed-points} for the specific functor 
        \begin{equation*}
            (\cC,\Sigma^{\selfMapDeg}) \colon \CEnd \to \Catperf.
        \end{equation*}

        \begin{definition}
            Let $\cC\in\Catperf$ and $\selfMapDeg\in \NN$. Define the category $\cC\SMd$ of objects with self-maps of degree $\selfMapDeg$ as the lax fixed points
            \begin{equation*}
                \cC\SMd \coloneqq \laxlim(\CEnd \xto{(\cC, \Sigma^{\selfMapDeg})} \Catperf).
            \end{equation*}
            Equivalently (\cref{lem:lax-fixed-points-I}) it is the lax equalizer of $\cC \xrightrightarrows[\id]{\Sigma^{\selfMapDeg}} \cC$.
        \end{definition}

        \begin{definition}
            Let $\cC\in\Catperf$ and $\selfMapDeg\in \NN$. Define the category $\cC^{h\ZZ[\selfMapDeg]}$ of objects with self-isomorphism of degree $\selfMapDeg$ as the fixed points
            \begin{equation*}
                \cC^{h\ZZ[\selfMapDeg]} \coloneqq \lim(\B\ZZ \xto{(\cC, \Sigma^{\selfMapDeg})} \Catperf) \simeq \lim(\CEnd \xto{(\cC, \Sigma^{\selfMapDeg})} \Catperf).
            \end{equation*}
            Equivalently, it is the equalizer of $\cC \xrightrightarrows[\id]{\Sigma^{\selfMapDeg}} \cC$.
        \end{definition}

        We summarize the main properties of of this construction from the previous section:
        \begin{proposition}\label{cor:properties-of-SM}\ 
            \begin{enumerate}
                \item\label{1} An object of $\cC\SMd$ is equivalent to the data of $(X,\periodNat[X])$, where $X\in \cC$ and $\periodNat[X]\colon  \Sigma^{\selfMapDeg} X \to X$.
                The mapping space between two objects $(X,\periodNat[X]), (Y, \periodNat[Y])$ is given by
                \begin{equation*}
                    \Map_{\cC\SMd}((X,v_X), (Y,v_Y)) 
                    \simeq \eq\left( 
                        \Map_{\cC}(X,Y) \xrightrightarrows[
                            { v_Y \circ \Sigma^{\selfMapDeg}- }
                            ]{ -\circ v_X } 
                            \Map_{\cC}(\Sigma^{\selfMapDeg} X,Y)
                        \right).
                \end{equation*}
    
                \item\label{2} The underlying object functor $u\colon  \cC\SMd \to \cC$ (induced naturally by $\pt \to \CEnd$) is exact, conservative and colimit-preserving.
    
                \item\label{3} A functor $\cD \to \cC\SMd$ is equivalent to the data of an exact functor $F\colon \cD \to \cC$ and a natural transformation $\Sigma^{\selfMapDeg} F \To F$.
    
                \item\label{4} There is a natural functor $\cC^{h\ZZ[\selfMapDeg]} \to \cC\SMd$ that forgets that the self-map is an isomorphism.
            \end{enumerate}
        \end{proposition}
        \begin{proof}
            This follows by \cref{cor:properties-of-lax-fixed-points-I} (which is just~\cite[Proposition~II.1.5.]{Nikolaus-Scholze-2018-TC}), \cref{rmrk:lax-cone-over-S1}, \cref{rmrk:fixed-points-map-I}.
        \end{proof}

        \begin{proposition}\label{lem:natural-map-for-SM}
            There exists a natural transformation
            \begin{equation*}
                \hat{v} \colon  \Sigma^{\selfMapDeg} \To \id_{\cC\SMd}
            \end{equation*}
            of endofunctors of $\cC\SMd$.
            At an object $(X,\periodNat[X])$ it is given by 
            \begin{equation*}
                \begin{tikzcd}
                    {\Sigma^{2\selfMapDeg}X} & {\Sigma^{\selfMapDeg}X} \\
                    {\Sigma^{\selfMapDeg}X} & X
                    \arrow["{\Sigma^{\selfMapDeg}v}", dashed, from=1-1, to=1-2]
                    \arrow["{\Sigma^{\selfMapDeg} v}", from=1-1, to=2-1]
                    \arrow["v", from=1-2, to=2-2]
                    \arrow["v", dashed, from=2-1, to=2-2]
                \end{tikzcd}
            \end{equation*}
        \end{proposition}    
        \begin{proof}
            This follows by \cref{lem:morpshim-of-categories-lax-fixed-points} and \cref{lem:2-morphism-categorices-of-lax-fixed-points}.
        \end{proof}

        \subsubsection*{Power maps}
        We now focus on the interactions between $\cC\SMd$ for different values of $\selfMapDeg$:
        \begin{definition}\label{dfn:power-map}
            Let $\powerDeg\in \NN$. Consider the $\powerDeg$-composition map $\CEnd = \B\NN \xto{\powerDeg} \B\NN = \CEnd$ and look at the commutative diagram 
            \begin{equation*}
                \begin{tikzcd}
            	{\CEnd} && \Catperf \\
            	{\CEnd}
            	\arrow["{(\cC,\Sigma^{\powerDeg\selfMapDeg})}", from=1-1, to=1-3]
            	\arrow["\powerDeg"', from=1-1, to=2-1]
            	\arrow["{(\cC,\Sigma^{\selfMapDeg})}"', from=2-1, to=1-3]
                \end{tikzcd}.
            \end{equation*}
            The functoriality of lax limits induces the \quotes{$\powerDeg$-power map}
            \begin{equation*}
                {(-)}^{\powerDeg} \colon \cC\SMd \to \cC\SMmd.
            \end{equation*}
        \end{definition}

        \begin{lemma}\label{lem:power-map-on-SMd}
            ${(-)}^{\powerDeg}$ sends a self-map $v\colon \Sigma^{\selfMapDeg} X \to X$ to its $\powerDeg$-power
            \begin{equation*}
                v^m \colon \Sigma^{\powerDeg\selfMapDeg}X \xto{\Sigma^{(\powerDeg-1)\selfMapDeg} v} \Sigma^{(\powerDeg-1)\selfMapDeg} X \xto{\Sigma^{(\powerDeg-2) \selfMapDeg} v} \cdots \xto{\Sigma^{\selfMapDeg} v} \Sigma^{\selfMapDeg} X \xto{v} X.
            \end{equation*}
        \end{lemma}
        \begin{proof}
            Let $(F,\gamma)$ be a lax cone over 
            \begin{equation*}
                (\cC,\Sigma^{\selfMapDeg}) \colon \CEnd \to \Catperf.
            \end{equation*}
            Thus, $F\colon \cD\to \cC$ is an  exact functor and $\gamma \colon \Sigma^{\selfMapDeg}F \To F$ (\cref{cor:properties-of-SM}(\ref{3})).
            Composing with $\powerDeg \colon \CEnd \to \CEnd$ induces the lax cone $(F, \gamma^{\powerDeg})$ over 
            \begin{equation*}
                (\cC,\Sigma^{\powerDeg\selfMapDeg}) \colon \CEnd \to \Catperf,
            \end{equation*}
            where
            \begin{equation*}
                \gamma^{\powerDeg} \colon \Sigma^{\powerDeg\selfMapDeg} F \xTo{\Sigma^{(\powerDeg-1) \selfMapDeg} \gamma} \Sigma^{(\powerDeg-1)\selfMapDeg} F \xTo{\Sigma^{(\powerDeg-2)\selfMapDeg} \gamma} \cdots \xTo{\gamma} F.
            \end{equation*}
            In particular, the universal lax cone $(u\colon \cC\SMd \to \cC, v)$ is sent to $(u\colon \cC\SMd \to \cC, v^{\powerDeg})$. This represents the functor $\cC\SMd \to \cC\SMmd$ sending $(X,v_X)$ to $(X, v_X^{\powerDeg})$.
        \end{proof}

    \subsection{Categories with shifted endomorphisms of the identity}\label{subsec:shifted-endomorphisms}
        
        We now categorify the construction of \cref{subsec:SM-general}, finding a nice category in which ${(-)}\SMd$, along with the natural transformation $v$ of \cref{lem:natural-map-for-SM}, lands. 

        We will use \cref{subsubec:categorical-S2} with $\cV = \Catperf$ and $\gamma = \Sigma^{\selfMapDeg}$. We will also replace $\Cat_2$ by a presentable version for applications in \cref{subsec:asymptotically-defined}.
        
        The functor $\Sigma^\selfMapDeg$ is an (invertible) endomorphism of each object of $\Catperf$, or equivalently, an (invertible) endomorphism of $\id_{\Catperf}$ in $\bigCatt$. We shall define the category $\Catperf\SEd$ to be the corresponding lax limit along $\CEndII$. The objects of $\Catperf\SEd$ are categories $\cC\in\Catperf$ with a natural transformation $\Sigma^{\selfMapDeg} \To \id_{\cC}$. 

        Although it is possible to define $\Catperf\SEd$ directly, using 3-categories significantly simplifies both the construction and its universal property. Therefore, we work in the higher categorical framework to enable more streamlined formal proofs. For better readability, we provide concrete descriptions of the outcomes of the 3-categorical manipulations throughout the text.

        \subsubsection*{Presentability}

        It will be important to us in \cref{subsec:asymptotically-defined} that $\Catperf$, and our constructed categories, are presentable. Because we work in 2-categories, we would have liked to use a notion of presentable 2-categories. Unfortunately, there is not yet such an accepted notion. One candidate was suggested by Stefanich in~\cite{stefanich-2020-presentable}. We will take a similar approach, yet simpler to define, as we do not need to require that the functor categories are also presentable. Thus we will define a 3-category $\PPrr$ of presentable categories enriched in categories:
        
        By~\cite{Heine-2023-equivalence-enriched},~\cite{Ben-Moshe-2024-Naturality-Yoneda} there exists a 2-fully faithful embedding 
        \begin{equation*}
            \chi \colon \Mod_{\Cat}(\PrL) \into (\bigCatt)^\L.
        \end{equation*}
        from the category of modules over $\Cat$ in $\PrL$ to the category of (large) categories enriched in categories, i.e. the category of 2-categories, and left adjoints. We can pull the 3-category structure on $\bigCatt$ to $\Mod_{\Cat}(\PrL)$:

        \begin{definition}
            Let $\PPrr$ be the 3-full subcategory generated by the embedding $\chi\colon \Mod_{\Cat}(\PrL) \into {(\bigCatt)}^{\mrm{L}}$.
        \end{definition}
        Notice that as a 2-category, $\Catperf$ is tensored over $\Cat$, thus lies in the 3-category $\PPrr$.

        \subsubsection*{Definitions}

        \begin{definition}      
            A local system $\CEndII \to \PPrr$ is equivalent to a choice of a presentable 2-category $\cU\in\PPrr$ and an endomorphism $\gamma \colon \id_{\cU} \To \id_{\cU}$. 
        \end{definition}

        Let $\selfMapDeg\in \NN$. The $\selfMapDeg$-suspension functor is a natural transformation 
        \begin{equation*}
            \Sigma^{\selfMapDeg} \colon \id_{\Catperf} \To \id_{\Catperf}.
        \end{equation*}
        \begin{definition}\label{dfn:cats-with-shifted-endomorphisms-of-id}
            Define the category of stable, idempotent-complete categories with $\selfMapDeg$-shifted natural endomorphism of the identity as the lax limit
            \begin{equation*}
                \Catperf\SEd \coloneqq \laxlim_{\CEndII} (\Catperf, \Sigma^{\selfMapDeg})
            \end{equation*}
            as discussed in \cref{subsec:lax-limit-along-sphere}.

            Similarly, noting that $\Sigma^{\selfMapDeg}$ is invertible, define the category $\Catperf\SEdiso$ of stable, idempotent-complete categories with $\selfMapDeg$-shifted natural \emph{isomorphism} of the identity as the limit
            \begin{equation*}
                \Catperf\SEdiso \coloneqq \lim_{S^2} (\Catperf, \Sigma^{\selfMapDeg}) \simeq \lim_{\CEndII} (\Catperf, \Sigma^{\selfMapDeg}).
            \end{equation*}
        \end{definition}

        Noticing that limits and lax limits in $\PPrr$ are computed in $\bigCatt$, we rewrite the results from \cref{subsubec:categorical-S2} for this case:

        \begin{corollary}\label{cor:objects-of-SE}
            An object of $\Catperf\SEd$ is equivalent to the data of $(\cC,\alpha)$ where $\cC\in\Catperf$ and $\alpha\colon \Sigma^{\selfMapDeg} \To \id_{\cC}$. A map between two objects $(\cC,\alpha) \to (\cD, \beta)$ is an exact functor $f \colon \cC\to \cD$ and a 3-isomorphism
            \tikzset{
                labl/.style={anchor=south, rotate=90, inner sep=.5mm}
            }
            \begin{equation*}
                \begin{tikzcd}
            	{\Sigma^{\selfMapDeg}f = f\Sigma^{\selfMapDeg}} && f
            	\arrow[""{name=0, anchor=center, inner sep=0}, "{f\alpha}", curve={height=-24pt}, Rightarrow, from=1-1, to=1-3]
            	\arrow[""{name=1, anchor=center, inner sep=0}, "{\beta f}"', curve={height=24pt}, Rightarrow, from=1-1, to=1-3]
            	\arrow["{\rotatebox[origin=c]{90}{$\sim$}}", shorten <=6pt, shorten >=6pt, Rightarrow, scaling nfold=3, from=0, to=1]
                \end{tikzcd}.
            \end{equation*}
        \end{corollary}
        \begin{proof}
            This is \cref{lem:lax-fixed-points-II}.
        \end{proof}

        \begin{corollary}\label{cor:univ-prop-of-SE}
            A functor $\cU \to \Catperf\SEd$ is the data of a functor $F \colon \cU \to \Catperf$ along with a 3-morphism
            \begin{equation*}
                \begin{tikzcd}
                    F && F
                    \arrow[""{name=0, anchor=center, inner sep=0}, "{\Sigma^{\selfMapDeg}}", curve={height=-18pt}, Rightarrow, from=1-1, to=1-3]
                    \arrow[""{name=1, anchor=center, inner sep=0}, "\id"', curve={height=18pt}, Rightarrow, from=1-1, to=1-3]
                    \arrow["\eta", shorten <=5pt, shorten >=5pt, Rightarrow, scaling nfold=3, from=0, to=1]
                \end{tikzcd}.
            \end{equation*}
        \end{corollary}
        \begin{proof}
            This follows by \cref{rmrk:lax-cone-over-S2}.
        \end{proof}

        \begin{corollary}\label{cor:map-from-iso-SE}
            The map $\pt \to \CEndII$ induces an underlying category functor $U\colon \Catperf\SEd \to \Catperf$, sending $(\cC, \alpha)$ to $\cC$. There is a natural functor $\Catperf\SEdiso \to \Catperf\SEd$, forgetting that the shifted endomorphism is an isomorphism.
        \end{corollary}
        \begin{proof}
            The first part follows from \cref{lem:lax-fixed-points-II}, the second is \cref{rmrk:universal-cone-categorical-II}.
        \end{proof}

        \begin{corollary}
            The functor ${(-)}\SMd \colon \Catperf \to \Catperf$ admits a lift, which by abuse of notations we denote ${(-)}\SMd\colon \Catperf \to \Catperf\SEd$ rendering the following diagram commutative:
            \begin{equation*}
                \begin{tikzcd}
            	&& {\Catperf\SEd} \\
            	\\
            	\Catperf && \Catperf
            	\arrow["U", from=1-3, to=3-3]
            	\arrow["{{(-)}\SMd}", from=3-1, to=1-3]
            	\arrow["{{(-)}\SMd}", from=3-1, to=3-3]
                \end{tikzcd}.
            \end{equation*}
        \end{corollary}
        \begin{proof}
            This is \cref{lem:lax-fixed-points-lifts}.
        \end{proof}
        
        \begin{proposition}\label{prop:SMd-adjoint-U}
            The functor ${(-)}\SMd \colon \Catperf \to \Catperf\SEd$ is right adjoint to the underlying category functor $U\colon \Catperf\SEd \to \Catperf$.
        \end{proposition}
        \begin{proof}
            This is exactly \cref{prop:hN-adjoint-to-U}.
        \end{proof}

        \subsubsection*{Power maps}        

        \begin{definition}\label{dfn:SED-power-map}
            By the functoriality of lax limits, the commutative diagram
            \begin{equation*}
                \begin{tikzcd}
            	{S^2} &&& {\PPrr} \\ \\
            	{S^2}
            	\arrow["\powerDeg", from=1-1, to=3-1]
            	\arrow["{\Catperf, \Sigma^{\powerDeg\selfMapDeg}}", from=1-1, to=1-4]
            	\arrow["{\Catperf, \Sigma^{\selfMapDeg}}"', from=3-1, to=1-4]
                \end{tikzcd}
            \end{equation*}
            induces maps 
            \begin{equation*}
                {(-)}^{\powerDeg}\colon \Catperf\SEd \to \Catperf\SEmd, \qquad {(-)}^{\powerDeg}\colon \Catperf\SEdiso \to \Catperf\SEdiso[\powerDeg\selfMapDeg]
            \end{equation*}
            in $\PPrr$, which we call the $\powerDeg$-power maps.  
        \end{definition}

        \begin{lemma}
            ${(-)}^{\powerDeg}$ sends $(\cC,\alpha)$ to $(\cC,\alpha^\powerDeg)$, where 
            \begin{equation*}
                \alpha^{\powerDeg}\colon \Sigma^{\powerDeg\selfMapDeg} \xto{\Sigma^{(\powerDeg-1)\selfMapDeg}\alpha} \Sigma^{(\powerDeg-1)\selfMapDeg} \xto{\Sigma^{(\powerDeg-2)\selfMapDeg}\alpha} \cdots \xto{\Sigma^{\selfMapDeg}\alpha} \Sigma^{\selfMapDeg} \xto{\alpha} \id_{\cC}.
            \end{equation*}
        \end{lemma}
        \begin{proof}
            Recall (\cref{cor:univ-prop-of-SE}) that lax cones over $(\Catperf, \Sigma^{\selfMapDeg})\colon \CEndII \to \PPrr$ are the data of $\cU\in\PPrr$, a map $F\colon \cU \to \Catperf$ in $\PPrr$ and a 3-morphism
            \begin{equation*}
                \begin{tikzcd}
                    F && F
                    \arrow[""{name=0, anchor=center, inner sep=0}, "{\Sigma^{\selfMapDeg}}", curve={height=-18pt}, Rightarrow, from=1-1, to=1-3]
                    \arrow[""{name=1, anchor=center, inner sep=0}, "\id"', curve={height=18pt}, Rightarrow, from=1-1, to=1-3]
                    \arrow["\gamma", shorten <=5pt, shorten >=5pt, Rightarrow, scaling nfold=3, from=0, to=1]
                \end{tikzcd}.
            \end{equation*}
            Composition with $\powerDeg\colon S^2 \to S^2$, sends such a lax cone to the lax cone
            \begin{equation*}
                F\colon \cU \to \Catperf, 
                \begin{tikzcd}
                    F && F
                    \arrow[""{name=0, anchor=center, inner sep=0}, "{\Sigma^{\selfMapDeg}}", curve={height=-18pt}, Rightarrow, from=1-1, to=1-3]
                    \arrow[""{name=1, anchor=center, inner sep=0}, "\id"', curve={height=18pt}, Rightarrow, from=1-1, to=1-3]
                    \arrow["\gamma^{\powerDeg}", shorten <=5pt, shorten >=5pt, Rightarrow, scaling nfold=3, from=0, to=1]
                \end{tikzcd}.
            \end{equation*}
            In particular, the universal lax cone
            \begin{equation*}
                U\colon \Catperf\SEd \to \Catperf, 
                \begin{tikzcd}
                    U && U
                    \arrow[""{name=0, anchor=center, inner sep=0}, "{\Sigma^{\selfMapDeg}}", curve={height=-18pt}, Rightarrow, from=1-1, to=1-3]
                    \arrow[""{name=1, anchor=center, inner sep=0}, "\id"', curve={height=18pt}, Rightarrow, from=1-1, to=1-3]
                    \arrow["\alpha", shorten <=5pt, shorten >=5pt, Rightarrow, scaling nfold=3, from=0, to=1]
                \end{tikzcd}
            \end{equation*}
            is sent to the lax cone
            \begin{equation*}
                U\colon \Catperf\SEd \to \Catperf, 
                \begin{tikzcd}
                    U && U
                    \arrow[""{name=0, anchor=center, inner sep=0}, "{\Sigma^{\selfMapDeg}}", curve={height=-18pt}, Rightarrow, from=1-1, to=1-3]
                    \arrow[""{name=1, anchor=center, inner sep=0}, "\id"', curve={height=18pt}, Rightarrow, from=1-1, to=1-3]
                    \arrow["\alpha^{\powerDeg}", shorten <=5pt, shorten >=5pt, Rightarrow, scaling nfold=3, from=0, to=1]
                \end{tikzcd}
            \end{equation*}
            which implies that the corresponding functor $\Catperf\SEd \to \Catperf\SEmd$ sends $(\cC,\alpha)$ to $(\cC,\alpha^\powerDeg)$.
        \end{proof}

        Note that as a map in $\PPrr$, ${(-)}^{\powerDeg}\colon \Catperf\SEd \to \Catperf\SEmd$ admits a right adjoint 
        \begin{equation*}
            \sqrt[\powerDeg]{(-)}\colon \Catperf\SEmd \to \Catperf\SEd.
        \end{equation*}
        \begin{proposition}
            Let $(\cC,\alpha)\in\Catperf\SEmd$. Then $\sqrt[\powerDeg]{(\cC,\alpha)}$ is equivalent to $\cC\SMd \times_{\cC\SMmd} \cC$ --- the category of tuples $(X, v, \gamma)$ where $X\in \cC$, $v\colon \Sigma^{\selfMapDeg} X \to X$ and $\gamma$ is an isomorphism $v^{\powerDeg} \isoto \alpha_X$.
        \end{proposition}
        \begin{proof}
            Let $(\cC,\alpha)$ be a category. Then by \cref{cor:properties-of-SM}, $\alpha$ defines a functor $\alpha\colon \cC \to \cC\SMmd$ which is a section of $u$. Using the $\powerDeg$-power map ${(-)}^{\powerDeg}\colon \cC\SMd\to \cC\SMmd$, define
            \begin{equation*}
                R(\cC,\alpha) \coloneqq \cC\SMd \times_{\cC\SMmd} \cC.
            \end{equation*}
            Then objects of $R$ are self maps $\Sigma^{\selfMapDeg} X \xto{v} X$ with an isomorphism $v^{\powerDeg} \simeq \alpha_X$. It is clear that $R$ lifts to a functor $R\colon \Catperf\SEmd \to \Catperf$. Notice moreover, that for every $(\cC,\alpha)$, $R(\cC,\alpha)$ admits a natural transformation $v\colon \Sigma^{\selfMapDeg} \To \id$ induced by the natural transformation of \cref{lem:natural-map-for-SM}. By functoriality this lifts to a 3-map
            \begin{equation*}
                \begin{tikzcd}
                    R && R
                    \arrow[""{name=0, anchor=center, inner sep=0}, "{\Sigma^{\selfMapDeg}}", curve={height=-18pt}, Rightarrow, from=1-1, to=1-3]
                    \arrow[""{name=1, anchor=center, inner sep=0}, "\id"', curve={height=18pt}, Rightarrow, from=1-1, to=1-3]
                    \arrow["v", shorten <=5pt, shorten >=5pt, Rightarrow, scaling nfold=3, from=0, to=1]
                \end{tikzcd}
            \end{equation*}
            which is a lax cone over $\CEndII \xto{(\Catperf, \Sigma^{\selfMapDeg})} \bigCatt$ and therefore induces a functor 
            \begin{equation*}
                R\colon \Catperf\SEmd \to \Catperf\SEd.
            \end{equation*}
            We will show it is right adjoint to ${(-)}^{\powerDeg}$.  
            
            We will construct the unit as a map 
            \begin{equation*}
                u \colon \Catperf\SEd \to {(\Catperf\SEd)}^{[1]} \simeq {({\Catperf}^{[1]})}\SEd,
            \end{equation*}
            which is equivalent to, by \cref{cor:univ-prop-of-SE}, to a functor $u \colon \Catperf\SEd \to {\Catperf}^{[1]}$ and a 3-morphism
            \begin{equation*}
                \begin{tikzcd}
                    u && u
                    \arrow[""{name=0, anchor=center, inner sep=0}, "{\Sigma^{\selfMapDeg}}", curve={height=-18pt}, Rightarrow, from=1-1, to=1-3]
                    \arrow[""{name=1, anchor=center, inner sep=0}, "\id"', curve={height=18pt}, Rightarrow, from=1-1, to=1-3]
                    \arrow["\eta", shorten <=5pt, shorten >=5pt, Rightarrow, scaling nfold=3, from=0, to=1]
                \end{tikzcd}.
            \end{equation*}
            Let $(\cC,\alpha)\in\Catperf\SEd$. Recall that $R(\cC,\alpha^{\powerDeg}) \simeq \cC\SMd \times_{\cC\SMmd} \cC$, where the functor $\cC\SMd \to \cC\SMmd$ is the $\powerDeg$-power map and the functor $\cC \to \cC\SMmd$ is $\alpha^{\powerDeg}$. We have a natural functor
            \begin{equation*}
                \begin{split}
                    u_{(\cC,\alpha)} \colon \cC & \to \cC\SMd \times_{\cC\SMmd} \cC \\
                    X & \mapsto (\Sigma^{\powerDeg} X \xto{\alpha_{X}} X, X).
                \end{split}
            \end{equation*}
            More formally, $\alpha$ defines a functor $\alpha \colon \cC \to \cC\SMd$ and we have a commutative diagram
            \begin{equation*}
                \begin{tikzcd}
                    \cC & \cC \\
                    {\cC\SMd} & {\cC\SMmd}
                    \arrow[Rightarrow, no head, from=1-1, to=1-2]
                    \arrow["\alpha", from=1-1, to=2-1]
                    \arrow["{\alpha^{\powerDeg}}", from=1-2, to=2-2]
                    \arrow["{{(-)}^{\powerDeg}}", from=2-1, to=2-2]
                \end{tikzcd}
            \end{equation*}
            which induces a functor $u_{(\cC,\alpha)} \colon \cC \to R(\cC,\alpha^{\powerDeg})$. Moreover this construction is clearly functorial in $(\cC,\alpha)$ and so gives rise to a functor
            \begin{equation*}
                u \colon \Catperf\SEd \to {\Catperf}^{[1]}.
            \end{equation*}
            For each such $(\cC,\alpha)$ there is a commutative diagram of natural transformations
            \begin{equation*}
                \begin{tikzcd}
                    \cC && \cC \\
                    \\
                    \\
                    {R(\cC,\alpha^{\powerDeg})} && {R(\cC,\alpha^{\powerDeg})}
                    \arrow[""{name=1, anchor=center, inner sep=0}, "{\Sigma^{\selfMapDeg}}", curve={height=-12pt}, from=1-1, to=1-3]
                    \arrow[""{name=0, anchor=center, inner sep=0}, "\id"', curve={height=12pt}, from=1-1, to=1-3]
                    \arrow["{u_{(\cC,\alpha)}}", from=1-1, to=4-1]
                    \arrow["{u_{(\cC,\alpha)}}", from=1-3, to=4-3]
                    \arrow[""{name=3, anchor=center, inner sep=0}, "{\Sigma^{\selfMapDeg}}", curve={height=-12pt}, from=4-1, to=4-3]
                    \arrow[""{name=2, anchor=center, inner sep=0}, "\id"', curve={height=12pt}, from=4-1, to=4-3]
                    \arrow["\alpha"', shorten <=3pt, shorten >=3pt, Rightarrow, from=1, to=0]
                    \arrow["v"', shorten <=3pt, shorten >=3pt, Rightarrow, from=3, to=2]
                \end{tikzcd}
            \end{equation*}
            which is a 2-morphism $\eta_{(\cC,\alpha)}$ in ${\Catperf}^{[1]}$. Again, by functoriality, it lifts to a 3-morphism
            \begin{equation*}
                \begin{tikzcd}
                    u && u
                    \arrow[""{name=0, anchor=center, inner sep=0}, "{\Sigma^{\selfMapDeg}}", curve={height=-18pt}, Rightarrow, from=1-1, to=1-3]
                    \arrow[""{name=1, anchor=center, inner sep=0}, "\id"', curve={height=18pt}, Rightarrow, from=1-1, to=1-3]
                    \arrow["\eta", shorten <=5pt, shorten >=5pt, Rightarrow, scaling nfold=3, from=0, to=1]
                \end{tikzcd}   
            \end{equation*}
            which induces the desired unit functor.

            To define the counit we will use similar methods: Let $(\cC,\alpha)\in\Catperf\SEmd$. Denote the projection functor by
            \begin{equation*}
                c_{(\cC,\alpha)} \colon R(\cC,\alpha) \simeq \cC\SMd \times_{\cC\SMmd} \cC \to \cC,
            \end{equation*}
            which is functorial in $(\cC,\alpha)$ and so gives rise to a functor $\Catperf\SEmd \to {\Catperf}^{[1]}$.
            For each such $(\cC,\alpha)$ we are given (as part of the data of $R(\cC,\alpha)$) a natural transformations
            \begin{equation*}
                \begin{tikzcd}
                    {R(\cC,\alpha)} && {R(\cC,\alpha)} \\
                    \\
                    \\
                    \cC && \cC
                    \arrow[""{name=0, anchor=center, inner sep=0}, "{\Sigma^{\powerDeg\selfMapDeg}}", curve={height=-12pt}, from=1-1, to=1-3]
                    \arrow[""{name=1, anchor=center, inner sep=0}, "\id"', curve={height=12pt}, from=1-1, to=1-3]
                    \arrow["{c_{(\cC,\alpha)}}", from=1-1, to=4-1]
                    \arrow["{c_{(\cC,\alpha)}}", from=1-3, to=4-3]
                    \arrow[""{name=2, anchor=center, inner sep=0}, "{\Sigma^{\powerDeg\selfMapDeg}}", curve={height=-12pt}, from=4-1, to=4-3]
                    \arrow[""{name=3, anchor=center, inner sep=0}, "\id"', curve={height=12pt}, from=4-1, to=4-3]
                    \arrow["v^{\powerDeg}"', shorten <=3pt, shorten >=3pt, Rightarrow, from=0, to=1]
                    \arrow["\alpha"', shorten <=3pt, shorten >=3pt, Rightarrow, from=2, to=3]
                \end{tikzcd}
            \end{equation*}
            which, by functoriality, lifts to a 3-morphism 
            \begin{equation*}
                \begin{tikzcd}
                    c && c
                    \arrow[""{name=0, anchor=center, inner sep=0}, "{\Sigma^{\powerDeg\selfMapDeg}}", curve={height=-18pt}, Rightarrow, from=1-1, to=1-3]
                    \arrow[""{name=1, anchor=center, inner sep=0}, "\id"', curve={height=18pt}, Rightarrow, from=1-1, to=1-3]
                    \arrow["\epsilon", shorten <=5pt, shorten >=5pt, Rightarrow, scaling nfold=3, from=0, to=1]
                \end{tikzcd} .
            \end{equation*}
            $(c,\epsilon)$ is a data of a lax cone so it induces a functor 
            \begin{equation*}
                c \colon \Catperf\SEmd \to ({\Catperf}^{[1]})\SEmd \simeq {(\Catperf\SEmd)}^{[1]}
            \end{equation*}
            which is the counit of the adjunction.

            Verifying the zig-zag identities now is simple.
        \end{proof}

        \begin{corollary}\label{cor:sqrt-of-SMd}
            For any $\cC\in\Catperf$, $\sqrt[\powerDeg]{(\cC\SMmd)}\simeq \cC\SMd$.
        \end{corollary}
        \begin{proof}
            The corollary follows by taking right adjoints of the commutative diagram
            \begin{equation*}
                \begin{tikzcd}
                    {\Catperf\SEd} \\
                    {\Catperf\SEmd} & \Catperf
                    \arrow["{{(-)}^{\powerDeg}}"', from=1-1, to=2-1]
                    \arrow["U", from=1-1, to=2-2]
                    \arrow["U", from=2-1, to=2-2]
                \end{tikzcd}.
            \end{equation*}
        \end{proof}

    \subsection{Locally nilpotent and invertible natural transformations}\label{subsec:nilpotents-and-invertible}
        The canonical map $\CEndII \to |\CEndII| \simeq S^2$ induces a map in $\PPrr$ 
        \begin{equation*}
            \Catperf\SEdiso \to \Catperf\SEd,
        \end{equation*}
        forgetting that the natural map is an isomorphism. As a functor in $\PPrr$ it admits a right adjoint sending $(\cC,\alpha)$ to the subcategory of $\cC$ on which $\alpha$ acts invertibly. In this subsection we construct a left adjoint $L \colon \Catperf\SEd \to \Catperf\SEdiso$, inverting the natural transformation. 
        $L$ is a Bousfield localization and its kernel is the category of stable, idempotent complete categories with $\selfMapDeg$-shifted, locally-nilpotent natural transformation of the identity. 
        It will actually be simpler to first construct the colocalization functor, constructing $L$ as its kernel.

        We then show that the category $\Catperf\SEdiso$ has a natural action of the group $\Omega S^2 / \selfMapDeg$, which is equivalent to an action of $\Omega S^1/ \selfMapDeg$ on any $(\cC,\alpha)\in\Catperf\SEdiso$. The identification $\selfMapDeg \simeq 0$ is given by $\alpha$. Moreover, the $\powerDeg$-power map of \cref{dfn:SED-power-map} is equivariant with respect to this action.

        \begin{definition}
            Let $(\cC,\alpha)\in\Catperf\SEd$. We say that $\alpha$ is locally-nilpotent if for any $X\in\cC$ the map $\alpha_X\colon \Sigma^{\selfMapDeg} X \to X$ is nilpotent. That is, there exists $k\in\ZZ$ such that $\alpha_X^k\colon \Sigma^{k\selfMapDeg} X \to X$ is null.
            We denote by $\Catperf\SEdnil \subseteq \Catperf\SEd$ the full 2-subcategory of $\Catperf\SEd$ consisting of locally nilpotent natural transformations.
        \end{definition}

        \begin{definition}
            Define $\Nil \colon \Catperf\SEd \to \Catperf\SEdnil$ as the subfunctor of $\id_{\Catperf\SEd}$ sending $(\cC,\alpha)$ to the full subcategory of $\cC$ consisting of objects on which $\alpha$ acts nilpotently.

            We will write $\Nil_{\alpha}(\cC)\coloneqq \Nil(\cC,\alpha)$.
        \end{definition}

        \begin{lemma}
            $\Nil\colon \Catperf\SEd\to \Catperf\SEdnil$ is right adjoint to the inclusion $i\colon \Catperf\SEdnil\to\Catperf\SEd$.
        \end{lemma}
        \begin{proof}
            Note that $\Nil\circ i \simeq \id_{\Catperf\SEdnil}$. The counit $i\circ \Nil \to \id_{\Catperf\SEd}$ is the natural embedding. It is now simple to verify the zig-zag identities.
        \end{proof}

        As a consequence, $\Nil$ is a colocalization functor. By standard abuse of notations, we will usually denote the composition $\Nil\circ i$ also by $\Nil$.

        \begin{definition}
            Let $L\colon \Catperf\SEd \to \Catperf\SEd$ be the cofiber of the counit map $\Nil \to \id_{\Catperf\SEd}$ in $\Fun(\Catperf\SEd, \Catperf\SEd)$.
        \end{definition}

        To understand $L$ better, we will use presentable categories, using the fact we can embed $\Catperf$ in $\Prst$ via the $\Ind$ construction.

        \begin{definition}
            Let $\cC\in\Prst$ and $\alpha\colon  \Sigma^{\selfMapDeg}\to \id_{\cC}$. Define $\cC[\alpha^{-1}]$ to be the full subcategory consisting of objects on which $\alpha$ is invertible. 
        \end{definition}
        \begin{remark}\label{rmrk:inverting-alpha-on-Ind}
            The inclusion functor $\cC[\alpha^{-1}]\subset \cC$ admits an essentially surjective left adjoint
            \begin{equation*}
                (-)[\alpha^{-1}] \colon  \cC \to \cC[\alpha^{-1}]
            \end{equation*}
            which on an object $X\in\cC$ is given by
            \begin{equation*}
                X[\alpha^{-1}] \simeq \colim(X \xto{\alpha} \Sigma^{-\selfMapDeg}X \xto{\alpha} \Sigma^{-2\selfMapDeg}X \xto{\alpha} \cdots).
            \end{equation*}
        \end{remark}

        \begin{lemma}\label{lem:formula-for-L}
            $L$ is given on objects by
            \begin{equation*}
                L(\cC,\alpha) = {((\Ind\cC)[\alpha^{-1}])}\comp.
            \end{equation*}
        \end{lemma}
        \begin{proof}
            Denote by $\PrL_{\st,\omega}$ the category of compactly generated stable presentable categories and colimit-preserving functors sending compact objects to compact objects. By~\cite[Lemma~5.3.2.9]{Lurie-HA}, there is an equivalence of categories
            \begin{equation*}
                \Ind\colon  \Catperf \rightleftarrows \PrL_{\st,\omega} \cocolon {(-)}\comp
            \end{equation*}
            and the inclusion functor $\PrL_{\st,\omega} \into \PrL$ preserves colimits. In particular
            \begin{equation*}
                L(\cC,\alpha) \simeq {\cofib(\,\Ind(\Nil[\alpha](\cC)) \to \Ind(\cC)\,)}\comp
            \end{equation*}
            and the quotient functor $\id_{\Catperf\SEd} \to L$ is given as the restrction to compacts of
            \begin{equation*}
                \Ind(\cC) \to \cofib(\,\Ind(\Nil[\alpha](\cC)) \to \Ind(\cC)\,)
            \end{equation*}
            where the cofiber is computed in $\Prst$.
            The cofiber can be computed as
            \begin{equation*}
                \begin{split}
                        & \{X\in \Ind(\cC) \mid \Map_{\Ind(\cC)}(Z,X) \simeq \pt \ \forall Z\in \Ind(\Nil[\alpha](\cC))\} \\
                        & \simeq \{X\in \Ind(\cC) \mid \Map_{\Ind(\cC)}(Z,X) \simeq \pt \ \forall Z\in \Nil[\alpha](\cC)\}.
                \end{split}
            \end{equation*}
            Any $X$ on which $\alpha$ is invertible satisfies this condition. On the other hand, assume $\Map_{\Ind(\cC)}(Z,X)$ is contractible for every $Z\in \Nil[\alpha](\cC)$. $\alpha^2$ acts trivially on $X/\alpha \coloneqq  \cofib(\Sigma^{\selfMapDeg}X \xto{\alpha} X)$, and in particular $X/\alpha\in\Nil[\alpha](\cC)$ and the cofiber sequence
            \begin{equation*}
                \Sigma X/\alpha \xto{0} \Sigma^{\selfMapDeg} X \xto{\alpha} X
            \end{equation*}
            splits so $X \simeq \Sigma^{\selfMapDeg} X\oplus X/\alpha$. By our assumption the inclusion map $X / \alpha \to \Sigma^{\selfMapDeg} X\oplus X/\alpha \simeq X$ is null, so $X/\alpha \simeq 0$ and $\alpha$ is invertible on $X$.
        \end{proof}

        \begin{notation}
            The above lemma justifies denoting $L(\cC,\alpha)$ by $\cC[\alpha^{-1}]$. 
        \end{notation}

        \begin{remark}\label{rmrk:how-invertion-looks}
            By \cref{lem:formula-for-L} and \cref{rmrk:inverting-alpha-on-Ind}, for a category $\cC\in\Catperf$, the category $\cC[\alpha^{-1}]=L(\cC,\alpha)$ is equivalent to the idempotent completion of the full subcategory 
            \begin{equation*}
                \{X[\alpha^{-1}] \mid X\in \cC\}\subseteq \Ind(\cC).
            \end{equation*}
        \end{remark}

        \begin{corollary}\label{cor:universal-prop-of-inverting}
            $L\colon \Catperf\SEd \to \Catperf\SEd$ lands in $\Catperf\SEdiso$ and is left adjoint the the inclusion $\Catperf\SEdiso \into \Catperf\SEd$. 
        \end{corollary}
        \begin{proof}
            By \cref{lem:formula-for-L}, $L$ lands in $\Catperf\SEdiso$.
            Let $(\cC,\alpha)\in\Catperf\SEd$, $(\cD,\beta)\in\Catperf\SEdiso$, then a map $F\colon (\cC,\alpha) \to (\cD,\beta)$ must satisfy that the composition
            \begin{equation*}
                (\Nil_{\alpha}(\cC), \alpha) \to (\cC, \alpha) \xto{F} (\cD,\beta)
            \end{equation*}
            is null, as on any object $X$, $F\alpha\colon \Sigma^{\selfMapDeg} X \to X$ is both invertible and nilpotent. Thus $F$ factors uniquely through $L(\cC,\alpha) = \cofib((\Nil_{\alpha}(\cC),\alpha) \to (\cC,\alpha))$.
        \end{proof}

        \subsubsection*{Group action}
        We next note that $\Catperf\SEdiso$ has a residual action by the cofiber of the map $\selfMapDeg \colon \Omega S^2 \to \Omega S^2$:

        \begin{definition}
            Let $S^2/\selfMapDeg$ be the cofiber of the map $\selfMapDeg\colon S^2 \to S^2$. Note that in the category of $\EE_1$-groups, the cofiber of the map $\selfMapDeg\colon \Omega S^2 \to \Omega S^2$ is $\Omega (S^2/\selfMapDeg)$.
        \end{definition}
        \begin{lemma}\label{lem:groups-action-on-SEdiso}
            $\Catperf\SEdiso$ is naturally an $S^2/\selfMapDeg$-local system, i.e.\ has a natural action by the group $\Omega S^2 / \selfMapDeg$.
        \end{lemma}
        \begin{proof}
            Consider the pushout square
            \begin{equation*}
                \begin{tikzcd}
                    {S^2} & {S^2} \\
                    \pt & {S^2 / \selfMapDeg}
                    \arrow["d", from=1-1, to=1-2]
                    \arrow[from=1-1, to=2-1]
                    \arrow["\pi_{\selfMapDeg}", from=1-2, to=2-2]
                    \arrow["e", from=2-1, to=2-2]
                    \arrow["\lrcorner"{anchor=center, pos=0.125, rotate=180}, draw=none, from=2-2, to=1-1]
                \end{tikzcd}
            \end{equation*}
            and the diagram of local systems
            \begin{equation*}
                \begin{tikzcd}
                    {\PPrr^{S^2/\selfMapDeg}} && {\PPrr} \\
                    \\
                    {\PPrr^{S^2}} && {\PPrr^{S^2}}
                    \arrow["{e^*}", from=1-1, to=1-3]
                    \arrow["{\pi_{\selfMapDeg}^*}", from=1-1, to=3-1]
                    \arrow["{{(S^2)}^*}"', from=1-3, to=3-3]
                    \arrow["{(\pi_{\selfMapDeg})_*}", curve={height=-18pt}, dashed, from=3-1, to=1-1]
                    \arrow["{\selfMapDeg^*}", from=3-1, to=3-3]
                    \arrow["{(S^2)_*}"', curve={height=18pt}, dashed, from=3-3, to=1-3]
                \end{tikzcd}.
            \end{equation*}
            By Beck--Chevalley, $e^*(\pi_{\selfMapDeg})_* \simeq S^2_* \selfMapDeg^*$. In particular, for $(\Catperf, \Sigma) \in \PPrr^{S^2}$, 
            \begin{equation*}
                e^*(\pi_{\selfMapDeg})_* (\Catperf, \Sigma) \simeq S^2_* \selfMapDeg^* (\Catperf, \Sigma) = S^2_*(\Catperf, \Sigma^\selfMapDeg) = \Catperf\SEdiso.
            \end{equation*}
            The functor $e^*$ is the underlying category functor, thus $(\pi_{\selfMapDeg})_*(\Catperf, \Sigma) \in \PPrr^{S^2 / \selfMapDeg}$ is a lift of $\Catperf\SEdiso$.
        \end{proof}

        \begin{remark}\label{rmrk:action-is-suspension}
            By \cref{rmrk:G-action-on-id-categorified}, a presentation of $\Catperf\SEdiso$ as an $S^2/\selfMapDeg$-local system is equivalent to a local system $S^1/\selfMapDeg \to \End(\Catperf\SEdiso)$ choosing $\id_{\Catperf\SEdiso}$. That is, a $\Omega (S^1/\selfMapDeg)$-action on every $(\cC,\alpha)\in\Catperf\SEdiso$. $\Omega (S^1/\selfMapDeg)$ is the free $\EE_1$-group where $\selfMapDeg=0$. The action is given by the suspension functor, and $\alpha\colon \Sigma^{\selfMapDeg} \isoto \id_{\cC}$ gives the required identification.
        \end{remark}

        Recall that we have constructed an $\powerDeg$-power map ${(-)}^{\powerDeg} \colon \Catperf\SEdiso \to \Catperf\SEdiso[\powerDeg\selfMapDeg]$, and by the above lemma $\Catperf\SEdiso$ has a natural $\Omega (S^2/\selfMapDeg)$-action and $\Catperf\SEdiso[\powerDeg\selfMapDeg]$ has a natural $\Omega (S^2/\powerDeg\selfMapDeg)$-action. We have a homomorphism $q\colon \Omega (S^2/\powerDeg\selfMapDeg) \to \Omega (S^2/\selfMapDeg)$. Carefully reviewing the construction of the action, we can see that the $\powerDeg$-power map is equivariant.
        \begin{lemma}\label{lem:power-map-is-equivariant}
            The $\powerDeg$-power map ${(-)}^{\powerDeg}\colon \Catperf\SEdiso \to \Catperf\SEdiso[\powerDeg\selfMapDeg]$ is $\Omega (S^2/\powerDeg\selfMapDeg)$-equivariant.
        \end{lemma}
        \begin{proof}
            Look at the commutative diagram
            \begin{equation*}
                \begin{tikzcd}
            	{S^2} & {S^2} \\
            	{\Sigma\B\ZZ/\powerDeg\selfMapDeg} & {\Sigma\B\ZZ/\selfMapDeg}
            	\arrow[Rightarrow, no head, from=1-1, to=1-2]
            	\arrow["{\pi_{\powerDeg\selfMapDeg}}", from=1-1, to=2-1]
            	\arrow["{\pi_{\selfMapDeg}}", from=1-2, to=2-2]
            	\arrow["q", from=2-1, to=2-2]
                \end{tikzcd},
            \end{equation*}
            then the Beck--Chevalley map
            \begin{equation*}
                q^*{(-)}\SEd \simeq q^* (\pi_{\selfMapDeg})_* \to (\pi_{\powerDeg\selfMapDeg})_*\pi_{\powerDeg\selfMapDeg}^* q^* (\pi_{\selfMapDeg})_* \simeq (\pi_{\powerDeg\selfMapDeg})_* \pi_{\selfMapDeg}^* (\pi_{\selfMapDeg})_* \to (\pi_{\powerDeg\selfMapDeg})_* \simeq {(-)}\SEmd
            \end{equation*}
            is such a lift.
        \end{proof}

    \subsection{Asymptotically defined shifted endomorphisms}\label{subsec:asymptotically-defined}

        We will now use the constructed power maps (\cref{dfn:power-map}, \cref{dfn:SED-power-map}) and take the colimit of the resulted diagrams, transitioning from a finite quotient group $D$ to a general one. 
        In the decategorified case, the colimit $\colim_{D \onto \ZZ/\selfMapDeg} \cC\SMd$ will be the category of $\selfMapDeg$-self-maps in $\cC$ for any $\selfMapDeg$ with a surjection $D \onto \ZZ /\selfMapDeg$. Two such self-maps are identified if they are the same after some power. This category has an \quotes{asymptotically defined natural endomorphism of the identity functor} corresponding to the self-map itself.

        The categorified colimit $\colim_{\selfMapDeg} \Catperf\SEd$ will be the category of categories with such asymptotically defined natural endomorphism of the identity. In particular $\colim_{\selfMapDeg} \cC\SMd$ enhances to an object in there and the functor sending $\cC$ to $\colim_{\selfMapDeg} \cC\SMd$ is right adjoint to the underlying category functor.

        Restricting to categories with asymptotically defined natural isomorphisms, we extend the results of \cref{subsec:nilpotents-and-invertible} and get an action of the group $\Omega^2  \lim_{\selfMapDeg} (S^2/\selfMapDeg)$ on this category.
        
        \subsubsection*{Categories of asymptotically defined self maps}
        Fix a quotient group $D$ of $\Zhat$.
        
        \begin{definition}\label{dfn:SMD}
            Let $\cC\in \Catperf$. 
            Define the category of $D$-asymptotically defined self-maps $\cC\SMD$ as the colimit in $\Catperf$ of $\cC\SMd$ along the $\powerDeg$-power map:
            \begin{equation*}
                \cC\SMD \coloneqq  \colim_{D \onto \ZZ/\selfMapDeg} \cC\SMd.
            \end{equation*}

            The underlying object functor $u\colon  \cC\SMd \to \cC$ is compatible with this diagram
            \begin{equation*} 
                u({(X,\periodNat[X])}^{\powerDeg}) \simeq u(X,\periodNat[X])
            \end{equation*}
            and so defines an underlying object functor $u\colon  \cC\SMD \to \cC$.
        \end{definition}

        \begin{remark}\label{rmrk:SMD-sequential-colimit}
            Write $D$ as a sequential limit 
            \begin{equation*}
                D = \lim_i \ZZ / \selfMapDeg_i
            \end{equation*}
            of surjections, (in particular $\selfMapDeg_i$ divides $\selfMapDeg_{i+1}$). By cofinality, $\cC\SMD$ identifies with the sequential colimit
            \begin{equation*}
                \cC\SMD \simeq \colim(\cC\SMd[\selfMapDeg_1] \xto{(-)^{\selfMapDeg_2 / \selfMapDeg_1}} \cC\SMd[\selfMapDeg_2] \xto{(-)^{\selfMapDeg_3 / \selfMapDeg_2}} \cdots ).
            \end{equation*}
        \end{remark}
        
        \begin{remark}
            Using \cref{cor:properties-of-SM}, an object of $\cC\SMD$ is equivalent to the data of $X\in \cC$ and $v\colon  \Sigma^\selfMapDeg X \to X$ for some $D \onto \ZZ /\selfMapDeg$. Two such objects $(X, \Sigma^{\selfMapDeg}X \xto{v} X)$ and $(X, \Sigma^{\selfMapDeg'}X \xto{v'} X)$ are identified if they are asymptotically the same: there exists $c,c'$ such that $c\selfMapDeg = c'\selfMapDeg'\in D$ and $v^c \simeq {(v')}^{c'}$.
        \end{remark}

        \subsubsection*{Categories of asymptotically defined self maps}
    
        \begin{definition}\label{dfn:cats-with-asymp-defined-endomorphism}
            Define the category of \emph{stable, idempotent complete categories with $D$-asymptotically defined shifted endomorphisms of the identity} as the colimit in $\PPrr$ of $\Catperf\SEd$ along the $\powerDeg$-power maps:
            \begin{equation*}
                \Catperf\SED \coloneqq \colim_{D \onto \ZZ / \selfMapDeg}\Catperf\SEd.
            \end{equation*}

            Similarly, define the category of stable, idempotent complete categories with $D$-asymptotically defined nilpotent endomorphisms and $D$-asymptotically defined isomorphisms as
            \begin{align*}
                & \Catperf\SEDiso \coloneqq \colim_{D \onto \ZZ / \selfMapDeg} \Catperf\SEdiso \\
                & \Catperf\SEDnil \coloneqq \colim_{D \onto \ZZ / \selfMapDeg} \Catperf\SEdnil \\
            \end{align*}
        \end{definition}

        \begin{remark}\label{rmrk:SED-sequential-colimit}
            As in \cref{rmrk:SMD-sequential-colimit}, $\Catperf\SED$, $\Catperf\SEDiso$, $\Catperf\SEDnil$ can be defined as sequential colimits.
        \end{remark}

        \begin{example}
            If $D = \ZZ/\selfMapDeg$ is finite then $\Catperf\SED = \Catperf\SEd$.
        \end{example}

        \begin{remark}\label{rmrk:objects-of-SED}
            Write $D$ as a sequential limit of surjections
            \begin{equation*}
                D = \lim_i \ZZ / \selfMapDeg_i.
            \end{equation*}
            The colimit in $\PPrr$ is computed as the colimit in $\PrL$, which in turn is computed as the limit of the right adjoints. That is
            \begin{equation*}
                \Catperf\SED \simeq \lim( \Catperf\SEd[\selfMapDeg_1] \xgets{\sqrt[\selfMapDeg_2/\selfMapDeg_1]{(-)}} \Catperf\SEd[\selfMapDeg_2] \xgets{\sqrt[\selfMapDeg_3/\selfMapDeg_2]{(-)}} \cdots ).
            \end{equation*}
            Therefore, an object in $\Catperf\SED$ is a sequence of categories $\cC_i$, natural transformations $\alpha_i \colon \Sigma^{\selfMapDeg_i} \to \id_{\cC_i}$ and identifications $(\cC_i, \alpha_i) \simeq \sqrt[\selfMapDeg_{i+1}/\selfMapDeg_i]{(\cC_{i+1},\alpha_{i+1})}$. 

            Using the power-root adjunction we have a series of maps ${(\cC_i,\alpha_i)}\SMd[\selfMapDeg_{i+1} / \selfMapDeg_i] \to (\cC_{i+1}, \alpha_{i+1})$, and by the above condition, the image is exactly all elements $X$ such that $\alpha_{i+1}\colon \Sigma^{\selfMapDeg_{i+1}} X \to X$ admits a root.
        \end{remark}

        \begin{definition}
            As the underlying functor category $U\colon \Catperf\SEd\to \Catperf$ is a left adjoint and commutes with the power map, it induces an underlying category functor $U\colon \Catperf\SED \to \Catperf$ in $\PrL$. We will denote its right adjoint $R \colon \Catperf \to \Catperf\SED$.
        \end{definition}

        \begin{lemma}\label{lem:U-SMD-adjunction}
            Using the notations of \cref{rmrk:objects-of-SED}, $U$ sends the system $(\cC_i,\alpha_i)$ to the colimit $\colim_i \cC_i$ (induced by the maps $(\cC_i,\alpha_i^{\selfMapDeg_{i+1}/\selfMapDeg_i}) \to (\cC_{i+1}, \alpha_{i+1})$) and $R$ sends $\cC$ to the system $(\cC\SMd[\selfMapDeg_i], v)$. In particular $U\circ R \simeq {(-)}\SMD$. 
        \end{lemma}
        \begin{proof}
            $R$ is defined as the limit of the functors ${(-)}\SMd[\selfMapDeg_i] \colon \Catperf \to \Catperf\SEd[\selfMapDeg_i]$ and so the formula follows. $U$ is defined as the left adjoint of $R$. Using (the $(-)\op$ version of)~\cite[Corollary~1.3]{Horev-Yanovski-2017-adjoints}, we know that $U$, as a left adjoint to a limit of functors, is the colimit of the left adjoints of the functors, i.e.
            \begin{equation*}
                U((\cC_i,\alpha_i)) = \colim_i U(\cC_i,\alpha_i) = \colim_i \cC_i.
            \end{equation*}
        \end{proof}

        Justified by \cref{lem:U-SMD-adjunction}, we will continue the abuse of notations of \cref{subsec:shifted-endomorphisms} and denote $R$ by ${(-)}\SMD$.

        \begin{remark}
            Using \cref{rmrk:objects-of-SED} together with \cref{lem:U-SMD-adjunction}, we will think of an object in $\Catperf\SED$ as the category $\cC = \colim_i \cC_i$ together with an \quotes{asymptotically defined natural transformation $\alpha$}, i.e.\ an exhausting filtration $\cC_1 \to \cC_2 \to \cdots \to \cC$ and natural transformations $\alpha_i\colon \Sigma^{\selfMapDeg_i} \to \id_{\cC_i}$ with the obvious compatibility conditions.
        \end{remark}

        Repeating the argument of \cref{lem:U-SMD-adjunction} we deduce the following:
        \begin{proposition}
            The embedding $\Catperf\SEDiso \into \Catperf\SED$ admits a left adjoint $L$ sending $(\cC,\alpha)$ to $\cC[\alpha^{-1}]$ (sending each $(\cC_i,\alpha_i)$ to $(\cC_i[\alpha_i^{-1}],\alpha_i)$). Moreover $L$ is a Bousfield localization with kernel $\Catperf\SEDnil$ and the colocalization functor is $\Nil\colon \Catperf\SED \to \Catperf\SEDnil$ sending $(\cC,\alpha)$ to $\Nil_{\alpha}(\cC)$.
        \end{proposition}

        Recall that $\Catperf\SEdiso$ admits a natural $\Omega (S^2/\selfMapDeg)$-action (\cref{lem:groups-action-on-SEdiso}) and that the $\powerDeg$-power map is $\Omega (S^2/\powerDeg\selfMapDeg)$-equivariant (\cref{lem:power-map-is-equivariant}). We deduce immediately:
        \begin{corollary}\label{cor:group-action-on-SEDiso}
            The category $\Catperf\SEDiso$ admits an action of the group $\Omega S^2_D$, where
            \begin{equation*}
                S^2_D\coloneqq \lim_{D \onto \ZZ/\selfMapDeg} S^2/\selfMapDeg.
            \end{equation*}
        \end{corollary}

        \begin{lemma}\label{lem:pi2-of-SD}
            $\pi_2(S^2_D) \cong D$.
        \end{lemma}
        \begin{proof}
            Write $D$ as a sequential limit of surjections.
            \begin{equation*}
                D = \lim_i \ZZ / \selfMapDeg_i.
            \end{equation*}
            By cofinality $S^2_D = \lim_{i} S^2 / \selfMapDeg_i$.
            Milnor's exact sequence for homotopy groups (see e.g.~\cite[Chapter~IX, Theorem~3.1]{Bousfied-Kan-1972-limits-completions-localizations}), says there is an exact sequence
            \begin{equation*}
                0 \to \limone_{i} \pi_3(S^2/\selfMapDeg_i) \to \pi_2(\lim_{i} S^2/\selfMapDeg_i) \to \lim_{i} \pi_2(S^2 / \selfMapDeg_i) \to 0.
            \end{equation*}
            $S^2/\selfMapDeg_i$ is simply connected and has finite homology groups thus has finite homotopy groups. In particular, by the Mittag--Leffler criterion (see e.g.~\cite[Theorem~7.75]{Switzer-1975-algebraic-topology}), $\limone_{i} \pi_3(S^2/\selfMapDeg_i) = 0$. Thus
            \begin{equation*}
                \pi_2(\lim_{i} S^2/\selfMapDeg_i) \cong \lim_i \pi_2(S^2 / \selfMapDeg_i) = \lim_{i} \ZZ/\selfMapDeg_i = D
            \end{equation*}
            as needed.
        \end{proof}

        \begin{corollary}\label{cor:group-action-on-id-SEDiso}
            There is an action of $\Omega^2 S^2_D$ on $\id_{\Catperf\SEDiso}$ which induces, for any $(\cC,\alpha)\in\Catperf\SEDiso$,  a group homomorphism $\Sigma^{(-)} \colon D \to \pi_0(\Aut(\cC))$ lifting the suspension.
        \end{corollary}

        \begin{proof}
            There is such a group action by \cref{cor:group-action-on-SEDiso} and \cref{rmrk:G-action-on-id-categorified-part-2}. Therefore there is a group homomorphism
            \begin{equation*}
                \pi_2 (S^2_D) = \pi_0 (\Omega^2 S^2_D) \to \pi_0(\Aut(\cC)).
            \end{equation*}
            By \cref{lem:pi2-of-SD} $\pi_2(S^2_D) \cong D$. By \cref{rmrk:action-is-suspension}, $1\in D$ acts by $\Sigma$.
        \end{proof}

\section{The periodicity theorem}\label{sec:Tn-cat}
    The main goal of this section is constructing a group homomorphism 
    \begin{equation*}
        \Zn \coloneqq \lim_{\vnSelfNum} \ZZ/\vnSelfDeg \to \pi_0 \Aut(\TnComp).
    \end{equation*}
    Our main tool will be Devinatz and Smith's classical periodicity theorem. We will show that it implies the existence of a category in $\Catperf\SEvniso$ with underlying category $\TnComp$ (\cref{cor:underlying-of-nil-and-inv}). Then, using \cref{cor:group-action-on-id-SEDiso}, we get the required group homomorphism.
    \subsection{Telescopic localizations}\label{subsec:preliminaries}
        We start with some preliminary well known introduction to the telescopic categories in chromatic homotopy theory. All results here are well known, but for some we did not find a proof in the literature. For a more thorough introduction to chromatic homotopy theory we refer the reader to~\cite{Barthel-Beaudry-2020-intro-to-chromatic}. For proofs of these facts for the $\Kn$-local category we refer to~\cite{Hovey-Strickland-1999-666}.
        
        Fix a prime $p$. Let $\Kn$ be the $\chrHeight$-th Morava $K$-theory associated to a formal group law of height $\chrHeight$ over $\Fp$. It has homotopy groups given by
        \begin{equation*}
            \pi_*\Kn \cong \Fp[\vn^{\pm 1}]
        \end{equation*}
        where $\vn$ is of degree $\vnSize = 2(\extOrder)$. It is also common to define $\Kn[\infty] = \Fp$, $\Kn[0] = \QQ$, $\Kn[-1] = 0$.
        
        \begin{definition}
            Let $\FSpc\subseteq \plocSp\comp$ be the thick subcategory of $\Kn[\chrHeight-1]$-acyclics. That is
            \begin{equation*}
                \FSpc = \{X\in \plocSp\comp \mid K(\chrHeight - 1)\otimes X \simeq 0 \}
            \end{equation*}
            A spectrum in $\FSpc$ is said to be of type~$\ge \chrHeight$.
        \end{definition}
        
        Ravenel proved in~\cite[theorem~2.11]{Ravenel-1984-localizations} that $\FSpc[\chrHeight+1] \subseteq \FSpc$, Mitchell proved in~\cite[Theorem~B]{Mitchell-1985-An} that these inclusions are strict and Hopkins and Smith proved that these are all thick subcategories of $\plocSp\comp$~\cite[Theorem~7]{Hopkins-Smith-nilpotenceII}.
        
        \begin{theorem}[Thick subcategory theorem]
            The thick subcategories of $\plocSp\comp$ are exactly $\FSpc$ and they assemble into a strictly ascending filtraion
            \begin{equation*}
                \{0\} = \FSpc[\infty] \subsetneq \cdots \subsetneq \FSpc[2] \subsetneq \FSpc[1] \subsetneq \FSpc[0] = \plocSp\comp.
            \end{equation*}    
        \end{theorem}
        
        \begin{definition}[\cite{Hopkins-Smith-nilpotenceII}]\label{dfn:vn-self-map}
            Let $X$ be a $p$-local spectrum. A map $v\colon  \Sigma^d X \to X$ is said to be a $\vn$-self-map if
            \begin{enumerate}
                \item $\id_{\Kn}\otimes v \colon  \Kn\otimes \Sigma^d X \to \Kn\otimes X$ is an isomorphism.
                \item $\id_{\Kn[m]}\otimes v \colon  \Kn[m] \otimes \Sigma^d X \to \Kn[m] \otimes X$ is nilpotent for $m \neq \chrHeight$.
            \end{enumerate}
        \end{definition}
        
        Hopkins and Smith's periodicity theorem states that any spectrum of type $\ge \chrHeight$ admits an asymptotically unique $\vn$-self-map:
        
        \begin{theorem}[{Periodicity theorem~\cite[Theorem~9, Corollary~3.7]{Hopkins-Smith-nilpotenceII}}]\label{thm:periodicity}
            A $p$-local compact spectrum $X$ admits a $\vn$-self-map if and only if $X\in\FSpc$. If $f,g$ are two $\vn$-self-maps of $X$ then there exist integers $i,j$ such that $f^i \simeq g^j$. Any spectrum admitting a $\vn$-self-map admits also a $\vn$-self-map of degree $\vnSelfDeg = 2p^k(\extOrder)$ for some $\vnSelfNum$.
        \end{theorem}

        \begin{proposition}[{\cite[Corollary~3.8]{Hopkins-Smith-nilpotenceII}}]\label{prop:periodicity-morphisms}
            Let $(X,v), (Y,w)$ be two $p$-local spectra with a $\vn$-self-map of degrees $d,e$ respectively. Let $f\colon X\to Y$ be a morphism. Then there exists $i,j$ such that $di = ej$ and the following diagram commutes
            \begin{equation*}
                \begin{tikzcd}
                    {\Sigma^{di}X} & {\Sigma^{ej}Y} \\
                    X & Y
                    \arrow["{\Sigma^{di}f}", from=1-1, to=1-2]
                    \arrow["{v^i}", from=1-1, to=2-1]
                    \arrow["{w^j}", from=1-2, to=2-2]
                    \arrow["f", from=2-1, to=2-2]
                \end{tikzcd}.
            \end{equation*}
        \end{proposition}

        \begin{remark}
            We work under the convention that the $\vn[0]$-self-maps are multiplication by powers of $p$.
        \end{remark}
        
        \cref{subsec:periodicity} is devoted to the reformulation of the periodicity theorem as an equivalence of categories using the language of \cref{sec:self-maps}.
        
        Let $\chrHeight < \infty$. We let $\Fn \in \FSpc \setminus \FSpc[\chrHeight+1]$ be any compact spectrum of type $\chrHeight$. By the periodicity theorem there exists a $\vn$-self-map $v\colon  \Sigma^d \Fn \to \Fn$. Its corresponding height $\chrHeight$ telescope is defined as $\Tn \coloneqq  \Fn {[v^{-1}]}$. 
        
        Now let $\FSp\subseteq \plocSp$ be the subcategory generated by $\FSpc$ (or equivalently by $\Fn$) under colimits. Denote by $\LnfSp$ the Verdier quotient of $\FSp[\chrHeight+1]\to \plocSp$. This category is equivalent to the category of $(\Tn[0] \oplus \cdots \oplus \Tn)$-local spectra. Denote the corresponding quotient functor by
        \begin{equation*}
            \Lnf\colon  \plocSp \to \LnfSp
        \end{equation*}
        which is a smashing localization. 
        
        The category $\SpTn$ of $\Tn$-local spectra is identified with the Verdier quotient of $\FSp[\chrHeight+1]\to \FSp$. The quotient map is given by the restricted localization
        \begin{equation*}
            \LTn = \Lnf\restriction_{\FSp} \colon  \FSp \to \SpTn.
        \end{equation*}
        
        The category of $\Tn$-local spectra inherits a symmetric monoidal structure given by the $\Tn$-localized tensor product of spectra. Its unit is the $\Tn$-local sphere $\SSTn\coloneqq \LTn\SS$.
        
        A related category is the monochromatic category $\MnfSp = \ker(\LnfSp\xto{\Lnf[\chrHeight-1]}\LnfSp[\chrHeight-1])$. The inclusion $\MnfSp\to \LnfSp$ admits a right adjoint $\Mnf\colon \LnfSp \to \MnfSp$. 
        If $X\in\FSp$, then $ \Lnf[\chrHeight-1](\Lnf X) \simeq  \Lnf[\chrHeight-1]X \simeq 0$, therefore the functor $\Lnf\colon \FSp\to \LnfSp$ factors through $\MnfSp$.
        
        We will use the following simple fact:
        \begin{proposition}\label{prop:type-n-compactly-generated}
            The category $\FSp$ is compactly generated and its compacts are exactly $\FSpc$.
        \end{proposition}
        \begin{proof}
            As $\FSp$ is closed under colimits in spectra, it is clear that any compact spectrum of type~$\ge \chrHeight$ is compact in $\FSp$. As $\FSpc$ is closed under finite colimits, $\FSp$ is generated from $\FSpc$ under filtered colimits. 
            Now let $X\in \FSp$ be compact. By the previous part, $X$ can be written as a filtered colimit of compact spectra of type $\ge n$: 
            \begin{equation*}
                X \simeq \colim_i X_i.
            \end{equation*}
            Then, as $X$ is compact, the identity map
            \begin{equation*}
                X \xto {\id} X \simeq \colim_i X_i
            \end{equation*}
            factors through $X_i$ for some $i$,
            so $X$ is a retract of $X_i\in\FSpc$ whence in $\FSpc$.
        \end{proof}
    
        \begin{corollary}\label{cor:Mnf-compactly-generated}
            The monochromatic category $\MnfSp$ is compactly generated. Its compact objects are retracts of $\Lnf X$ for $X\in\FSpc$.
        \end{corollary}
        \begin{proof}
            Restricted to $\FSp$, the functors $\Mnf, \Lnf$ coincide. As $\Lnf$ is smashing, the right adjoint $\MnfSp \into \FSp$ preserves colimits. Therefore $\Lnf\colon \FSp\to \MnfSp$ preserves compact objects.
            Thus any $\Lnf$-localization of a compact type~$\ge \chrHeight$ spectrum is compact in $\MnfSp$, and whence so are their retracts.
            
            For the other direction, let $X\in \MnfSp\subseteq \LnfSp$ be compact. Write $X$ as a filtered colimit $X \simeq \colim_i \Lnf X_i$ for $X_i\in\FSpc$. Then, as $X$ is compact, the identity map
            \begin{equation*}
                X\xto {\id} X \simeq \colim_i \Lnf X_i
            \end{equation*}
            factors through $\Lnf X_i$ for some $i$, so $X$ is a retract of $\Lnf X_i$. This also shows that $\MnfSp$ is compactly generated.
        \end{proof}
    
        The functors $\LTn \colon \MnfSp \rightleftarrows \SpTn \cocolon \Mnf$ are inverses, thus we deduce:
        \begin{corollary}\label{cor:Tn-spectra-compactly-generated}
            The category $\SpTn$ of $\Tn$-local spectra is compactly generated. Its compact objects are exactly retracts of $\Tn$-localization of compact type~$\ge \chrHeight$ spectra.
        \end{corollary}
    
        \begin{corollary}\label{cor:cpt-implies-dbl}
            Any compact $\Tn$-local spectrum is dualizable.
        \end{corollary}
        \begin{proof}
            The $\Tn$-localization functor is symmetric monoidal, thus sends dualizable objects to dualizable objects. Any spectrum $X\in\FSpc\subseteq \plocSp\comp$ is dualizable and therefore so is its $\Tn$-localization.
            The claim follows since dualizability is preserved under retracts. 
        \end{proof}
    
    \subsection{The periodicity theorem rephrased}\label{subsec:periodicity}
        We now specialize to the category of compact spectra of type $\ge \chrHeight$, showing the periodicity theorem is equivalent to the existence of of a lift of $\FSpc$ in $\Catperf\SEvn$, corresponding to the asymptotically defined $\vn$-self-map. We will then show that the the $\vn$-nilpotents gives a lift of $\FSpc[\chrHeight+1]$ and inverting $\vn$ lifts the category of compact $\Tn$-local spectra.
        
        \begin{definition}
           We say that $(X, \periodNat[X]) \in {(\plocSp\comp)}\SMd$ is a $\vn$-self-map if $\periodNat[X]$ is a $\vn$-self-map in the sense of \cref{dfn:vn-self-map}. We define $\vnSpk \subseteq {(\FSpc)}\SMd[\vnSelfDeg]$ to be the full subcategory consisting of $\vn$-self-maps.
           By an abuse of notation we will denote the natural transformation of \cref{lem:natural-map-for-SM} for this category by $\vn \colon  \Sigma^{\vnSelfDeg} \to \id$.
        \end{definition}
        
        \begin{remark}\label{rmrk:subcategory-of-type->=n}
            By \cref{thm:periodicity} we could have defined $\vnSpk$ as the full subcategory of ${(\plocSp\comp)}\SMd[\vnSelfDeg]$ consisting of $\vn$-self-maps.
        \end{remark}

        \begin{lemma}
            The sequence $(\vnSpk, \vn)$ lifts to a compatible system $(\vnSp, \vn)\in\Catperf\SEvn$.
        \end{lemma}
        \begin{proof}
            We need to provide isomorphisms $\sqrt[p]{\vnSpk[\vnSelfNum+1]} \to \vnSpk$. Recall the isomorphism 
            \begin{equation*}
                F\colon \sqrt[p]{{(\FSpc)}\SMd[{\vnSelfDeg[\vnSelfNum+1]}]} \isoto {(\FSpc)}\SMd[\vnSelfDeg]
            \end{equation*}
            of \cref{cor:sqrt-of-SMd}. Recall that $\sqrt[p]{{(\FSpc)}\SMd[{\vnSelfDeg[\vnSelfNum+1]}]}$ consists of tuples $((X,v), u)$ of elements $(X,v)\in {(\FSpc)}\SMd[{\vnSelfDeg[\vnSelfNum+1]}]$ and a self-map $u$ of it of degree $\vnSelfDeg$, i.e.\ squares of the form
            \begin{equation*}
                \begin{tikzcd}
                    {\Sigma^{(\vnSelfDeg+\vnSelfDeg[\vnSelfNum+1])}X} && {\Sigma^{\vnSelfDeg[\vnSelfNum+1]}X} \\ \\
                    {\Sigma^{\vnSelfDeg}X} && X
                    \arrow["{\Sigma^{\vnSelfDeg}u}", from=1-1, to=1-3]
                    \arrow["{\Sigma^{\vnSelfDeg[\vnSelfNum+1]}v}", from=1-1, to=3-1]
                    \arrow["v", from=1-3, to=3-3]
                    \arrow["u", from=3-1, to=3-3]
                \end{tikzcd}
            \end{equation*}
            with an identification $u^p\simeq v$. The functor $F$ then sends such a square to $(X,u)$. As $u^p\simeq v$ is a $\vn$-self-map, so is $u$. Thus, restricting $F$ to $\sqrt[p]{\vnSpk[\vnSelfNum+1]}$ we get a fully faithful functor
            \begin{equation*}
                F\colon \sqrt[p]{\vnSpk[\vnSelfNum+1]} \into \vnSpk.
            \end{equation*}
            It is left verifying that $F$ is essentially surjective. Let $(\Sigma^{\vnSelfNum}X \xto{u} X) \in \vnSpk$, then $u$ is the image of
            \begin{equation*}
                \begin{tikzcd}
                    {\Sigma^{\vnSelfDeg+\vnSelfDeg[\vnSelfNum+1]}} & {\Sigma^{\vnSelfDeg[\vnSelfNum+1]}X} \\
                    {\Sigma^{\vnSelfDeg}X} & X
                    \arrow["{\Sigma^{\vnSelfDeg}u}", from=1-1, to=1-2]
                    \arrow["{\Sigma^{\vnSelfDeg[\vnSelfNum+1]}u^p}", from=1-1, to=2-1]
                    \arrow["u^p", from=1-2, to=2-2]
                    \arrow["u", from=2-1, to=2-2]
                \end{tikzcd} \in \sqrt[p]{\vnSpk[\vnSelfNum+1]}
            \end{equation*}
            under $F$.
        \end{proof}
        
        \begin{remark}
            We could have defined $\vnSp$ living in $\Catperf\SE$, as the category of $\vn$-self-maps of arbitrary degree, and not necessarily $p$-typical. By \cref{thm:periodicity}, the underlying categories in both situations identify.
        \end{remark}
        
        \begin{theorem}\label{thm:periodicity-rephrased}
            The composition
            \begin{equation*}
                u\colon U(\vnSpk) \into U((\FSpc)\SEvn) \xto{u} \FSpc
            \end{equation*}
            is an isomorphism, where $u$ is the underlying spectrum functor as in \cref{cor:properties-of-SM}(\ref{2}).
        \end{theorem}
        
        To prove this theorem we will need the following technical lemma, which proves that a full, conservative exact functor between stable categories is already fully faithful.
        \begin{lemma}\label{lem:exact-functor-equivalence}
            Let $F:\cC\to\cD$ be an exact functor between stable categories. Assume that $F$ is full\footnote{A functor $F$ is full if it is surjective on path components of mapping space, i.e.\ for every $X,Y\in \cC$ the map $\pi_0 F \colon  \pi_0\Map_{\cC}(X,Y) \to \pi_0\Map_{\cD}(FX, FY)$ is surjective.} and conservative. Then $F$ is  fully faithful.
        \end{lemma}
        
        \begin{proof}
        \newcommand{\tmpDeg}{m}
            It suffices to prove that for all $X,Y\in\cC$, 
            \begin{equation*}
                F\colon \hom_{\cC}(X,Y) \to \hom_{\cD}(FX, FY)
            \end{equation*}
            is an isomorphism. Equivalently 
            \begin{equation*}
                \pi_{\tmpDeg} F\colon  \pi_{\tmpDeg}\hom_{\cC}(X,Y) \to \pi_{\tmpDeg} \hom_{\cD}(FX,FY)
            \end{equation*}
            is an isomorphism for all $\tmpDeg$. 
            By exactness, $\pi_{\tmpDeg} F$ is given by
            \begin{equation*}
                \begin{split}
                    \pi_{\tmpDeg} \hom_{\cC}(X,Y)
                    & = \pi_0 \hom_{\cC}(\Sigma^{\tmpDeg} X, Y)
                      \xrightarrow{\pi_0 F} \pi_0 \hom_{\cD}(F\Sigma^{\tmpDeg} X, FY) \\
                    & \simeq \pi_0 \hom_{\cD}(\Sigma^{\tmpDeg} FX, FY)
                      = \pi_{\tmpDeg}\hom_{\cD}(FX,FY)
                \end{split}
            \end{equation*}
            Thus it is enough to prove that $\pi_0 F$ is an isomorphism of groups. 
            As $F$ is full, 
            \begin{equation*}
                \pi_0 F\colon  \pi_0\hom_{\cC}(X,Y) = \pi_0\Map_{\cC}(X,Y) \to \pi_0 \Map_{\cD}(FX, FY) = \pi_0 \hom_{\cD}(FX, FY)
            \end{equation*} 
            is surjective, so it is left to prove that its kernel is trivial, i.e.\ for any $f\colon X\to Y$, if $Ff\simeq 0$ then $f\simeq 0$.
            Let $f\colon X\to Y$ with $Ff \simeq 0$.
            Since $Ff\simeq 0$, the exact sequence 
            \begin{equation*}
                FY\to \cofib(Ff) \xto{Fg} \Sigma FX\simeq F\Sigma X
            \end{equation*}
            splits. Let $\overline{s}\colon F\Sigma X \to \cofib(Ff) \simeq F\cofib(f)$ be a splitting and $s\colon \Sigma X\to \cofib(f)$ a lifting, which exists as $F$ is full. Denote by $g\colon \cofib(f)\to \Sigma X$ the associated map. 
            \begin{equation*}
                F(g\circ s) = Fg \circ \overline{s} \simeq \id
            \end{equation*} 
            and as $F$ is conservative, $g\circ s$ is an isomorphism. Choose $s'\coloneqq s\circ {(g\circ s)}^{-1}$.
            Then 
            \begin{equation*}
                g\circ s' = g\circ s \circ {(g\circ s)}^{-1}\simeq \id.     
            \end{equation*}
            So $Y\to \cofib(f) \to \Sigma X$ splits which proves $f\simeq 0$.
        \end{proof}

        \begin{proof}[Proof of \cref{thm:periodicity-rephrased}]
            As $u$ is an exact functor of stable categories, by \cref{lem:exact-functor-equivalence} it is enough to prove $u$ is full, conservative and essentially surjective.
                
            Conservativity follows from \cref{cor:properties-of-SM} and as conservative functors are closed under sequential colimits.
            The periodicity theorem --- \cref{thm:periodicity} proves that $u$ is essentially surjective.
            
            Let $f\colon X\to Y$ be a map between compact spectra of type $\ge \chrHeight$. As $u$ is essentially surjective, there exist some liftings $(X, \periodNat[X]), (Y,\periodNat[Y])\in \vnSp$. Replacing the maps by large enough powers of them, we may assume $(X, \periodNat[X]),(Y, \periodNat[Y]) \in \vnSpk$ for some $\vnSelfNum$. By \cref{cor:properties-of-SM}
            \begin{equation*}
                \pi_0 \Map_{\vnSpk}((X, \periodNat[X]),(Y, \periodNat[Y])) = \{ (g, H) \mid g \colon  X\to Y, H\colon  g\circ \periodNat_X \simeq \periodNat_Y \circ g \}
            \end{equation*}
            so to show $u$ is full it is enough to show that (maybe for larger $\vnSelfNum$) $f$ commutes with $\periodNat$ up to homotopy. This is exactly \cref{prop:periodicity-morphisms}.
        \end{proof}

        \begin{corollary}\label{cor:underlying-of-nil-and-inv}
           The isomorphism $u$ of \cref{thm:periodicity-rephrased} induces isomorphisms of categories
           \begin{equation*}
               U(\Nil_{\vn}(\vnSp)) \simeq \FSpc[\chrHeight+1], \qquad U(\vnSp[\vn^{-1}]) \simeq \TnComp. 
           \end{equation*}
        \end{corollary}

        \begin{proof}
            Recall (\cref{prop:SMd-adjoint-U}) that $U$ is a left adjoint, therefore 
            \begin{equation*}
                U(\Nil_{\vn}(\vnSp)) \simeq \colim_{\vnSelfNum} U(\Nil_{\vn}(\vnSpk)).
            \end{equation*}
            The functor $U(\Nil_{\vn}(\vnSpk)) \to \FSpc$ lands in spectra with nilpotent $\vn$-self-maps, therefore of type $\ge \chrHeight+1$. Thus, restricting the isomorphism of \cref{thm:periodicity-rephrased} to $\Nil_{\vn}(\vnSp)$ gives a fully faithful functor
            \begin{equation*}
                u\colon U(\Nil_{\vn}(\vnSp)) \to \FSpc[\chrHeight+1].
            \end{equation*}
            It is left verifying it is essentially surjective. Let $X\in\FSpc[\chrHeight+1]$, then $0\colon \Sigma^{p|\vn|}X\to X$ is a $\vn$-self map and $X \simeq u(X,0)$.

            As cofibers commute with filtered colimits, $\vnSp[\vn^{-1}] = L(\vnSp,\vn)$ is the cofiber of $\Nil_{\vn}(\vnSp) \to \vnSp$. Since $U$ commutes with cofibers
            \begin{equation*}
                U(\vnSp[\vn^{-1}]) \simeq \cofib(U(\Nil_{\vn}(\vnSp)) \to U(\vnSp)) \simeq \cofib(\FSpc[\chrHeight+1] \to \FSpc) \simeq \TnComp.
            \end{equation*}
        \end{proof}

        \begin{corollary}\label{cor:group-action-on-Tn-comp}
            There is a group homomorphism
            \begin{equation*}
                \Sigma^{(-)} \colon \Zn \to \pi_0\Aut(\TnComp)
            \end{equation*}
            sending the topological generator $1\in \lim_{\vnSelfNum} \ZZ/\vnSelfDeg$ to the suspension automorphism.
        \end{corollary}
        \begin{proof}
            By \cref{cor:group-action-on-id-SEDiso} there is a $\Omega^2 S^2_{\Zn}$-action on any category in $\Catperf\SEvniso$ where $1$ acts by suspension, in particular on $\vnSp[\vn^{-1}]$. As $U$ is an exact functor, there is such a group action on $\TnComp = U(\vnSp[\vn^{-1}])$. Using \cref{lem:pi2-of-SD}, we have a group homomorphism
            \begin{equation*}
                \Zn \cong \pi_0 (\Omega^2 S^2_{\Zn}) \to \pi_0 \Aut(\TnComp).
            \end{equation*}
        \end{proof}

\section{The telescopic Picard}\label{sec:Tn-Pic}
    We describe the $\Tn$-local Picard group as the group of automorphisms of $\TnComp$ using \cref{subsec:preliminaries}. We then use the results of \cref{subsec:periodicity} to construct a large subgroup of $\Pic(\SpTn)$, in particular finding a subgroup of the form $\Zp \times \ZZ / (\extOrder)$ in the even Picard group.

    \subsection{Automorphisms and Picard}
        
        We recall the definition of the group of units. Let $\cD \in\CAlg(\PrL)$. Then $\cD$ admits a symmetric monoidal unit map
        \begin{equation*}
            \ounit[-]\colon  \spc \to \cD
        \end{equation*}
        which induces a functor in $\PrL$
        \begin{equation*}
            \ounit[-]\colon  \CMon = \CAlg(\spc) \to \CAlg(\cD).
        \end{equation*}
        Its right adjoint is corepresented by $\ounit$, i.e.\ it sends $R$ to $\Map(\ounit, R)$ which admits a natural commutative monoid structure via the multiplication of $R$.
        The right adjoint to the composition
        \begin{equation*}
            \cnSp \simeq \CMon\gp \subseteq \CMon \xto{\ounit[-]} \CAlg(\cD)
        \end{equation*}
        is called the group of units and denoted by ${(-)}\units$.
        Notice that for $R\in \CAlg(\cD)$, the underlying space of $R\units$ consists of the connected components of invertible elements in $\Map(\ounit, R)$.

        We now apply it to $\cD = \Cat$:    
        \begin{definition}
            The Picard spectrum of $\cC\in \CAlg(\Cat)$ is defined to be $\pic(\cC)\coloneqq \cC\units$. The Picard group is defined to be $\Pic(\cC) = \pi_0\pic(\cC)$.
        \end{definition}
        
        \begin{lemma}
            For $\cC\in \CAlg(\PrL)$, the Picard group $\Pic(\cC)$ is isomorphic to the group of $\cC$-linear automorphisms $\pi_0\Aut^{\cC}(\cC)$.
        \end{lemma}
        \begin{proof}
            The Picard spectrum $\pic(\cC)=\cC\units$ is given as the invertible components of $\Map(\pt, \cC)$ where $\pt$ is the terminal category (i.e.\ the unit in $\Cat$). 
            Now,
            \begin{equation*}
                \cC\core \simeq \Map(\pt, \cC) \simeq \Map^{\mrm{L}} (\spc, \cC) \simeq \Map^{\cC}(\cC, \cC) = \End^{\cC}(\cC)
            \end{equation*}
            where $\Map^{\mrm{L}}$ is the space of colimit-preserving functors (i.e.\ functors in $\PrL$) and $\Map^{\cC}$ is the space of $\cC$-linear functors in $\PrL$. The isomorphism $\cC\core \isoto \End^{\cC}(\cC)$ is given by $X\mapsto X\otimes -$ and is an isomorphism of monoids. We thus get
            \begin{equation*}
                \Pic(\cC) = \pi_0\pic(\cC) \cong \pi_0 {\End^{\cC}(\cC)}\units = \pi_0\Aut^{\cC}(\cC).
            \end{equation*}
        \end{proof}
        
        \begin{corollary}\label{thm:pic=aut}
            $\Pic(\SpTn) \cong \pi_0 \Aut(\TnComp)$.
        \end{corollary}
        \begin{proof}
            As $\SpTn$ is a mode, i.e.\ an idempotent algebra in $\PrL$ (see~\cite[\S~5]{CSY-ambiheight}), the category of $\SpTn$-modules sits fully-faithfully inside $\PrL$. Therefore
            \begin{equation*}
                \Pic(\SpTn) \cong \pi_0 \Aut^{\SpTn}(\SpTn) = \pi_0\Aut(\SpTn) \cong \pi_0 \Aut(\TnComp),
            \end{equation*}
            where the last isomorphism is true since $\SpTn$ is compactly generated (\cref{cor:Tn-spectra-compactly-generated}).
        \end{proof}
        
    \subsection{Constructing Picard elements}
        We use \cref{sec:Tn-cat} to define automorphisms of the compact $\Tn$-local category. Using \cref{thm:pic=aut}, we construct telescopic Picard elements. By \cref{cor:group-action-on-Tn-comp} the suspension lifts to a group homomorphism
        \begin{equation*}
            \Sigma^{(-)} \colon \Zn \to \pi_0\Aut(\TnComp)
        \end{equation*}
        
        Note that for $p$ odd, $\vnSize$ is prime to $p$ and so 
        \begin{equation*}
            \Zn = \lim_\vnSelfNum \ZZ / \vnSelfDeg \cong \lim_\vnSelfNum \ZZ / p^\vnSelfNum \times \ZZ / \vnSize \cong \Zp\times \ZZ / \vnSize
        \end{equation*}
        and for $p=2$
        \begin{equation*}
            \Zn = \lim_{\vnSelfNum} \ZZ / \vnSelfDeg = \lim_{\vnSelfNum} \ZZ / (2^{\vnSelfNum}\cdot 2(2^{\chrHeight}-1)) \cong \lim_{\vnSelfNum} \ZZ / 2^{\vnSelfNum+1} \times \ZZ / (2^{\chrHeight}-1) \cong \Zq{2} \times \ZZ / (2^{\chrHeight}-1).
        \end{equation*}
    
        Denoting 
        \begin{equation*}
            a_p = \begin{cases}
                2, & p \text{ is odd} \\
                1, & p = 2
            \end{cases}
        \end{equation*}
        we write $\Zn \cong \Zp\times \ZZ/ (a_p (\extOrder))$.
    
        \begin{theorem}\label{thm:ZpxZ/vn->Pic}
            The group homomorphism $\Sigma^{(-)}\colon \Zn \cong \Zp\times \ZZ/ (a_p (\extOrder)) \to \Pic(\SpTn)$ is injective.
        \end{theorem}
        
        To prove this theorem we will look at two functors and the image of Picard elements under them. Consider the $\Kn$ homology functor
        \begin{equation*}
            \Kn_*(-) \colon \SpTn \to \Mod_{\Kn_*}(\GrAb).
        \end{equation*}
        $\Kn_* \cong \Fp[v]$ is a graded field and satisfies a K\"unneth isomorphism, i.e.\ this functor is symmetric monoidal.
        \begin{lemma}
            A graded $\Kn_*$-module is invertible if and only if it is of the form $\Kn_{*-d} \coloneqq \pi_*\Sigma^{\selfMapDeg}\Kn$ for some $d\in \ZZ$. 
        \end{lemma}
        \begin{proof}
            If $M = \Kn_{*-d}$ then $M^{-1}\coloneq \Kn_{*+d}$ is an inverse.
            Assume now that $M$ is invertible. As $\Kn_*$ is a graded field, $M$ and its inverse $M^{-1}$ are free and can be written as
            \begin{equation*}
                M = \bigoplus_{i\in I} \Kn_{*-d_i}, \qquad M^{-1} = \bigoplus_{j\in J} \Kn_{*-e_j}.
            \end{equation*}
            Therefore
            \begin{equation*}
                M\otimes_{\Kn_*} M^{-1} \simeq \bigoplus_{i\in I,j\in J} \Kn_{*-(d_i+e_j)}.
            \end{equation*}
            It is isomorphic to $\Kn_*$ if and only if both $I,J$ consists of a single element $I=\{i\}, J=\{j\}$ and $d_i + e_j \equiv 0$ modulo $2(p^\chrHeight-1)$.
        \end{proof}
    
        \begin{corollary}\label{cor:Z/vn->Pic}
            The composition
            \begin{equation*}
                \ZZ/ (a_p (\extOrder)) \into \Zp \times \ZZ/ (a_p (\extOrder)) \xto{\Sigma^{(-)}} \Pic(\SpTn)
            \end{equation*}
            is injective.
        \end{corollary}
        \begin{proof}
            It is enough to check that this composition is injective after composing with the map 
            \begin{equation*}
                \Kn_*(-) \colon \Pic(\SpTn) \to \Pic(\Mod_{\Kn_*}(\GrAb)).
            \end{equation*} 
            By construction, this map sends $d\in\ZZ/ (a_p (\extOrder))$ to $\Kn_{*+d}$ if $p$ is odd or to $\Kn_{*+2d}$ if $p=2$ \footnote{Under the isomorphism $\lim_{\vnSelfNum} \ZZ/\vnSelfDeg \cong \Zp\times \ZZ/ (a_p(\extOrder))$, the generator of $\ZZ/ (a_p(\extOrder))$ is $1$ if $p$ is odd and is $2$ if $p=2$.} and so it is injective (and an isomorphism for $p\neq 2$).
        \end{proof}
    
    
        Let $\Gamma_{\chrHeight}$ be a formal group of height ${\chrHeight}$ and $\En=\En(\Gamma_{\chrHeight}, {\Fpbar})$ be a Morava $E$-theory over ${\Fpbar}$. $\En$ admits a natural action of the (extended) Morava stabilizer group $\extMorStb = \Aut(\Gamma_{\chrHeight}, {\Fpbar})$. 
        Let $\Witt({\Fpbar})$ be the $p$-typical Witt vectors ring. Then by Lubin--Tate theory, the local ring $\pi_0\En = \Witt({\Fpbar})\llbracket u_1,\dots,u_{\chrHeight-1}\rrbracket$ corepresents deformations of $\Gamma_{\chrHeight}$. The invariant differentials $\omega$ of the universal deformation of $\Gamma_n$ forms an invertible $\pi_0 \En$-module and $\pi_* \En = \bigoplus_{t\in 2\ZZ} \omega^{\otimes t/2}$, where $\omega$ is of degree -2, as commutative rings with $\extMorStb$-action.
    
        The Morava stabilizer group splits as $\extMorStb= \MorStb \rtimes \Gal$ where $\Gal=\Gal({\Fpbar}/\Fp) \cong \Zhat$ and $\MorStb = \Aut_{{\Fpbar}}(\Gamma_{\chrHeight})$ is isomorphic to the group of units $\unitsOrd$ of the order
        \begin{equation*}
            \MorOrd = \Wn \langle S \mid S^{\chrHeight} = p, Sw = w^{\varphi} S \ \forall w\in \Wn \rangle
        \end{equation*}
        where $\Wn = \Witt(\Fpn)$ is the $p$-typical Witt vectors ring of $\Fpn$ and $\varphi$ is a lift of the Frobenius to $\Wn$. The topological generator $1$ of $\Zhat$ acts on $\MorStb=\unitsOrd$ by conjugation with $S$.
        In particular we have an embedding
        \begin{equation*}
            \Zp \xinto{1+p\cdot-} \Zp\units \subseteq  \Wn\units \into \unitsOrd \into \extMorStb.
        \end{equation*}
        Consider the symmetric monoidal functor
        \begin{equation*}
            \En\otimes -\colon \SpTn \xto{{(\B\Zp)}^*} \SpTn^{\B\Zp}\xto{\En\otimes-} \Mod_{\En}(\SpTn^{\B\Zp})
        \end{equation*}
       tensoring with $\En$ while remembering the $\Zp$-action. As it is symmetric monoidal, it sends Picard elements to Picard elements.
    
       \begin{lemma}\label{lem:Zp->Pic}
           The group homomorphism
           \begin{equation*}
               \Zp \into \Zp\times \ZZ / (a_p (\extOrder)) \xto{\Sigma^{(-)}} \Pic(\SpTn)
           \end{equation*}
           is injective.
       \end{lemma}
       \begin{proof}
           By continuity it is enough to check this for $\ZZ\subseteq \Zp$. It is enough to see this after composing with 
           \begin{equation*}
               \En\otimes - \colon \Pic(\SpTn) \to \Pic(\Mod_{\En}(\SpTn^{\B\Zp})).
           \end{equation*}
           The composition sends $d\in\ZZ\subseteq \Zp$ to $\Sigma^d \En\in\Pic(\Mod_{\En}(\SpTn^{\B\Zp}))$. As $\Sigma E_n \neq \En$ it is enough to check the injectivity for $d$ even, i.e.\ that for $d$ even $\Sigma^d\En$ is not isomorphic to $\En$ in $\Mod_{\En}(\SpTn^{\B\Zp})$.
           Note that on the level of homotopy groups with $\Zp$-action
           \begin{equation*}
               \pi_* \Sigma^d \En \simeq \pi_* \En \otimes \omega^{d/2} \quad \in \quad \Mod_{\pi_*\En}(\catname{GrAb}^{\B\Zp})
           \end{equation*}
           (see e.g.~\cite[Lemma~1.3.1]{Heard-2015-Picard}) therefore it is not isomorphic to $\pi_*\En$ and in particular $\Sigma^d\En$ is not isomorphic to $\En$.
       \end{proof}
    
       \begin{proof}[Proof of \cref{thm:ZpxZ/vn->Pic}]
           The result follows from \cref{cor:Z/vn->Pic}, \cref{lem:Zp->Pic} and the fact that $\ZZ/ (a_p (\extOrder))$ is torsion while $\Zp$ is not.
       \end{proof}

        \begin{remark}\label{rmrk:also-Kn-pic}
            All computations done here hold in the $\Kn$-local category. Thus the constructed Picard subgroup is preserved under $\Kn$-localization.
        \end{remark}

        \subsection{Even Picard group}

         Invertible elements in $\cC$ are in particular dualizable. Thus every Picard element $X$ has a dimension $\dim (X)\in \pi_0\ounit$.
        \begin{theorem}[{\cite[Corollary~3.21]{CSY-cyclotomic}}]\label{thm:dim^2=1}
            Let $\cC$ be a symmetric monoidal category. For every $X \in \Pic(\cC)$, we have ${\dim(X)}^2 = 1$. In particular, if $\pi_0\ounit$ is a connected ring (that is, it does not contain any non-trivial idempotents) and 2 is invertible in $\pi_0\ounit$, then $\dim(X) = \pm 1$.
        \end{theorem}
        
        \begin{definition}[{\cite[Definition~3.22]{CSY-cyclotomic}}]
            The even Picard group of a symmetric monoidal category $\cC$, is the subgroup $\evPic(\cC) \subseteq \Pic(\cC)$ given by the kernel of the map $\Pic(\cC) \xto{\dim} {\pi_0(\ounit)}\units$.
        \end{definition}
        
        \begin{corollary}\label{cor:subgroup-of-evPic}
            There is an embedding $\Zp \times \ZZ / (\extOrder) \into \evPic(\SpTn)$.
        \end{corollary}
        
        \begin{proof}
            By \cref{thm:dim^2=1}, the dimension of every $X\in \Pic(\SpTn)$ squares to 1, thus the subgroup of squares ${(\Pic(\SpTn))}^2$ lies in the even Picard group.
            By \cref{thm:ZpxZ/vn->Pic} there is an embedding $\Sigma^{(-)} \colon  \Zp \times \ZZ / (a_p (\extOrder)) \into \Pic(\SpTn)$, where
            \begin{equation*}
                a_p \coloneqq  \begin{cases}
                    2, & \text{$p$ odd} \\
                    1, & p = 2
                \end{cases}.
            \end{equation*}
            Therefore we have an embedding of the group of squares:
            \begin{equation*}
                \Sigma^{(-)}\colon  2\Zp \times 2(\ZZ / (a_p (\extOrder)))\into {(\Pic(\SpTn))}^2 \subseteq \evPic(\SpTn). \ \footnote{As the groups $\Zp$ and $\ZZ / (a_p (\extOrder))$ are additive, and to avoid confusion, we denote their group of squares by $2\Zp$ and $2(\ZZ / (a_p (\extOrder))) = \{2x \mid x\in\ZZ/ (a_p (\extOrder))\}$ respectively.} 
            \end{equation*}
            When $p$ is odd $2\in\Zp$ is invertible, thus $2\Zp = \Zp$ and $2(\ZZ / (a_p (\extOrder))) = 2(\ZZ / (2(\extOrder))) \cong \ZZ / (\extOrder)$. 
            In the case $p=2$, $2\ZZ_2 \cong \ZZ_2$ with generator $2$, and $2\in \ZZ/ (a_2(2^{\chrHeight}-1)) = \ZZ/ (2^{\chrHeight}-1)$ is invertible thus $2(\ZZ/ (2^{\chrHeight}-1)) = \ZZ/ (2^{\chrHeight}-1)$.

            In both cases the group of squares is isomorphic to $\Zp \times \ZZ/ (\extOrder)$.
        \end{proof}

\section{Galois extensions}\label{sec:galois}
    In this section we use Kummer theory as introduced in~\cite{CSY-cyclotomic} to find a telescopic lifting of a specific non-Abelian Galois extensions of $\SSKn$. 
    We start with a brief review of Galois and Kummer theories, then move to the category of $\LTn \sphereWitt(\Fpn)$-modules, which admits a $(\extOrder)$-st root of unity, and study its Picard group.

    \subsection{Kummer theory}
        \begin{definition}[\cite{Rognes-2008-Galois}]
            Let $\cC\in\CAlg(\PrL)$, let $G$ be a finite group and let $R\in\CAlg(\cC^{\B G})$. We say that $R$ is a $G$-Galois extension (of $\ounit$) if it satisfies the following two conditions:
            \begin{enumerate}
                \item The canonical map $\ounit \to R^{h G}$ is an isomorphism in $\cC$.
                \item The canonical map $R\otimes R \to \prod_G R$ is an isomorphism in $\cC^{\B G}$.
            \end{enumerate}
            A Galois extension is called \emph{faithful} if in addition the functor $R\otimes -$ is conservative.
            Denote by $\hGal{G}(\cC)\subseteq {\CAlg(\cC^{\B G})}\core$ the subspace of $G$-Galois extensions.
        \end{definition}
        
        \begin{definition}[Roots of unity,~\cite{CSY-cyclotomic}]
            Let $\cC$ be an additive presentably symmetric monoidal category. Let $m \ge 1$. An $m$-th root of unity of $\cC$ is a map $\omega\colon  \Cn[m]\to \ounit\units$. We say that $\omega$ is \emph{primitive} if $\ounit$ is $m$-divisible (i.e.\ $m$ is divisible in $\pi_0\ounit$) and for every $d$ which strictly divides $m$, the only commutative algebra $S\in\CAlg(\cC)$ for which there exists a dotted arrow rendering the diagram of connective spectra
            \begin{equation*}
                \begin{tikzcd}
                    {\Cn[m]} && {\ounit\units} \\
                    {\Cn[d]} && {S\units}
                    \arrow["\omega", from=1-1, to=1-3]
                    \arrow[two heads, from=1-1, to=2-1]
                    \arrow[from=1-3, to=2-3]
                    \arrow[dashed, from=2-1, to=2-3]
                \end{tikzcd}
            \end{equation*} 
            commutative, is $S = 0$.
        \end{definition}

        \begin{theorem}[{Kummer theory,~\cite[Proposition~3.23]{CSY-cyclotomic}}]\label{thm:Kummer}
           Let $\cC$ be a presentable, additive, symmetric monoidal category with a primitive $m$-th root of unity. There is a split short exact sequence of Abelian groups
            \begin{equation*}
                0 
                \to (\pi_0\ounit\units) / {(\pi_0 \ounit\units)}^m 
                \to \pi_0 \hGal{\ZZ / m}(\cC)
                \to \evPic(\cC)[m]
                \to 0.
            \end{equation*}
        \end{theorem}   
    
    \subsection{Galois theory of $\LKn \sphereWitt(\Fpn)$}
        Let $\Witt$ be the Witt vectors functor and $\sphereWitt$ be the spherical Witt vectors functor as defined in~\cite[Example~5.2.7]{Lurie-2018-Elliptic2}, and later in~\cite[Section~2]{Burklund-Schlank-Yuan-2022-Nullstellensatz},~\cite{Antieau-2023-spherical-Witt}, \cite{Nikolaus-Yakerson-2024-spherical-Witt}.
        \begin{notation}
            Denote the Witt vectors of $\Fpn$ and the $\Kn$-localization of the spherical Witt vectors of $\Fpn$ by
            \begin{align*}
                & \Wn \coloneqq  \Witt(\Fpn) \\
                & \SWn \coloneqq  \LKn(\sphereWitt(\Fpn))
            \end{align*}
            respectively. We will also denote $\ModSWn \coloneq \Mod_{\SWn}(\SpKn)$.
        \end{notation}
        Let $\En=\En(\Gamma_n, {\Fpbar})$ be a Morava $E$-theory.
        
        \begin{proposition}[{\cite[Theorem~5]{Devinatz-Hopkins-2004-Morava-stabilizer},\cite{Baker-Richter-2008-Galois-En}~\cite[Theorem~10.9]{Mathew-2016-Galois}}]
            The map $\SSKn \to \En$ exhibits $\En$ as the Galois closure of $\SSKn$ and its Galois group is the Morava stabilizer group $\extMorStb$.
        \end{proposition}
        \begin{proposition}[{\cite[Proposition~5.13]{CSY-cyclotomic}}]
            The Galois group of $\SWn$ over $\SSKn$ is $\ZZ / \chrHeight$, given as the quotient group of $\Zhat \cong \Gal \le \extMorStb$.
        \end{proposition}  
        
        \begin{corollary}
            The map $\SWn\to \En$ idenitfies $\En$ as the algebraic closure of $\SWn$ with Galois group $\ourStb \coloneqq  \MorStb \times \chrHeight \Zhat \le \extMorStb$\footnote{It is a product as the action of $\chrHeight\in \Zhat$ on $\MorStb$ is given by conjugation with $S^\chrHeight = p$, which is trivial}.
        \end{corollary}
        
        Using~\cite[Section~9]{Mathew-2016-Galois} we deduce the following:
        \begin{corollary}\label{cor:Galois-characters-correspondence}
            Let $G$ be a finite group. Then there is a bijection 
            \begin{equation*}
                \normalfont
                \{ \text{continuous homomorphisms $\ourStb\to G$} \} / \text{conj.} 
                \longleftrightarrow
                \{ \text{$G$-Galois extensions of $\SWn$} \} / \text{iso.}
            \end{equation*}
            taking a character $\rho\colon  \ourStb \to G$ to ${C_{\rho}(G,\En)}^{h \ourStb}$, where 
            \begin{equation*}
                C_{\rho}(G,E_n) = \prod_G E_n\in\underCat{\CAlg(\SpKn)}{\SWn}
            \end{equation*}
             is equipped with the $\rho$-twisted $\ourStb$-action: $\ourStb$ acts on each $\En$ by Galois conjugation and permutes the coordinates through $\rho$ and the left regular action of $G$ on itself.
        \end{corollary}
        
        \begin{definition}\label{dfn:outExt}
            The ring $\MorOrd = \Wn\left\langle S \middle\vert S^{\chrHeight} = p, Sw = w^{\varphi} S \right\rangle$ admits a surjection of rings
            \begin{equation*}
                \pi \colon \MorOrd
                \xonto{S\mapsto 0} \Fpn 
            \end{equation*}
            which induces a surjective group homomorphism
            \begin{equation*}
                \ourStb = \MorStb \times \chrHeight\Zhat \onto \MorStb = \unitsOrd \xonto{\pi} \uFpn.
            \end{equation*}
            Let $\ourExt$ be the $\uFpn$-Galois extension of $\SWn$ corresponding to this homomorphism.
        \end{definition}
        
        \begin{definition}
            Let $\redMorStb$ be the kernel of the map $\pi \colon \MorStb\onto \uFpn$. 
        \end{definition}
        The multiplicative lifts function $\uFpn \to \Wn\units \subseteq \MorStb$ splits $\MorStb$ as $\MorStb = \redMorStb \rtimes \uFpn$.

        We have now introduced all ingredients used in the classification of the even Picard group of $\ModSWn$, a key step in the lifting of the Galois extension.
        \begin{proposition}\label{prp:evPic-cyclic}
            The group $\evPic(\ModSWn)[\extOrder]$ is a cyclic group of order $\extOrder$ and is generated by the image of $\ourExt$ under the map in \cref{thm:Kummer}.
        \end{proposition}
        
        To prove this proposition we will need the following two lemmas, which are variants of~\cite[Proposition~2.2.6]{CSY-ambiheight} and~\cite[Proposition~5.17]{CSY-cyclotomic}. The proof of \cref{lem:pi0-of-SWn} is almost identical, and the proof of \cref{prp:evPic-cyclic} using these lemmas is the same as the proof of~\cite[Proposition~5.23]{CSY-cyclotomic}. We write the proofs for completeness.
        
        \begin{lemma}\label{lem:pi0-of-SWn}
            The image of the unit map $u\colon  \pi_0 \SWn \to \pi_0 \En$ is $\Wn\subseteq \pi_0 \En$ and the kernel is precisely the nil-radical.
        \end{lemma}
        \begin{proof}
            As $\SWn = \En^{h \ourStb}$, the map $u$ factors through the fixed points ${(\pi_0 \En)}^{\ourStb} \subseteq \pi_0 \En$. By~\cite[Lemma~1.33]{Bobkova-Goerss-2018-K2-local-at-2},

            \begin{equation*}
                {(\pi_0 \En)}^{\ourStb} = {({(\pi_0 \En)}^{\MorStb})}^{\chrHeight \Zhat} = {\Witt({\Fpbar})}^{\chrHeight \Zhat} = \Wn.
            \end{equation*}
            By~\cite[Corollary~2.3.10]{Beaudry-Goerss-Henn-2022-splitting-for-K2}, the $E_{\infty}$-page of the descent spectral sequence
            \begin{equation*}
                \H^s_c(\ourStb, \pi_t\En) \Rightarrow \pi_{t-s}(\SWn)
            \end{equation*}
            has a horizontal vanishing line. 
            Since the spectral sequence is multiplicative, this implies that all elements in $\pi_0 \SWn$ with positive filtration degree are nilpotent. Finally, since $\SWn$ admits a ring map from $\sphereWitt(\Fpn)$, the map $u\colon  \pi_0\SWn \to \Wn$ is surjective.
        \end{proof}
        
        \begin{lemma}\label{lem:ab-of-mor-stb}
            The Abelianization of $\MorStb$ is given by
            \begin{equation*}
                \MorStb\ab \cong \Zp\units \oplus_{\Fp\units} \uFpn \cong \begin{cases}
                    \Zp \times \uFpn, & p > 2 \\
                    \Zq{2} \times \{\pm 1\} \times \Fq{2^{\chrHeight}}\units, & p = 2
                \end{cases}.
            \end{equation*}
        \end{lemma}         
        
        \begin{proof}
            For $n=1$, $\MorStb = \Zp\units$ is already Abelian and the lemma is trivial.
            Following~\cite[\textsection 4]{Henn-2017-Morava-stabilizer}, let $v\colon  \MorOrd \to \frac{1}{\chrHeight}\ZZ_{\ge 0}$ be the valuation defined by $v(S) = \frac{1}{\chrHeight}$ and define a $\frac{1}{\chrHeight}\ZZ_{\ge 0}$-filtration on $\MorStb = \MorOrd\units$ by
            \begin{equation*}
                \filtS{i} = \{x\in \MorStb \mid v(x-1) \ge i \}.
            \end{equation*}
            With this filtration $\MorStb = \filtS{0}$ and $\redMorStb = \filtS{\frac{1}{\chrHeight}}$.
            The group commutator and $p$-th power map endows $\bigoplus_{i > 0} \grS{i}$ with the structure of a graded mixed Lie algebra in the sense of Lazard~\cite{Lazard-1965-p-adic} (see~\cite[\textsection~4.2]{Henn-2017-Morava-stabilizer}). As Henn, denote the corresponding operations by $[-,-] \colon  \grS{i}\otimes \grS{j} \to \grS{i+j}$ and $P\colon  \grS{i} \to \grS{\phi(i)}$ where $\phi(i) = \min\{i+1,pi\}$.
            
            The mod $S$ reduction maps
            \begin{align}\label{eq:mod-S-map}
                & \grS{0} \to \uFpn && \grS{i} \to \Fpn   \\ 
                & [x] \mapsto x \mod S && [1+xS^k] \mapsto x \mod S \nonumber
            \end{align}
            are isomorphisms, and it is a straightforwatd calculation to show that for $1+x\in \grS{i}, 1+y \in \grS{j}$, 
            \begin{equation}\label{eq:commutator}
                [1+x, 1+y] \equiv 1 + xy - yx \in \grS{i+j}.
            \end{equation}
            For every $k\in \frac{1}{\chrHeight}\ZZ_{\ge 0}$, we get a natural filtration on $(\filtS{k})\ab = \filtS{k} / \overline{[\filtS{k}, \filtS{k}]}$ given by 
            \begin{equation*}
                \abFilt{i}{k} \coloneqq  \filtS{i} / \overline{\langle [\filtS{a}, \filtS{b}] \mid k\le a,b, a+b=i\rangle}
            \end{equation*}
            which simply induces the equalities:
            \begin{equation*}
                \gr{i}(\filtS{k})\ab = \gr{i}(\filtS{k+\frac{1}{\chrHeight}})\ab / [\grS{k}, \grS{i-k}].
            \end{equation*}
            In particular, assigning $k = 0$:
            \begin{equation*}
                \grSab{i} = \gr{i}\redMorStb\ab / [\grS{0}, \grS{i}] \text{ for all $i > 0$.}
            \end{equation*}
            In the proofs of~\cite[Proposition~5.2 and Propsition~5.3]{Henn-2017-Morava-stabilizer}, it is shown that for any $p$,
            \begin{equation*}
                \gr{i}\redMorStb\ab = \begin{cases}
                    \Fpn, & i = \frac{1}{\chrHeight} \\
                    \Fp, & i\in \ZZ_{\ge 1} \\
                    0, & \text{otherwise}
                \end{cases}
            \end{equation*}
            and that the $p$-th power map $P$ is
            \begin{enumerate}[label=\alph*)]
                \item an isomorphism on the integral grades for $p$ odd;
                \item and an isomorphism on intergral grades greater than 1 and trivial for the first grade for $p=2$.
            \end{enumerate}
            Let $j\in \ZZ_{\ge 1}$. Using \cref{eq:mod-S-map}, we may choose representatives for $\grS{j}$ of the form $1 + \omega p^j$, $\omega\in\Fpn$, and $\uFpn$ represent $\grS{0}$. Let $1+\omega p^j\in \grS{j}$, $\omega'\in \uFpn \cong \grS{0}$. Then by~\cref{eq:commutator}:
            \begin{equation*}
                [\omega', 1+\omega p^j] = [1+(\omega'-1), 1+\omega p^j] \equiv 1 + (\omega'-1)\omega p^j - \omega p^j(\omega'-1)= 1.
            \end{equation*}
            Therefore $[\grS{0}, \grS{j}] = 0$.
            Again, using \cref{eq:mod-S-map}, we can choose representatives for $\grS{\frac{1}{\chrHeight}}$ of the form $1+\omega S$ for $\omega\in \Fpn$. Let $\omega'\in\uFpn\cong \grS{0}$ and $1+\omega S\in\grS{\frac{1}{\chrHeight}}$, then
            \begin{equation*}
                [\omega', 1+\omega S] = [1 + (\omega'-1), 1+\omega S] \equiv 1 + (\omega' - 1)\omega S - \omega S(\omega'-1) = 1 + \omega(\omega' - \omega'^{\varphi}) S.
            \end{equation*}
            Choosing $\omega' \notin \Fp\units$, varying over all $\omega$ we see that $[\grS{0}, \grS{\frac{1}{\chrHeight}}] = \grS{\frac{1}{\chrHeight}}$. We get
            \begin{equation*}
                \grSab{i} = \begin{cases}
                    \uFpn, & i = 0 \\
                    \Fp, & i \in \ZZ_{\ge 1} \\
                    0, & \text{otherwise}
                \end{cases}
            \end{equation*}
            and the $p$-th power map is 
            \begin{enumerate}[label=\alph*)]
                \item an isomorphism on positive grades for $p$ odd;
                \item and an isomorphism on grades greater than 1 and trivial for the first grade for $p=2$.
            \end{enumerate}
            The claim then follows.
        \end{proof} 

        \begin{corollary}\label{cor:ab-of-stb}
            The Abelianization of $\ourStb$ is given as
            \begin{equation*}
                (\ourStb)\ab \cong (\Zp\units \oplus_{\Fp\units} \uFpn) \times \Zhat
            \end{equation*}
        \end{corollary}
        \begin{proof}
            Indeed, we have
            \begin{equation*}
                (\ourStb)\ab = (\MorStb \times \chrHeight\Zhat)\ab = \MorStb\ab \times \chrHeight \Zhat \cong (\Zp\units \oplus_{\Fp\units} \uFpn) \times \Zhat.
            \end{equation*}
        \end{proof}

        \begin{proof}[Proof of \cref{prp:evPic-cyclic}]
            Using \cref{thm:Kummer} and its naturality with respect to the functor
            \begin{equation*}
                \LKn\colon  \Mod_{\sphereWitt(\Fpn)}(\pCompSp) \to \ModSWn,
            \end{equation*}
            we obtain the following morphism of short exact sequences of abelain groups
            \begin{equation*}
                \begin{tikzcd}
                    0 & 0 \\
                    \\
                    {(\pi_0{\sphereWitt(\Fpn)}\units) / {(\pi_0{\sphereWitt(\Fpn)}\units)}^{p^{\chrHeight} -1}} & {(\pi_0\SWn\units) / {(\pi_0 \SWn\units)}^{\extOrder}} \\
                    \\
                    {\pi_0\hGal{\uFpn}(\Mod_{\sphereWitt(\Fpn)}(\pCompSp))} & {\pi_0\hGal{\uFpn}(\ModSWn)} \\
                    \\
                    {\evPic(\Mod_{\sphereWitt(\Fpn)}(\pCompSp))[\extOrder]} & {\evPic(\ModSWn)[\extOrder]} \\
                    \\
                    0 & 0
                    \arrow["f", from=3-1, to=3-2]
                    \arrow[from=3-1, to=5-1]
                    \arrow[from=5-1, to=7-1]
                    \arrow[from=1-1, to=3-1]
                    \arrow[from=1-2, to=3-2]
                    \arrow[from=7-1, to=9-1]
                    \arrow["g", from=5-1, to=5-2]
                    \arrow[from=3-2, to=5-2]
                    \arrow[from=5-2, to=7-2]
                    \arrow[from=7-2, to=9-2]
                    \arrow[from=7-1, to=7-2]
                \end{tikzcd}
            \end{equation*}
            In the top left corner we have
            \begin{equation*}
                {(\pi_0{\sphereWitt(\Fpn)}\units) / {(\pi_0{\sphereWitt(\Fpn)}\units)}^{p^{\chrHeight} -1}}
                = (\Wn\units) / {(\Wn\units)}^{p^{\chrHeight} - 1}
                \cong \uFpn.
            \end{equation*}
            The top horizontal map $f$ is an isomorphism: Indeed, by \cref{lem:pi0-of-SWn}, the map
            \begin{equation*}
                \Wn = \pi_0 \sphereWitt(\Fpn) \to \pi_0 \SWn
            \end{equation*}
            admits a retract $r\colon  \pi_0 \SWn \to \Wn$ whose kernel consists of nilpotent elements. In
            particular every element in the kernel of $r\units\colon  \pi_0\SWn\units \to \Wn\units$ is of the form $x = (1+\varepsilon)$ for some nilpotent $\varepsilon\in \pi_0\SWn$. As $\extOrder$ is invertible in $\pi_0\SWn$, every such element $x$ has a $(\extOrder)$-st root, and hence $r\units$ induces an isomorphism after modding out the $(\extOrder)$-st powers. Since this induced isomorphism is a left-inverse of $f$ it follows that $f$ is an isomorphism as well.
            
            Next, by \cref{cor:Galois-characters-correspondence}, \cref{cor:ab-of-stb} and using that the absolute Abelian Galois group of $\sphereWitt(\Fpn)$ is $\Zhat$ (e.g.~\cite[Theorem~6.17]{Mathew-2016-Galois}), $g$ can be identified with the inclusion 
            \begin{equation*}
                \hom(\Zhat, \uFpn)
                \into \hom(\Zhat, \uFpn) \oplus \hom(\Zp\units \oplus_{\Fp\units} \uFpn, \uFpn).
            \end{equation*}
            
            Since $\hom(\Zhat, \uFpn) \cong \uFpn$, the entire diagram can be identified with
            \begin{equation*}
                \begin{tikzcd}
                    0 & 0 \\
                    \\
                    {\uFpn} & {\uFpn} \\
                    \\
                    {\uFpn} & {\uFpn\oplus \hom(\Zp\units\oplus_{\Fp\units} \uFpn, \uFpn)} \\
                    \\
                    0 & {\evPic(\ModSWn)[\extOrder]} \\
                    \\
                    0 & 0
                    \arrow[Rightarrow, no head, from=3-1, to=3-2]
                    \arrow[Rightarrow, no head, from=3-1, to=5-1]
                    \arrow[from=5-1, to=7-1]
                    \arrow[from=1-1, to=3-1]
                    \arrow[from=1-2, to=3-2]
                    \arrow[from=7-1, to=9-1]
                    \arrow[hook, from=5-1, to=5-2]
                    \arrow[from=3-2, to=5-2]
                    \arrow[from=5-2, to=7-2]
                    \arrow[from=7-2, to=9-2]
                    \arrow[from=7-1, to=7-2]
                \end{tikzcd}
            \end{equation*}
            thus the bottom-right veritcal map restricts to an isomorphism
            \begin{equation*}
                \hom(\Zp\units \oplus_{\Fp\units} \uFpn, \uFpn) \cong \evPic(\ModSWn)[\extOrder].
            \end{equation*}
            Moreover, $\hom(\Zp\units \oplus_{\Fp\units} \uFpn, \uFpn) \cong \uFpn$ (this can be checked separately for odd and even primes, using the formula in \cref{lem:ab-of-mor-stb}), so $\evPic(\ModSWn)[\extOrder] \cong \uFpn$.
            
            Now, $\ourExt$ was chosen in \cref{dfn:outExt} as corresponding to the character
            \begin{equation*}
                \ourStb 
                \onto (\ourStb)\ab 
                \cong \Zhat \times (\Zp\units \oplus_{\Fp\units} \uFpn) \onto \Zp\units \oplus_{\Fp\units} \uFpn \to \uFpn
            \end{equation*}
            which generates a cyclic subgroup of $\hom(\Zp\units \oplus_{\Fp\units} \uFpn, \uFpn)$ of order $\extOrder$, and so the same holds also in $\evPic(\ModSWn)[\extOrder]$.
        \end{proof} 
        
        \begin{proposition}\label{prp:Kn-pic}
            $\evPic(\SpKn)[\extOrder]$ is a cyclic group of order $\extOrder$.
        \end{proposition}
        
        \begin{proof}
            Using \cref{cor:subgroup-of-evPic} and \cref{rmrk:also-Kn-pic}, there is an injective homomorphism
            \begin{equation*}
                 \ZZ / (\extOrder) \into \evPic(\SpKn)[\extOrder].
            \end{equation*}
            Consider the composition
            \begin{equation*}
                \ZZ / (\extOrder) \into \evPic(\SpKn)[\extOrder] \xto{\SWn\otimes-} \evPic(\ModSWn)[\extOrder].
            \end{equation*}
            We will show it is injective, thus finishing by \cref{prp:evPic-cyclic}. It is enough to show that before moving to even Picard elements, i.e.\ to show that the composition
            \begin{equation*}
                \ZZ / (a_p (\extOrder)) \into \Pic(\SpKn) \xto{\SWn\otimes -} \Pic(\ModSWn)
            \end{equation*}
            is injective.
            Assume $d\in \ZZ / (a_p (\extOrder))$ is in the kernel and let $X$ be the corresponding Picard spectrum. Then $\SWn\otimes X \simeq \SWn$ in $\ModSWn$. Composing with the forgetful functor $\ModSWn \to \SpKn$, remembering that as a spectrum $\sphereWitt(\Fpn) \simeq \SS^{\oplus \chrHeight}$ we get that $X^{\oplus \chrHeight} \simeq \SSKn^{\oplus \chrHeight}$ in $\SpKn$.
            Tensoring now with $\Kn$ we get that $\Kn\otimes X^{\oplus \chrHeight} \simeq \Kn^{\oplus \chrHeight}$, therefore $\Kn\otimes X\simeq \Kn$. But we showed in the proof of \cref{cor:Z/vn->Pic} that the composition
            \begin{equation*}
                \ZZ / (a_p (\extOrder)) \into \Pic(\SpTn) \xto{\Kn\otimes -} \Pic(\BiMod_{\Kn})
            \end{equation*}
            is injective, and it is equivalent to the composition
            \begin{equation*}
                \ZZ / (a_p (\extOrder)) \into \Pic(\SpTn) \xto{\LKn} \Pic(\SpKn) \xto{\Kn\otimes-} \Pic(\BiMod_{\Kn}).
            \end{equation*}
            Thus, as $X\otimes \Kn\simeq \Kn$, $d$ is in the kernel of the above injective map and so is trivial as required.
        \end{proof}
        
    
    \subsection{Telescopic lifting}
        In \cref{cor:subgroup-of-evPic} we constructed a $(\extOrder)$-cyclic subgroup of $\evPic(\SpTn)$. Let \mbox{$\SWnf \coloneqq  \LTn(\sphereWitt(\Fpn))$} and denote also by $\ModSWnf \coloneqq \Mod_{\SWnf}(\SpTn)$. By~\cite[Theorem~5.9 and Theorem~5.27]{CSY-cyclotomic} $\SWnf$ is a faithful $\ZZ / \chrHeight$-Galois extension of $\SSTn$, and $\SWn = \LKn(\SWnf)$ is classified as a Galois extension of $\SSKn$ by the character 
        \begin{equation*}
            \extMorStb \onto \Zhat \onto \ZZ / \chrHeight.
        \end{equation*}
        Using \cref{rmrk:also-Kn-pic}, \cref{prp:evPic-cyclic} \cref{prp:Kn-pic} we get a commutative diagram of groups
        \begin{equation*}
            \begin{tikzcd}
                & {\evPic(\SpTn)[\extOrder]} && {\evPic(\ModSWnf)[\extOrder]} \\
                {\uFpn} \\
                & {\evPic(\SpKn)[\extOrder]} && {\evPic(\ModSWn)[\extOrder]}
                \arrow[hook, from=2-1, to=1-2]
                \arrow["{\SWnf \otimes -}", from=1-2, to=1-4]
                \arrow["\sim"', "j", from=2-1, to=3-2]
                \arrow["\LKn", from=1-2, to=3-2]
                \arrow["\sim"', "{\SWn \otimes -}", from=3-2, to=3-4]
                \arrow["\LKn", from=1-4, to=3-4]
            \end{tikzcd}.
        \end{equation*}
        Let $P_{\chrHeight}\in\evPic(\SpKn)[\extOrder]$ be such that $\SWn\otimes P_{\chrHeight}\in\evPic(\ModSWn)[\extOrder]$ is the image of $\ourExt$ under the Kummer map \cref{thm:Kummer}. 
        Let $P_{\chrHeight}^f$ be the image of $j^{-1}(P_{\chrHeight})$ in $\evPic(\SpTn)[\extOrder]$, so $\LKn(P_{\chrHeight}^f) = P_{\chrHeight}$. Using \cref{thm:Kummer} again choose a lifting $\ourExtf$ of $\SWnf \otimes P_{\chrHeight}^f\in\evPic(\ModSWnf)[\extOrder]$ in
        \begin{equation*} 
            \hGal{\ZZ / (\extOrder)}(\ModSWnf) = \hGal{\uFpn}(\ModSWnf).
        \end{equation*}
        Then the image of $\LKn \ourExtf$ in $\evPic(\ModSWn)$ is a generator of the embedded cyclic group. Thus by \cref{prp:evPic-cyclic} we can assume $\LKn(\ourExtf) \simeq \ourExt$.
        
        To summarize:
        \begin{theorem}\label{thm:Gal-summary}
            There exists a $(\uFpn\rtimes\ZZ / \chrHeight)$-Galois extension $\ourExtf$ of $\SSTn$, lifting the Galois extension $\ourExt$ of $\SSKn$ corresponding to the character
            \begin{equation*}
                \extMorStb = \MorStb \rtimes \Zhat \xonto{\pi\rtimes\id} \uFpn \rtimes \Zhat \onto \uFpn \rtimes \ZZ / \chrHeight.
            \end{equation*} 
        \end{theorem}

        \begin{corollary}\label{cor:strong-Gal}
            There exists an $((\uFpn \oplus_{\Fp\units} \Zp\units) \rtimes \Zhat)$-pro-Galois extension $\ourExtTf$ of $\SSTn$, lifting the Galois extension $\ourExtT$ of $\SSKn$ corresponding to the character
            \begin{equation*}
                \extMorStb = \MorStb \rtimes \Zhat \xonto{(\det \oplus_{\Fp\units} \pi) \rtimes \id} (\Zp\units \oplus_{\Fp\units} \uFpn) \rtimes \Zhat.
            \end{equation*}
        \end{corollary}
        \begin{proof}
            We saw in \cref{cor:ab-of-stb} that the Abelianization of $\ourStb$ is $(\ourStb)\ab \cong (\Zp\units \oplus_{\Fp\units} \uFpn) \times \chrHeight\Zhat$. The quotient map
            \begin{equation*}
                \ourStb \to (\Zp\units \oplus_{\Fp\units} \uFpn) \times \chrHeight\Zhat \onto \Zp\units \times \chrHeight\Zhat
            \end{equation*}
            corresponds to the the maximal Abelian pro-Galois extension of $\SSKn$, thought of as an extension of $\SWn$:
            \begin{equation*}
                \SWn \to \LKn\sphereWitt(\Fpbar)[\omega_{p^{\infty}}^{(\chrHeight)}],
            \end{equation*}
            see~\cite[Proposition~5.17]{CSY-cyclotomic}. The quotient
            \begin{equation*}
                \ourStb \to (\Zp\units \oplus_{\Fp\units} \uFpn) \times \chrHeight\Zhat \onto \uFpn
            \end{equation*}
            is the character $\pi$ of $\ourExt$. That means the maximal Abelian pro-Galois extension $\ourExtT$ of $\SWn$, corresponding to the quotient map
            \begin{equation*}
                \ourStb \to (\Zp\units \oplus_{\Fp\units} \uFpn) \times \chrHeight\Zhat
            \end{equation*}
            is obtained by adding all roots of unity and higher roots of unity to $\ourExt$, or euqivalently:
            \begin{equation*}
                \ourExtT \coloneqq \ourExt\otimes_{\SWn} \LKn\sphereWitt(\Fpbar)[\omega_{p^{\infty}}^{(\chrHeight)}].
            \end{equation*}
            We can lift it to a pro-Galois extension of $\SWnf$, using that $\LKn\sphereWitt(\Fpbar)[\omega_{p^{\infty}}^{\chrHeight}]$ lifts to an extension of $\SSTn$ (\cite[Theorem~5.31]{CSY-cyclotomic}):
            \begin{equation*}
                \ourExtTf \coloneqq \ourExtf \otimes_{\SWnf} \LTn\sphereWitt(\Fpbar)[\omega_{p^{\infty}}^{(\chrHeight)}].
            \end{equation*}
            $\ourExtf$ is a $((\Zp\units \oplus_{\Fp\units} \uFpn) \rtimes \Zhat)$-pro-Galois extension of $\SSTn$ lifting the claimed character.
        \end{proof}

\bibliographystyle{alpha}
\phantomsection\addcontentsline{toc}{section}{\refname}
\bibliography{references}
\end{document}